\documentclass[11pt]{article}

\usepackage{latexsym}
\usepackage{amssymb}
\usepackage{amsthm}
\usepackage{amscd}
\usepackage{amsmath}
\usepackage{mathrsfs}
\usepackage{graphicx}
\usepackage{hyperref}
\usepackage{shuffle}
\usepackage{tikz-cd}
\usepackage[all]{xy}
\input xy \xyoption{frame}
\usepackage{ytableau}

\usepackage{enumerate,tikz}
\usepackage{color}
\usepackage{mathdots}
\usepackage[vcentermath,enableskew]{youngtab}

\usepackage{mathabx}
\usepackage{mathtools}
\usepackage{array}

\def\PhiBS{\Psi}

\newcommand{\dK}[1]{\mathbin{\widesim{#1}{\mathsf{dK}}}}
\newcommand{\K}[1]{\mathbin{\widesim{#1}{\mathsf{K}}}}

\newcommand{\simRBS}[1]{\mathbin{\widesim{#1}{\rBS}}}
\newcommand{\simCBS}[1]{\mathbin{\widesim{#1}{\cBS}}}

\def\cI{I}
\def\I{\cI}

\def\Ifpf{I^{\mathsf{FPF}}}
\def\onefpf{1_{\mathsf{FPF}}}
\def\mmap{\iota_{\mathsf{asc}}}
\def\nmap{\iota_{\mathsf{des}}}

\definecolor{darkred}{rgb}{0.7,0,0} 
\newcommand{\defn}[1]{{\color{darkred}\emph{#1}}} 

\numberwithin{equation}{section}

\theoremstyle{definition}
\newtheorem* {theorem*}{Theorem}
\newtheorem* {conjecture*}{Conjecture}
\newtheorem{theorem}{Theorem}[section]

\theoremstyle{definition}

\theoremstyle{definition}

\newtheorem* {example*}{Example}

\newtheorem{lemma}[theorem]{Lemma}
\theoremstyle{definition}
\newtheorem{definition}[theorem]{Definition}
\theoremstyle{definition}

\newtheorem{proposition}[theorem]{Proposition}

\newtheorem* {remark*}{Remark}
\newtheorem {remark}[theorem]{Remark}
\theoremstyle{definition}
\newtheorem {example}[theorem]{Example}
\theoremstyle{definition}

\theoremstyle{definition}

\theoremstyle{definition}

\theoremstyle{definition}

\def\({\left(}
\def\){\right)}

\newcommand{\QQ}{\mathbb{Q}}

\def\cX{\mathcal{X}}

\def\NN{\mathbb{N}}

\def\ZZ{\mathbb{Z}}

\def\spanning{\textnormal{-span}}

\newcommand{\cM}{\mathcal{M}}
\newcommand{\cN}{\mathcal{N}}

\def\fk{\mathfrak}

\def\barr{\begin{array}}
\def\earr{\end{array}}
\def\ba{\begin{aligned}}
\def\ea{\end{aligned}}
\def\be{\begin{equation}}
\def\ee{\end{equation}}

\def\quand{\quad\text{and}\quad}

\def\quord{\quad\text{or}\quad}

\def\cH{\mathcal H}

\def\hs{\hspace{0.5mm}}

\def\ds{\displaystyle}

\def\ben{\begin{enumerate}}
\def\een{\end{enumerate}}

\def\hs{\hspace{0.5mm}}

\def\D{\hat D}

\def\Des{\mathrm{Des}}

\def\b{\textbf{b}}

\newcommand{\cA}{\mathcal{A}}

\def\cG{\mathcal{G}}

\def\Inv{\operatorname{Inv}}

\newcommand{\arc}[2]{ \ar @/^#1pc/ @{-} [#2] }
\def\arcstop{\endxy\ }
\def\arcstart{\ \xy<0cm,-0.5cm>\xymatrix@R=.1cm@C=.5cm }
\newcommand{\arcstartc}[1]{\ \xy<0cm,-.15cm>\xymatrix@R=.1cm@C=#1cm}

\def\m{\mathfrak{m}}
\def\n{\mathfrak{n}}

\def\invsim_i{\overset{\mathrm{i}}{\underset{\mathrm{inv}}{\sim}}}

\def\LL{\ZZ[x,x^{-1}]}
\def\H{\mathcal{H}}

\def\D{\mathsf{D}}
\def\Ifpf{\cI^{\mathsf{FPF}}}

\def\cG{\mathcal{G}}

\def\Asc{\mathrm{Asc}}

\def\m{\mathbf{m}}
\def\n{\mathbf{n}}

\def\row{\mathfrak{row}}

\def\b{\mathsf{ivalue}}
\def\f{\mathsf{frow}}

\def\from{\leftarrow}

\def\rBS{\mathsf{rB}}
\def\cBS{\mathsf{cB}}

\def\fromRBS{\xleftarrow{\hs\rBS\hs}}
\def\fromCBS{\xleftarrow{\hs\cBS\hs}}
\def\fromCBSalt{\xLeftarrow{\hs\cBS\hs}}

\def\fromMRSK{\fromRBS}
\def\fromNRSK{\fromCBS}

\newcommand{\widesim}[3][1.5]{
  \overset{#2}{\underset{#3}{\scalebox{#1}[1]{$\sim$}}}
}

\newcommand{\ytab}[1]{
\ytableausetup{boxsize = .4cm,aligntableaux=center}
 \begin{ytableau} #1 \end{ytableau}}
\newcommand{\ytabb}[1]{
\ytableausetup{boxsize = .9cm,aligntableaux=center}
 \begin{ytableau} #1 \end{ytableau}}

\usepackage[colorinlistoftodos]{todonotes}

\def\BPath{\mathrm{B}}
\def\ymark{{\fk e}}

\def\mLambda{\lambda_{\mathsf{row}}}
\def\nLambda{\lambda_{\mathsf{col}}}

\def\mIota{\iota_{\mathsf{row}}}
\def\nIota{\iota_{\mathsf{col}}}

\def\mGamma{\Gamma^{\mathsf{row}}}
\def\nGamma{\Gamma^{\mathsf{col}}}

\def\mOmega{\omega^{\mathsf{row}}}
\def\nOmega{\omega^{\mathsf{col}}}

\def\mAsc{\Asc^{\mathsf{row}}}
\def\nAsc{\Asc^{\mathsf{col}}}

\def\cGm{\cG^{\mathsf{asc}}}
\def\cGn{\cG^{\mathsf{des}}}

\def\PRBS{P_{\rBS}}
\def\PCBS{P_{\cBS}}

\def\PMRSK{\hat P_{\rBS}}
\def\PNRSK{\hat P_{\cBS}}

 \newcommand{\mleftrightarrow}[1]{\mathbin{\underset{\mathsf{row}}{\overset{#1}{\longleftrightarrow}}}}
 \renewcommand{\nleftrightarrow}[1]{\mathbin{\underset{\mathsf{col}}{\overset{#1}{\longleftrightarrow}}}}

\usepackage{fullpage}

\def\PRSK{P_{\mathsf{RS}}}
\def\QRSK{Q_{\mathsf{RS}}}
\def\fromRSK{\xleftarrow{\hs\mathsf{RS}\hs}}

\begin{document}
\title{Insertion algorithms for Gelfand $S_n$-graphs}
\author{
Eric MARBERG \\ Department of Mathematics \\  HKUST \\ {\tt emarberg@ust.hk}
\and
Yifeng ZHANG \\ Center for Combinatorics \\ Nankai University \\ {\tt zhang.yifeng@nankai.edu.cn}
}

\date{}

\maketitle

\abstract{
The two tableaux assigned by the Robinson--Schensted correspondence are equal if and only if the input permutation is an involution, so the RS algorithm restricts to a bijection between involutions in the symmetric group and standard tableaux. Beissinger found a concise way of formulating this restricted map, which involves adding an extra cell at the end of a row after a Schensted insertion process. We show that by changing this algorithm slightly to add cells at the end of columns rather than rows, one obtains a different bijection from involutions to standard tableaux. Both maps have an interesting connection to representation theory. Specifically, our insertion algorithms classify the molecules (and conjecturally the cells) in the pair of $W$-graphs associated to the unique equivalence class of perfect models for a generic symmetric group.
}

\setcounter{tocdepth}{2}

\section{Introduction}
The well-known \defn{Robinson--Schensted (RS) correspondence}
is a bijection $w\mapsto (\PRSK(w),\QRSK(w))$ from permutations to pairs of standard Young tableaux of same shape. 
This correspondence can be described by the row bumping process known as \defn{Schensted insertion} \cite{Schensted}.
In this formulation, the tableau $\PRSK(w) := \emptyset \fromRSK w_1 \fromRSK w_2 \fromRSK \cdots \fromRSK w_n$ is built up from the empty shape by inserting the values of $w$. Here $T \fromRSK a$ is the tableau formed by 
inserting a number $a$ into the first row of $T$, where either the smallest number $b>a$ is bumped and recursively inserted into the next row,
or $a$ is added to the end of the row if no such $b$ exists. For example,
\[
\ytab{
  1 & 3 
  \\4
} \fromRSK 2
   = 
\ytab{
   1 & 2  \\
   3  \\
   4
}
\]
since $2$ bumps $3$ in the first row, which bumps $4$ in the second row, which is added to the end of the third row. See Section~\ref{sch-sect} for
more background on this algorithm.

The two tableaux $\PRSK(w)$ and $\QRSK(w)$ are equal if and only if $w=w^{-1}$.
 Thus, the RS algorithm restricts to a bijection between involutions in the symmetric group and standard tableaux. In \cite{Beissinger}, Beissinger 
 shows how to directly construct this restricted bijection using a modified form of Schensted insertion, which we refer to as \defn{row Beissinger insertion}. This operation inserts an integer pair $(a,b)$ with $a\leq b$ into a tableau $T$ to
 form a larger tableau $T \fromRBS (a,b)$. If $a=b$ then $T\fromRBS(a,b)$ is given by adding $a$ to the end of the first \textbf{row} of $T$. If $a<b$
 and the operation $\fromRSK a$ adds a box to $T$ in \textbf{row} $i$,
 then $T\fromRBS(a,b)$ is formed from $T\fromRSK a$ by adding  $b$ to the end of \textbf{row} $i+1$.
For example, we have 
\[ \ytab{
  2 & 3 
  \\4
}\fromRBS (5,5)
=
 \ytab{
  2 & 4 & 5
  \\3
},
\quad
\ytab{
  1 & 3 
  \\4
} \fromRBS (2,5)
   = 
\ytab{
   1 & 2  \\
   3  \\
   4 \\ 
   5
},
\quand
\ytab{
  1 & 4 
  \\3
} \fromRBS (2,5)
   = 
\ytab{
   1 & 2  \\
   3  & 4\\ 
   5
}.
\]
Beissinger \cite[Thm. 3.1]{Beissinger}
proves that if  $w=w^{-1} \in S_n$ and $(a_1,b_1)$, $(a_2,b_2)$, \dots, $(a_q,b_q)$ 
are the integer pairs $(a,b)$ with $1\leq a \leq b = w(a)\leq n$, ordered such that $b_1<b_2<\dots<b_q$,
then
\be\label{beis-eq}
\PRSK(w)=\QRSK(w) =  \emptyset \fromRBS (a_1,b_1)\fromRBS (a_2,b_2) \fromRBS\cdots \fromRBS (a_q,b_q).
\ee
We will review more properties of row Beissinger insertion in Section~\ref{rbs-sect}.

There is  a  ``column'' version of Beissinger's insertion algorithm that gives another bijection 
from involutions in the symmetric group to standard tableaux. This map does not appear to have been described
previously in the literature and is the starting point of this article.
The main idea is as follows. 
Suppose $(a,b)$ is an integer pair with $a\leq b$ and $T$ is a tableau.
If $a=b$ then we define $T\fromCBS(a,b) $  by adding $a$ to the end of the first \textbf{column} of $T$. If $a<b$
 and  $\fromRSK a$ adds a box to $T$ in \textbf{column} $j$,
 then we define $T\fromCBS(a,b)$  from $T\fromRSK a$ by adding  $b$ to the end of \textbf{column} $j+1$.
This operation, which we call  \defn{column Beissinger insertion},
is given in exactly the same way as row Beissinger insertion,
just replacing the bold instances of the word ``row'' in the previous paragraph by ``column.''
For example, we have 
\[ \ytab{
  2 & 3 
  \\4
}\fromCBS (5,5)
=
 \ytab{
  2 & 3 
  \\4
  \\ 5
},
\quad
\ytab{
  1 & 3 
  \\4
} \fromCBS (2,5)
    = 
\ytab{
   1 & 2   \\
   3  & 5 \\
   4
},
\quand
\ytab{
  1 & 4 
  \\3
} \fromCBS (2,5)
    = 
\ytab{
   1 & 2  & 5 \\
   3  & 4
}.\]
Our first main result (see Theorem~\ref{pcbs-thm}) is to show that if $(a_i,b_i)$ are as in \eqref{beis-eq}, then the map
\be\label{pcbs-eq}
w \mapsto \PCBS(w):=  \emptyset \fromCBS (a_1,b_1)\fromCBS (a_2,b_2) \fromCBS\cdots \fromCBS (a_q,b_q)
\ee is another bijection from involutions $w=w^{-1} \in S_n$ to standard tableaux with $n$ boxes.

\begin{remark}\label{eqnat-rmk}
Considering our terminology, it would be equally natural to define ``column Beissinger insertion'' by making a different substitution in the definition of $T \fromRBS (a,b)$, namely by inserting $a$ using  \defn{Schensted column insertion} rather than the usual row bumping algorithm, and then still adding $b$ to the end of row $i+1$ if this adds a box in row $i$. 
If we write this alternative operation as $T\fromCBSalt(a,b)$,
then we always have
 $T^\top \fromCBSalt(a,b) = (T\fromCBS(a,b))^\top$
 where $\top$ denotes the usual transpose on tableaux.
 Therefore, replacing $\fromCBS$ with $\fromCBSalt$ in \eqref{pcbs-eq}
 leads to essentially the same bijection, 
 as 
$
 \PCBS(w)^\top=  \emptyset \fromCBSalt (a_1,b_1)\fromCBSalt (a_2,b_2) \fromCBSalt\cdots \fromCBSalt (a_q,b_q).
$
\end{remark}

In general, there does not seem to be a simple relationship between $\PCBS(w)$ and $\PRSK(w)$,
and we do not know of any natural way to extend the domain of $\PCBS$ from involutions to all permutations.
We provide a detailed analysis of both algorithms in Sections~\ref{rbs-sect} and \ref{cbs-sect}.
In particular, we exactly characterize when two involutions $y$ and $z$ are such that
 $\PCBS(y)$ and $\PCBS(z)$ differ by a single \defn{dual equivalence operation}; see Theorem~\ref{n-cb-thm}.
 The version of this result for $\PRSK$ can be derived in a more elementary way from
 known properties of \defn{(dual) Knuth equivalence}, and was discussed previously in \cite{Post}; see Theorem~\ref{direct-prop}.

We were unexpectedly lead to consider these maps for applications in representation theory,
specifically to the problem of classifying the cells and molecules in certain $W$-graphs for 
the symmetric group $W=S_n$. 
Recall that each Coxeter group $W$ has an associated \defn{Iwahori-Hecke algebra} $\cH$ 
which is equipped with both a \defn{standard basis} $\{ H_w : w \in W\}$ and a \defn{Kazhdan--Lusztig basis} $\{ \underline H_w : w \in W\}$.
The action of the standard basis on the Kazhdan--Lusztig basis by left and right multiplication is encoded in two directed graphs,
called the \defn{left and right Kazhan-Lusztig graphs} of $W$. These objects are the motivating examples of 
\defn{$W$-graphs}, which are certain weighted directed graphs that encode $\H$-representations with
canonical bases analogous to $\{ \underline H_w : w \in W\}$. For the precise definition of a $W$-graph, see Section~\ref{w-sect}.

The principal combinatorial problem related to a given $W$-graph is to classify its \defn{cells},
which are its strongly directed components. This is because the original $W$-graph structure restricts to a $W$-graph on each cell. Moreover, the collection of cells is naturally a directed acyclic graph which induces a filtration
on the $W$-graph's associated $\H$-module. A related problem is to describe the \defn{molecules} in $W$-graph:
these consist of the connected components in the undirected graph whose edges are the pairs of $W$-graph vertices $\{x,y\}$
with edges $x\to y$ and $y \to x$ in both directions. 

Finding the molecules in a $W$-graph is easier than identifying its cells,
and each cell is a union of one or more molecules. However, in some special cases of interest, the cells and molecules in a $W$-graph coincide. 
Most notably, this occurs for the left and right Kazhdan--Lusztig graphs of the symmetric group  \cite[\S6.5]{CCG}.
The molecules (equivalently, the cells) in these $W$-graphs are the subsets on which $\QRSK$ and $\PRSK$ are respectively constant \cite[Thm. 1.4]{KL}.

Our results in Section~\ref{3-sect}
show that the row\footnote{For parallelism, it is convenient to refer to ``row Beissinger insertion'' but note from \eqref{beis-eq} that this gives the same output when applied to $w=w^{-1}\in S_n$ as Robinson--Schensted insertion.}  and column Beissinger insertion algorithms described above
have a similar relationship to the molecules (and conjecturally, the cells)
in a different pair of $W$-graphs for $W=S_n$.
In \cite{MZ} we introduced the notion of a \defn{perfect model} for a finite Coxeter group.
A perfect model consists of a set of linear characters of subgroups satisfying some technical conditions;
the name derives from the requirement that each subgroup be the centralizer of a \defn{perfect involution}
in the sense of \cite{RV} in a standard parabolic subgroup.
Each perfect model gives rise to a pair of $W$-graphs whose underlying $\H$-representations are \defn{Gelfand models},
meaning they decompose as multiplicity-free sums of all irreducible $\H$-modules.

Our previous paper \cite{MZ2} classified the perfect models in all finite Coxeter groups up to a natural form of equivalence.
For the symmetric group $S_n$ when $n\notin\{2,4\}$, there is just one equivalence class of perfect models \cite[Thm. 3.3]{MZ2},
and this defines a canonical pair of \defn{Gelfand $S_n$-graphs}  $\mGamma$ and $\nGamma$.
We review the explicit construction of these graphs in Section~\ref{gel-sect}. Their underlying vertex sets are certain subsets of fixed-point-free
involutions in $S_{2n}$, whose images under both forms of Beissinger insertion are standard tableaux with $2n$ boxes.
We can summarize our main result connecting $\mGamma$ and $\nGamma$ to Beissinger insertion as follows: 

\begin{theorem*}
The molecules in the $S_n$-graphs $\mGamma$ (respectively, $\nGamma$) are the sets of vertices whose images under row (respectively, column) Beissinger insertion have the same shape when the boxes containing $n+1,n+2,\dots,2n$ are omitted.
\end{theorem*}

This result combines Theorems~\ref{mleft-thm} and \ref{nleft-thm}, which are proved in Section~\ref{mleft-sect};
see also Theorems~\ref{pmrsk-thm} and \ref{pnrsk-thm}.
At present, it is an open problem to upgrade this result to a classification of the cells in $\mGamma$ and $\nGamma$.
One reason this is difficult is that the $W$-graphs $\mGamma$ and $\nGamma$ are not \defn{admissible} in the sense of \cite{Nguyen,Stembridge}.
We suspect that the following is true, however:

\begin{conjecture*}[{\cite[Conj. 1.16]{MZ}}]
Every molecule in the $S_n$-graphs $\mGamma$ and $\nGamma$ is a cell.
\end{conjecture*}

We have done computer calculations to verify this conjecture for $n\leq 10$.
By dimension considerations, this statement is equivalent 
 to the claim that the cell representations for $\mGamma$ and $\nGamma$ are all irreducible, which we stated earlier as \cite[Conj. 1.16]{MZ}.

The rest of this paper is organized as follows. Section~\ref{2-sect} contains some preliminaries on the Robinson--Schensted correspondence and Knuth equivalence, as well as our main results on Beissinger insertion. Section~\ref{3-sect} reviews the construction
of the $S_n$-graphs  $\mGamma$ and $\nGamma$ and then proves our results about the molecules in these graphs.
Appendix~\ref{app-sect}, finally, carries out the technical proof of Theorem~\ref{n-cb-thm}.

\section{Insertion algorithms}\label{2-sect}

Throughout, $n$ is a fixed positive integer, $S_n$ is the  group of permutations of 
$[n] := \{1,2,\dots,n\}$,
$\cI_n := \{ w \in S_n : w=w^{-1}\}$ is the set of involutions in $S_n$,
and  $\Ifpf_{n}$ is the subset of fixed-point-free elements of $\cI_{n}$ (which is empty if $n$ is odd).
Let $s_i$ denote the simple transposition $(i,i+1) \in S_n$.

\subsection{Schensted insertion}\label{sch-sect}

The \defn{(Young) diagram} of an integer partition $\lambda =(\lambda_1\geq \lambda_2 \geq \dots \geq \lambda_k>0)$
is the set of positions
$\D_\lambda := \{ (i,j) \in [k]\times \ZZ: 1 \leq j \leq \lambda_i\}.$ A \defn{tableau} of shape $\lambda$
is a map $T : \D_\lambda \to \ZZ$, which we envision as an assignment of numbers
to some set of positions in a matrix. 

A tableau is \defn{semistandard} if its rows are weakly increasing and its columns are strictly increasing.
A tableau is \defn{standard} if its rows and columns are strictly increasing and its entries
are the numbers $1,2,3,\dots,n$ for some $n\geq 0$ without any repetitions.
Most of the tableaux considered in this article
will be semistandard but with all distinct entries; we refer to such tableaux as \defn{partially standard}.

As already discussed in the introduction, the \defn{Robinson--Schensted (RS) correspondence} is a bijection
from permutations 
to pairs 
 of standard tableaux of the same shape, which can be described 
using the following   insertion process.

\begin{definition}[\defn{Schensted insertion}]
Suppose $T$ is a partially standard tableau  and $x$ is an integer.
Start by inserting $x$ into the first row of $T$ by finding the row's first entry $y$ greater than $x$ and replacing $y$ by $x$.
If there is no such entry $y$ then $x$ is placed at the end of the row, and otherwise one proceeds by
inserting $y$ into the next row by the same process. 
Continue in this way until a new box is added to the end of a row of $T$.
Denote the result by $T \fromRSK x$. 
\end{definition}

\begin{example*}
We have $\ytab{ 1 & 5 \\ 3 & 6 \\ 4} \fromRSK 2 = \ytab{1 & 2 \\ 3 & 5 \\  4 & 6}$.
\end{example*}

\begin{definition}[RS correspondence]
For a permutation $w=w_1w_2\cdots w_n\in S_n$ 
let
\[\PRSK(w):= \emptyset \fromRSK w_1 \fromRSK w_2 \fromRSK \cdots \fromRSK w_n\] and 
let $\QRSK(w)$ be the tableau of the same shape with $i$ in the box added by the $\fromRSK w_i$ step.
\end{definition}

\begin{example*}
One can check that $\PRSK(31425) = \ytab{1 &2 & 5 \\ 3 & 4}$ and $\QRSK(31425) = \ytab{1 & 3 & 5 \\ 2 & 4}$.
\end{example*}

It is well-known that if $w \in S_n$ then $\PRSK(w^{-1}) = \QRSK(w)$; see, e.g., 
\cite[Thm. 6.4]{EG}.

\begin{example*}
It holds that $\PRSK(24135) =  \ytab{1 & 3 & 5 \\ 2 & 4}$ and $\QRSK(24135) = \ytab{1 &2 & 5 \\ 3 & 4}$.
\end{example*}

The \defn{row reading word} of
 a tableau $T$ is the sequence $\row(T)$ given by reading the rows of $T$ from left to right, but starting with the last row. For example, 
$
\row\(\hs\ytab{1&2&5\\3&4}\hs\)=34125.
$
It is easy to see that $\PRSK(\row(T)) = T$.

Fix a standard tableau $T$  with $n$ boxes.
Given a permutation $w \in S_n$, let $w(T)$ be the 
tableau formed by applying $w$ to each entry of $T$.
For each integer  $1<i<n$,
the \defn{elementary dual equivalence operator} $D_i$ is the map acting on $T$ by
\be\label{D-def}
D_i(T) :=\begin{cases}
s_{i-1}(T) &\text{ if $i+1$ lies between $i$ and $i-1$ in }\row(T),\\
s_{i}(T) &\text{ if $i-1$ lies between $i$ and $i+1$ in }\row(T),\\
T &\text{ if $i$ lies between $i-1$ and $i+1$ in }\row(T).
\end{cases}
\ee
This definition follows  \cite{Assaf}, and is equivalent to the one given by Haiman in \cite{Haiman}. It is an instructive exercise to check that the operator $D_i$ is an involution and always produces another standard tableau \cite[\S2.3]{Assaf}.

Suppose $a<b<c$ are integers.
There are four permutations of these numbers that are not strictly increasing or strictly decreasing, namely, $acb$, $bac$, $bca$, and $cab$. A \defn{Knuth move} on these words exchanges the $a$ and $c$ letters. Thus $acb$ and $cab$ are connected by a Knuth move, as are $bca$ and $bac$.

Suppose $v,w \in S_n$ and $i$ is an integer with $1<i<n$.
We write $v \K{i} w$
and say that a \defn{Knuth move} exists between $v$ and $w$ 
 if either 
 \ben
 \item[(a)] $w$ is obtained from $v$ by performing a Knuth move on $v_{i-1}v_iv_{i+1}$, or 
\item[(b)] $w$ is equal to $v$ and the subword $v_{i-1}v_iv_{i+1}$ is in monotonic order.
\een
Similarly, we write  $v \dK{i} w$ and say that a \defn{dual Knuth move} exists between $v$ and $w$ if $v^{-1} \K{i} w^{-1}$.
Two permutations that are connected by a sequence of (dual) Knuth moves are called \defn{(dual) Knuth equivalent}. 
For example, we have 
 \[25431\K{2}52431\K{3}54231
\quand 43251\dK{4} 53241\dK{3} 54231.\]
These relations are connected to the RS correspondence by the following identities.

\begin{theorem}[{\cite{EG,Haiman}}] 
\label{A-thm}
Let $v,w \in S_n$ and $1<i<n$.
Then:
\ben
\item[(a)] One has $v\K{i}w$ if and only if $\PRSK(v) = \PRSK(w)$ and $\QRSK(v)=D_i(\QRSK(w))$.
\item[(b)] One has $v\dK{i}w$ if and only if $\PRSK(v)=D_i(\PRSK(w))$ and $\QRSK(v) = \QRSK(w)$.

\item [(c)] The permutations $v$ and $w$ are Knuth equivalent if and only if $\PRSK(v)=\PRSK(w)$,
and dual Knuth equivalent if and only if $\QRSK(v)=\QRSK(w)$.
\een
\end{theorem}

This well-known theorem is usually attributed to 
Edelman--Greene \cite{EG} or Haiman \cite{Haiman}. 
It takes a bit of reading to find  equivalent statements in those sources, however.
These can be found in one place in the expository reference \cite[\S4.1]{Post}.
Specifically, parts (a) and (b) are  \cite[Thm. 4.2 and Cor. 4.2.1]{Post}; see also \cite{Reif}.
Part (c) is equivalent to \cite[Thm. 6.6 and Cor. 6.15]{EG}.


\begin{remark*}
Let $T$ and $U$ be two standard tableaux of the same shape.
Then it is well-known that there exists a sequence of dual equivalence operators with $T = D_{i_1} D_{i_2}\dots D_{i_k}(U)$. 
For example:
\[
\ytab{1&2&5\\3&4}\xleftrightarrow{\ D_2,D_3\ }\ytab{1&3&5\\2&4}\xleftrightarrow{\ D_4\ }\ytab{1&3&4\\2&5}\xleftrightarrow{\ D_3\ }\ytab{1&3&4\\2&5}.
\]
This property can be deduced from Theorem~\ref{A-thm}.
 There are unique permutations $v$ and $w$ with $\PRSK(v) = T$ and $\QRSK(v) = \PRSK(w) =\QRSK(w) = U$.
 These permutations must be dual Knuth equivalent by part (c) of the theorem
 and so there exists a chain of dual Knuth moves
 \[v\dK{i_1}\cdots \dK{i_k} w.\] Then by part (b) of the theorem   
 $T = \PRSK(v) =  D_{i_1} D_{i_2}\dots D_{i_k}(\PRSK(w)) = D_{i_1} D_{i_2}\dots D_{i_k}(U).$
\end{remark*}

\subsection{Row Beissinger insertion}\label{rbs-sect}

Here and in the next section 
we consider two variants $T \fromRBS (a,b)$ and $T \fromCBS (a,b)$ of the   Schensted insertion algorithm $T \fromRSK a$.
These variants insert a pair of integers $(a,b)$ with $a\leq b$ into a partially tableau $T$.
The first operation is given below:

\begin{definition}[Row Beissinger insertion]
\label{rBS-def}
Let $(i,j)$ be the box of $T \fromRSK a$ that is not in $T$.
If $a<b$  then form $T \fromRBS (a,b)$ by adding $b$ to the end of \textbf{row} $i+1$ of $T \fromRSK a$.
If $a=b$ then form $T \fromRBS (a,b)$ by adding $b$ to the end of the first \textbf{row} of $T$.
\end{definition}

\begin{example*}
We have
$ \ytab{
 1& 2 & 3 
  \\4
}\fromRBS (5,5)
=
 \ytab{
 1& 2 & 3 & 5
  \\4
}
$
and
$
\ytab{
  1 & 4 &6
  \\3
} \fromRBS (2,5)
   = 
\ytab{
   1 & 2 &6 \\
   3  & 4\\
   5
}.$
\end{example*}

The operation $T\fromRBS(a,b)$ is identical to the one which Beissinger denotes as $T+(a,b)$ in \cite[Alg. 3.1]{Beissinger},
so we refer to it as \defn{row Beissinger insertion}. 
The motivation for this operation in \cite{Beissinger} is to describe the Robinson--Schensted correspondence restricted to involutions.
Since $\PRSK(w^{-1}) = \QRSK(w)$, we have $\PRSK(w)  =\QRSK(w)$ if and only if $w=w^{-1} \in \cI_n$.

\begin{definition}[Row Beissenger correspondence]
Given $z \in \cI_n$ let $(a_1,b_1)$,  \dots, $(a_q,b_q)$ 
be the list of pairs $(a,b) \in [n]\times[n]$ with $a \leq b = z(a)$, ordered with $b_1<\dots<b_q$,
and define 
\[
\PRBS(z) :=  \emptyset \fromRBS (a_1,b_1)\fromRBS (a_2,b_2) \fromRBS\cdots \fromRBS (a_q,b_q).
\]
\end{definition}

\begin{example*} We have
$
\PRBS(4231)=\emptyset \fromRBS (2,2)\fromRBS (3,3)\fromRBS (1,4)=\ytab{1 &  3 \\ 2 \\ 4}.
$
\end{example*}

If $a\leq b$ are arbitrary positive integers, then 
$T\fromRBS (a,b)$ may fail to be partially standard or even semistandard (this is easy to see when $a=b$).
Therefore, it is not obvious that $\PRBS(z) $ is standard.
This turns out to hold because of the particular order in
which the pairs $(a_i,b_i)$ are inserted. For example, we will only insert $(a_i,b_i)$ with $a_i=b_i$ if all  numbers in the previous tableau are smaller than $a_i$.
This sequencing allows the following to hold:

\begin{theorem}[{Beissinger \cite[Thm. 3.1]{Beissinger}}]
\label{beissinger-thm1}
If $z\in\I_n$ then $\PRBS(z) = \PRSK(z)=\QRSK(z)$.
\end{theorem}

A row or column of a tableau is \defn{odd} if it has an odd number of boxes.

\begin{theorem}[\cite{Beissinger}]\label{prbs-thm}
The operation $\PRBS$ defines a bijection from 
$\cI_n$ to the set of standard Young tableaux with $n$ boxes. This map restricts to a bijection from 
the set of involutions in $S_n$ with $k$ fixed points to the set of standard Young tableaux
with $n$ boxes and $k$ odd \textbf{columns}.
\end{theorem}

\begin{proof}
This follows from Theorem~\ref{beissinger-thm1} 
since the  operation $\fromRBS (a,b)$ preserves the number of odd columns when $a<b$ and 
increases the number of odd columns by one when $a=b$.
\end{proof}

In view of this result, it is natural to introduce a relation $\simRBS{i}$ on $\cI_n$ for each $1<i<n$,
defined by requiring that $y\simRBS{i} z$ if and only if $\PRBS(y) = D_i(\PRBS(z))$.
The following shows that $\simRBS{i}$ is the same as what is called an \defn{involutive transformation} in \cite[\S4.2]{Post}.

\begin{lemma}\label{simrbs-lem}
Let $y,z\in \cI_n$. Then $y\simRBS{i}z$ if and only if
 $y\K{i}w\dK{i}z$ for some $w \in S_n$.
\end{lemma}

\begin{proof}
If  $y\K{i}w\dK{i}z$ then $\PRSK(y)=\PRSK(w)=D_i(\PRSK(z))$ by Theorem~\ref{A-thm} so  $y\simRBS{i}z$.
Conversely, if  $y\simRBS{i}z$ so that $\PRSK(y)=\QRSK(y)=D_i(\PRSK(z))=D_i(\QRSK(z))$,
then  $y\K{i}w\dK{i}z$ for the element $w\in S_n$ with
$\PRSK(w)=\PRSK(y)$ and $\QRSK(w)=\QRSK(z)$ by Proposition~\ref{A-thm}.
\end{proof}

If $y \in \I_n$ and $1<i<n$ then  $y \simRBS{i} z$ 
for a unique $z \in \I_n$, which has this characterization:

\begin{theorem}\label{direct-prop} If $y,z\in \cI_n$ have $y\simRBS{i}z$   and 
$A:=\{i-1,i,i+1\}$ then
\[
z = \begin{cases} y &\text{if $y(A) \neq A$ and $y(i)$ is between $y(i-1)$ and $y(i+1)$}
\\
 (i-1,i)y (i-1,i) &\text{if $y(A) \neq A$ and $y(i+1)$ is between $y(i-1)$ and $y(i)$}
\\
(i,i+1)y (i,i+1) &\text{if $y(A) \neq A$ and $y(i-1)$ is between $y(i)$ and $y(i+1)$ }
\\
 (i-1,i+1)y(i-1,i+1) &\text{if $y(A) = A$}.
\end{cases}
\]
\end{theorem}

\begin{remark*}
Besides the first case, we can also have $y=z$ in the fourth case if $y$ restricts
to either the identity permutation or reverse permutation of $\{i-1,i,i+1\}$.
\end{remark*}

\begin{proof}
The \defn{arc diagram} of $y\in \cI_n$ is the matching on $[n]$ whose edges give the cycles of $y$.
This is typically drawn so that the vertices corresponding to $1,2,3,\dots,n$ are arranged from left to right.
The theorem can be derived by inspecting \cite[Figure 4.11]{Post}, which 
lists the ways that the arc diagrams of $y$ and $z$ can differ if $y\simRBS{i} z \neq y$.
Translating \cite[Figure 4.11]{Post} into our formulation is not entirely straightforward, so we include a self-contained proof below.

First suppose $y(A)=A$ and define $z=(i-1,i+1)y(i-1,i+1)$. 
Then the sequence $y(i-1)y(i)y(i+1)$ is either 
\[(i-1)i(i+1), \quad (i+1)i(i-1), \quad i(i-1)(i+1), \quord (i-1)(i+1)i\]
and it is straightforward to check that 
we have
$y\K{i}y\dK{i}y=z$ in the first two cases,
$y\K{i}ys_i\dK{i}z$ in the third case, and
$y\K{i}ys_{i-1}\dK{i}z$ in the last case.
Thus, by Lemma~\ref{simrbs-lem}, we conclude that $y\simRBS{i}z$ as needed.

For the rest of this proof we assume   $y(A)\neq A$. Define $a<b<c$ to be the numbers with $\{a,b,c\} = \{y(i-1),y(i),y(i+1)\}=y(A)$.
If $y(i)$ is between $y(i-1)$ and $y(i+1)$, then the sequence $y(i-1)y(i)y(i+1)$ is either $abc$ or $cba$ so $y\K{i}y= y^{-1} \dK{i} y^{-1} =y$ and $y\simRBS{i} y$ as claimed.

Suppose instead that $y(i+1)$ is between $y(i-1)$ and $y(i)$.
Then $y(i-1)y(i)y(i+1)$ is either $cab$ or $acb$ so $y\K{i}ys_{i-1}$ and $y\dK{i} s_{i-1} y$.
The relative order of $i-1$, $i$, and $i+1$ in the one-line notation of $ys_{i-1}$
 can only differ from that of $y$ if $\{y(i-1),y(i)\} = \{a,c\}\subset \{i-1,i,i+1\}$,
 which would require us to have $y(A) = \{a<b<c\} = \{i-1< i<i+1\}=A$.
As we assume $y(A) \neq A$, 
the relation $y\dK{i} s_{i-1} y$ implies that $ys_{i-1} \dK{i} s_{i-1} y s_{i-1}$,
so $y\simRBS{i} s_{i-1}y s_{i-1}$ as claimed.

Finally, suppose 
 $y(i-1)$ is between $y(i)$ and $y(i+1)$. Then $y(i-1)y(i)y(i+1)$ is either $bca$ or $bac$ so $y\K{i}ys_{i}$
 and $y\dK{i} s_{i}y$. 
Now the relative order of $i-1$, $i$, and $i+1$ in the one-line notation of $ys_{i}$
 can only differ from that of $y$ if $\{y(i),y(i+1)\} = \{a,c\}\subset \{i-1,i,i+1\}$,
 which again would force us to have $y(A) = A$.
Therefore $y\dK{i} s_{i} y$ implies that $ys_{i} \dK{i} s_{i} y s_{i}$
so  $y\simRBS{i} s_{i}y s_{i}$.
\end{proof}

\subsection{Column Beissinger insertion}\label{cbs-sect}

 Again suppose $T$ is a partially standard tableau
and $a\leq b$ are integers.  
Changing ``row'' to ``column'' in Definition~\ref{rBS-def} gives the following insertion operation, which is the main topic of this section:

\begin{definition}[Column Beissinger insertion]
\label{cRS-def}
Let $(i,j)$ be the box of $T \fromRSK a$ that is not in $T$.
If $a<b$  
then form $T \fromCBS (a,b)$ by adding $b$ to the end of \textbf{column} $j+1$ of $T \fromRSK a$.
If $a=b$ then form $T \fromCBS (a,b)$   by adding $b$ to the end of the first \textbf{column} of $T$.
\end{definition}

\begin{example*}
We have
$ \ytab{
  1&2 & 3 
  \\4
}\fromCBS (5,5)
=
 \ytab{
  1&2 & 3 
  \\4
  \\ 5
}
$
and
$
\ytab{
  1 & 4 &6
  \\3
} \fromCBS (2,5)
    = 
\ytab{
   1 & 2  &6\\
   3  & 4 & 5 
}.$
\end{example*}

 By symmetry we refer to $\fromCBS$  
as
\defn{column Beissinger insertion}, although this operation is not considered in \cite{Beissinger}
and does not appear to have been studied previously.

\begin{definition}[Column Beissinger correspondence]
Given $z \in \cI_n$, let $(a_1,b_1)$, \dots, $(a_q,b_q)$ 
be the list of pairs $(a,b) \in [n]\times[n]$ with $a \leq b = z(a)$, ordered with $b_1<\dots<b_q$,
and define 
\[
\PCBS(z) :=  \emptyset \fromCBS (a_1,b_1) \fromCBS (a_2,b_2)\fromCBS\cdots \fromCBS (a_q,b_q).
\]
\end{definition}

\begin{example*}
We have $\PCBS(4231)=\emptyset \fromCBS (2,2)\fromCBS (3,3)\fromCBS (1,4)=\ytab{1 &  4 \\ 2 \\ 3}.$
\end{example*}

As with row Beissinger insertion, for an arbitrary pair $a\leq b$ the tableau $T \fromCBS (a,b)$ may fail to be partially standard.
Thus, the fact that $\PCBS(z)$ is always standard, which is part of Theorem~\ref{pcbs-thm} below, depends on the particular order in which the pairs $(a_i,b_i)$
are inserted in the preceding definition.

There does not seem to be any simple relationship between $\PRBS(z) $ and $\PCBS(z)$.
Nevertheless, we will see that the formal properties of the map $\PCBS$ closely parallel those of $\PRBS$.

One can perform \defn{inverse Schensted insertion} starting from any corner box in
a partially tableau $T$
 to obtain another partially tableau $U$ and an integer $x$
such that $T = U\fromRSK x$. Here $U$ and $x$ are uniquely determined by 
requiring that $T = U\fromRSK x$ and that the shape of $U$ be the shape of $T$ with the relevant corner box deleted.
The explicit algorithm starts by
removing the corner box, say with entry $c$ in row $k+1$, and inserting $c$ into row $k$. We replace the last entry $b<c$ in row $k$ by $c$ and then insert $b$ into row $k-1$ by the same procedure (replacing the last entry $a<b$ by $b$, then inserting $a$ into row $k-2$, and so on).
We continue in this way  to form $U$, and let $x$ be the entry replaced in the first row.

\begin{theorem}\label{pcbs-thm}
The operation $\PCBS$ defines a bijection from 
$\cI_n$ to the set of standard Young tableaux with $n$ boxes. This map restricts to a bijection from 
the set of involutions in $S_n$ with $k$ fixed points to the set of standard Young tableaux
with $n$ boxes and $k$ odd \textbf{rows.}
\end{theorem}

\begin{proof}
We show that $\PCBS$ is a bijection by constructing the inverse algorithm. 
Suppose $T$ is a partially standard tableau with $n$ boxes.
Find the largest entry $b$ in $T$. If this is in the first column, then let $a:=b$ and delete this box to form a smaller tableau $U$. 
Otherwise,  let $x$ be the entry in $T$ that is at the end of the column preceding $b$.
Delete the box of $T$ containing $b$ to form a tableau $\tilde T$. Then $x$ is an a corner box of $\tilde T$ so we can 
 do inverse Schensted insertion starting from $x$ to obtain a tableau $U$ and an integer $a$ such that $\tilde T \fromRSK a = U$. 
In either case we obtain a partially standard tableau $U$ and a pair of integers $a\leq b$ with $T = U \fromCBS (a,b)$.

If we apply this operation successively to a standard tableau $T$ with $n$ boxes,
then we obtain a sequence of pairs $(a_i, b_i)$ with $a_i \leq b_i$ 
and
$T = \emptyset \fromCBS (a_1,b_1) \fromCBS (a_2,b_2)\fromCBS\cdots \fromCBS (a_q,b_q)$.
By construction these pairs satisfy $b_i<b_j$ and $\{a_i,b_i\} \cap \{a_j,b_j\} = \varnothing$ for all $i<j$
while having $[n] = \{a_1,b_1,a_2,b_2,\dots,a_q,b_q\}$.
Hence there is a unique involution $y \in \cI_n$ whose disjoint (but possibly trivial) cycles are $(a_i,b_i)$ for $i \in [q]$
and this element has $\PCBS(y) = T$.
The map $\PCBS^{-1}(T) := y$ is the two-sided inverse of $\PCBS$.

The reason why $\PCBS$ turns  fixed points into odd rows, finally, is
because   $\fromCBS (a,b)$ preserves the number of odd rows when $a<b$ and 
increases the number of odd rows by one when $a=b$.
\end{proof}

Continuing our parallel stories, for each $1<i<n$ let $\simCBS{i}$ be the relation on $\cI_n$ 
with $y\simCBS{i} z$ if and only if $\PCBS(y) = D_i(\PCBS(z))$.
For any $y \in \I_n$ and $1<i<n$ there is a unique $z \in\I_n$ with $y \simCBS{i} z$.
This element has a slightly more complicated characterization than Theorem~\ref{direct-prop}. 

\begin{theorem}\label{n-cb-thm}
 Suppose $y,z\in \cI_n$ have $y\simCBS{i}z$ for some $1<i<n$. For $j \in \{i-1,i,i+1\}$ 
let 
\be\label{upsilon-def}
\ymark(j) := \begin{cases} 
y(j)&\text{if } y(j) \notin \{i-1,i,i+1\} \\
-j &\text{if }y(j)=j \\
j &\text{if } j \neq y(j) \in \{i-1,i,i+1\}.
\end{cases}
\ee
Then it holds that
\[
z = \begin{cases} y &\text{if $\ymark(i)$ is between $\ymark(i-1)$ and $\ymark(i+1)$}
\\
 (i-1,i)y(i-1,i) &\text{if $\ymark(i+1)$ is between $\ymark(i-1)$ and $\ymark(i)$}
\\
(i,i+1)y(i,i+1) &\text{if $\ymark(i-1)$ is between $\ymark(i)$ and $\ymark(i+1)$.}
\end{cases}
\]
\end{theorem}

Our proof of this result is quite technical. We postpone the details
of the argument to Appendix~\ref{app-sect} in order to avoid sidetracking our present discussion.

Comparing $\PRBS$ and $ \PCBS$ suggests an interesting
operation $\PhiBS$ on involutions: for each $y \in \I_n$ there is a unique $z \in \I_n$ such that
$\PRBS(y) = \PCBS(z)^\top$, where $\top$ denotes the transpose, and we define $\PhiBS(y) := z$.
Since taking transposes turns odd rows into odd columns,
the map $\PhiBS$ is a permutation of $\I_n$ which preserves each element's number of fixed points,
or equivalently which preserves the $S_n$-conjugacy classes in $\I_n$.
In addition, $\PhiBS$ commutes with the natural inclusion $\I_n \hookrightarrow \I_{n+1}$
adding $n+1$ as a fixed point to each element of $\I_n$.
For example, if $n=4$ then
\[
\ba \PhiBS:\ \  &1\mapsto 1,\quad (1,2)\mapsto(1,2),\quad
(1,3) \mapsto (2,3) \mapsto(1,3),\quad  
(1,4) \mapsto (2,4) \mapsto (3,4) \mapsto(1,4),\quad
\\&(1, 2)(3, 4) \mapsto (1, 3)(2, 4) \mapsto (1, 4)(2, 3) \mapsto (1, 2)(3, 4).
\ea
 \]
For large $n$ this map is fairly mysterious.
Its longest cycles for $n=1,2,3,\dots,10$ have sizes $1$, $1$, $2$, $3$, $12$, $15$, $46$, $131$, $630$, $1814$
 and in general $\PhiBS$ is very close to a derangement:
 
 \begin{proposition}\label{only-fixed-prop}
The only fixed points of $\PhiBS: \cI_n \to \cI_n$ are $1$ and $s_1=(1,2)$.
 \end{proposition}
 
 This result was a conjecture in an earlier version of this article.
 The following proof was shown to us by Joel Lewis.
 
 \begin{proof}
 We have $\PhiBS(1) = 1$ since
 $\PRBS(12\cdots n) = \PCBS(12\cdots n)^\top = \ytab{1&2&\cdots &n}
$, while $\PhiBS(s_1) = s_1$ since $\PRBS(213\cdots n) = \PCBS(213\cdots n)^\top = \ytab{1&3&\cdots &n \\ 2}.$
  To prove that $\PhiBS$ has no other fixed points, 
choose $w \in I_n$ and 
let $(a_1,b_1)$, $(a_2,b_2)$, \dots, $(a_q,b_q)$ 
be the integer pairs $(a,b)$ with $1\leq a \leq b = w(a)\leq n$, ordered such that $b_1<b_2<\dots<b_q$.
For each $i \in [q]$ let
\[ T_i:=  \emptyset \fromRBS (a_1,b_1)\fromRBS(a_2,b_2) \fromRBS\cdots \fromRBS (a_i,b_i)\] 
Also define $\fromCBSalt$ as in Remark~\ref{eqnat-rmk} and let
\[
 U_i:=  \emptyset \fromCBSalt (a_1,b_1)\fromCBSalt (a_2,b_2) \fromCBSalt\cdots \fromCBSalt (a_i,b_i).
\]
Then we have $\PRBS(w) = T_q$ and $\PCBS(w)^\top = U_q$.

Assume $w\neq 1$. Then there is a maximal
 $j \in [q]$  with $a_i<b_j$. 
The tableau $T_q$ is formed from $T_j$ by adding $b_{j+1},b_{j+2},\dots,b_q$ to the end of the first row, while $U_q$ is formed in the same way from $U_j$.
We therefore have $T_q=U_q$ if and only if $T_j = U_j$.

Further suppose $a_j > 1$.
Since $T_q$ and $U_q$ are standard, the number $1$ must be present in both $T_{j-1}$ and $U_{j-1}$. 
As $T_{j-1}$ and $U_{j-1}$ are partially standard, the number $1$ is necessarily in box $(1,1)$ of both tableaux.
But this means that $a_j$ will appear in the \textbf{first row but not the first column} of $T_j = T_{j-1} \fromRBS (a_j,b_j)$,
since $T_j$ is formed by row inserting $a_j$ into $T_{j-1}$ and then adding an extra box containing $b_j$.
On the other hand, $a_j$ will appear in the \textbf{first column but not the first row} of $U_j = U_{j-1} \fromCBSalt U_{j-1}$
since $U_j$ is formed by column inserting $a_j$ into $U_{j-1}$ and then adding $b_j$. Thus $T_j \neq U_j$ so also $T_q \neq U_q$.

Instead suppose $a_j=1$ but $b_j>2$. 
Then the number $2$ must appear in box $(1,1)$ of both  $T_{j-1}$ and $U_{j-1}$
since these tableaux are partially standard and do not contain $1$.
Therefore, when we form $T_j = T_{j-1} \fromRBS (1,b_j)$
our row insertion of $1$ will bump $2$ to box $(2,1)$,
but when we form 
$U_j = U_{j-1} \fromCBSalt (1,b_j)$
our column insertion of $1$ will bump $2$ to box $(1,2)$.
Thus we again have $T_j \neq U_j$ as these tableaux contain the number $2$ in different positions $(2,1)\neq (1,2)$. We conclude as before that $T_q \neq U_q$.

In the remaining case when $a_j=1$ and $b_j=2$,
 the maximality of $j$ implies that $w=s_1$.
Hence if $w\notin \{1,s_1\}$ then $\PRBS(w) = T_q$ is distinct from $\PCBS(w)^\top = U_q$, so $\PhiBS(w)\neq w$.
 \end{proof}

\section{Molecules in Gelfand $W$-graphs}\label{3-sect}

In this section we explain how the row and column Beissenger insertion algorithms 
are related to certain \defn{$W$-graphs} (for $W=S_n$) studied in \cite{MZ}.
The latter objects are derived from a pair of Iwahori-Hecke algebra modules 
described in Section~\ref{gel-sect}.

\subsection{Iwahori-Hecke algebras and $W$-graphs}\label{w-sect}

We briefly review some general background material from \cite[Chapter 7]{Humphreys}.
The \defn{Iwahori-Hecke algebra} $\H=\H(W)$ of an arbitrary Coxeter system $(W,S)$ with length function $\ell : W \to \NN$
is the $\ZZ[x,x^{-1}]$-algebra with basis $\{ H_w : w \in W\}$ satisfying 
\[ H_s H_w = \begin{cases} H_{sw} &\text{if }\ell(sw) > \ell(w) \\ 
H_{sw} + (x-x^{-1}) H_w &\text{if }\ell(sw) < \ell(w)\end{cases}
\quad\text{for }s \in S\text{ and } w\in W.\]
The unit of this algebra is $H_1=1$.
There is a unique ring involution of $\H$, written $h \mapsto \overline h$ and called the \defn{bar operator},
such that $\overline{x} = x^{-1}$ and $\overline{H_{s}} = H_{s}^{-1} = H_s -(x-x^{-1})$ for all $s \in S$.
More generally, an \defn{$\H$-compatible bar operator} for an $\H$-module $\cA$ is a $\ZZ$-linear map $\cA\to \cA$,
 also written $a \mapsto \overline{a}$, such that $\overline{ha} = \overline{h}\cdot\overline{a}$
for all $h \in \H$ and $a \in \cA$.

Following the conventions in \cite{Stembridge}, we define a \defn{$W$-graph} 
  to be a triple $\Gamma=(V,\omega,\tau)$  consisting of a set $V$ with maps
 $\omega : V\times V \to \LL$
and $ \tau : V \to \{\text{ subsets of $S$ }\}$
such that the free $\LL$-module with basis $\{ Y_v : v \in V \}$ has a left $\H$-module structure 
   in which
  \be\label{hy-eq}
    H_s Y_v = \begin{cases} x Y_v &\text{if }s \notin \tau(v) \\ 
 -x^{-1} Y_v +  \ds\sum_{ \substack{w \in V \\ s\notin \tau(w)}} \omega(v,w) Y_w &\text{if }s \in \tau(v)
  \end{cases}
\quad\text{for all $s \in S$ and $ v \in V$.}\ee
We view $\Gamma$ as a weighted digraph 
with edges $v \xrightarrow{\omega(v,w)} w$ for each $v,w \in V$ with $\omega(v,w)\neq 0$.

\begin{remark}\label{reduced-rmk}
The values of $\omega(v,w)$ when $\tau(v) \subseteq \tau(w)$ play no role in the formula \eqref{hy-eq}.
Thus, when considering the problem of classifying $W$-graphs,
it is natural to impose the further condition (called \defn{reducedness} in \cite{Stembridge}) that $\omega(v,w) = 0$ if $\tau(v) \subseteq \tau(w)$.
Although we adopted this convention in \cite{MZ}, we omit it here. This simplifies some formulas.
\end{remark}

\begin{example}\label{kl-ex}
The \defn{left and right Kazhdan--Lusztig $W$-graphs}
are described as follows.
The \defn{Kazhdan--Lusztig basis} of $\H $ is the unique set of elements $\{ \underline H_w : w \in W\}$
satisfying \be\label{kl-eq}
\underline H_w = \overline{\underline H_w} \in H_w + \sum_{\substack{y \in W \\ \ell(y)<\ell(w)}} x^{-1} \ZZ[x^{-1}]  H_y.\ee
The uniqueness of this set can be derived by the following simple argument.
Since for any $w \in W$ we have $\overline{H_w} \in H_w + \sum_{\ell(y) < \ell(w)} \ZZ[x,x^{-1}] H_y$,
the only element of $  x^{-1}\ZZ[x^{-1}]\spanning\{ H_y : y \in W \}$
that is bar invariant is zero. But if $C_w \in \H$ has
$C_w =\overline{C_w}  \in H_w + \sum_{\ell(y) < \ell(w)} x^{-1} \ZZ[x^{-1}] H_y$
then the difference $\underline H_w - C_w$ is such a bar invariant element, so we must have $\underline H_w = C_w$.

Let $h_{yw} \in \ZZ[x^{-1}]$ be the unique polynomials such that $\underline H_w = \sum_{y \in W} h_{yw} H_y$
and define $\mu_{yw}$ to be the coefficient of $x^{-1}$ in $h_{yw}$.  
It turns out that $h_{yw} = 0$ unless $y \leq w$ in the \defn{Bruhat order} on $W$ \cite[\S7.9]{Humphreys},
so it would be equivalent to define $\underline H_w $ as the unique element of $\H$ with 
 \[\underline H_w = \overline{\underline H_w} \in H_w + \sum_{\substack{y \in W \\ y<w}} x^{-1} \ZZ[x^{-1}]  H_y.\]
 This formulation is more common in the literature than \eqref{kl-eq}, but \eqref{kl-eq} will serve as a slightly better prototype for our definitions in the next section.
 
 Finally, set $ \omega_{\mathsf{KL}}(y,w) = \mu_{yw} + \mu_{wy}$ for $y, w \in W$
and define 
\[
\Asc_L(w) := \{ s \in S : \ell(sw) > \ell(w)\}
\quand
\Asc_R(w) := \{ s \in S : \ell(ws) > \ell(w)\}.
\]
The triples $(W, \omega_{\mathsf{KL}}, \Asc_L)$
and $(W,\omega_{\mathsf{KL}},\Asc_R)$ are both $W$-graphs,
whose associated $\H$-modules \eqref{hy-eq} are isomorphic to the left and right regular representations of $\H$
 \cite[Thm. 1.3]{KL}.
The edge weights of these $W$-graphs are actually nonnegative; in fact, one has $h_{yw} \in \NN[x^{-1}]$ \cite[Cor. 1.2]{EliasWilliamson}.
\end{example}

From this point on, we specialize to the case when
$\H = \H(S_n)$,
where  $S_n$ is viewed as a Coxeter group with simple generating set
$S = \{s_1,s_2,\dots,s_{n-1}\}$.
If we set $x=1$ then $\H$ becomes the group ring $\ZZ S_n$ and any $\H$-module 
becomes an $S_n$-representation.
We say that an $\H$-module $\cA$ is a \defn{Gelfand model} if the character of this specialization 
is the multiplicity-free sum of all irreducible characters of $S_n$.
This is equivalent to saying that $\cA$ is isomorphic to the direct sum of all isomorphism classes 
of irreducible $\H$-modules when the scalar ring $\ZZ[x,x^{-1}]$ is extended to the field $\QQ(x)$;
see the discussion in \cite[\S1.2]{MZ}.

\subsection{Gelfand models}\label{gel-sect}

We now review the construction of two Gelfand models for $\H=\H(S_n)$. The bases of these models 
are indexed by the images of two 
natural embeddings  $\cI_n \hookrightarrow \Ifpf_{2n}$ to be denoted $\mmap$ and $\nmap$.
Let $\onefpf$ be the permutation of $\ZZ$ sending $i \mapsto i -(-1)^i$.
Choose $w\in \cI_n$ and let $c_1<c_2<\cdots<c_q$ be the numbers $c \in [n]$ with $w(c)=c$.
Both $\mmap(w)$ and $\nmap(w)$ will be elements of $\Ifpf_{2n}$
sending \[
 i \mapsto w(i)\text{ for $i \in [n]\setminus\{c_1,c_2,\dots,c_q\}$}
 \quand 
 i \mapsto \onefpf(i)\text{ for $i \in [2n]\setminus[n+q]$.}\] 
  The only difference between these two permutations is that we define
\[
\mmap(w) : c_i \leftrightarrow n+i
\quand 
\nmap(w) : c_i \leftrightarrow n+q+1-i
\qquad\text{for all }i \in [q].
\]
 We refer to $\mmap$ as the \defn{ascending embedding}, since it turns each of $n+1,n+2,\dots,n+q-1$ into ascents,
 and to $\nmap$ as the \defn{descending embedding}. Both maps are injective.
 Finally let
\be
\cGm_n := \{ \mmap(w) : w\in\cI_n\}
\quand \cGn_n := \{ \nmap(w) : w\in\cI_n\}.\ee
The set $\cGm_n$ consists of the elements $z \in \Ifpf_{2n}$ with no \defn{visible descents} greater than $n$,
where an integer $i$ is a visible descent of $z$ if $z(i+1) < \min\{i,z(i)\}$ \cite[Prop. 2.9]{MP2022}.

\begin{example}If $n=4$ and $w=(1,3)$ then $\mmap(w) = (1,3)(2,5)(4,6)(7,8)$
 and $\nmap(w) = (1,3)(2,6)(4,5)(7,8)$. 
 Is it useful to draw involutions  in $S_n$ as matchings on $[n]$ with edges corresponding to $2$-cycles.
Our examples are given in terms of such pictures as
\[\ba
\mmap\ \ :\ \
 \begin{tikzpicture}[xscale=0.6,yscale=1,baseline=(a.base)]
  \tikzstyle{every node}=[draw=none,shape=circle,inner sep=1pt];
  \node at (0,0) (a) {$1$};
  \node at (1,0) (b) {$2$};
  \node at (2,0) (c) {$3$};
  \node at (3,0) (d) {$4$};
\draw[-]  (a) to [bend left] (c);
\end{tikzpicture}
\ \ \mapsto
\ \
 \begin{tikzpicture}[xscale=0.6,yscale=1,baseline=(a.base)]
  \tikzstyle{every node}=[draw=none,shape=circle,inner sep=1pt];
  \node at (0,0) (a) {$1$};
  \node at (1,0) (b) {$2$};
  \node at (2,0) (c) {$3$};
  \node at (3,0) (d) {$4$};
  \node at (4,0) (e) {$5$};
  \node at (5,0) (f) {$6$};
  \node at (6,0) (g) {$7$};
  \node at (7,0) (h) {$8$};
\draw[-]  (a) to [bend left] (c);
\draw[-,thick]  (b) to [bend left] (e);
\draw[-,thick]  (d) to [bend left] (f);
\draw[-,thick]  (g) to [bend left] (h);
\end{tikzpicture}\ ,
\\
\nmap\ \ :\ \
 \begin{tikzpicture}[xscale=0.6,yscale=1,baseline=(a.base)]
  \tikzstyle{every node}=[draw=none,shape=circle,inner sep=1pt];
  \node at (0,0) (a) {$1$};
  \node at (1,0) (b) {$2$};
  \node at (2,0) (c) {$3$};
  \node at (3,0) (d) {$4$};
\draw[-]  (a) to [bend left] (c);
\end{tikzpicture}
\ \ \mapsto
\ \
 \begin{tikzpicture}[xscale=0.6,yscale=1,baseline=(a.base)]
  \tikzstyle{every node}=[draw=none,shape=circle,inner sep=1pt];
  \node at (0,0) (a) {$1$};
  \node at (1,0) (b) {$2$};
  \node at (2,0) (c) {$3$};
  \node at (3,0) (d) {$4$};
  \node at (4,0) (e) {$5$};
  \node at (5,0) (f) {$6$};
  \node at (6,0) (g) {$7$};
  \node at (7,0) (h) {$8$};
\draw[-]  (a) to [bend left] (c);
\draw[-,thick]  (b) to [bend left] (f);
\draw[-,thick]  (d) to [bend left] (e);
\draw[-,thick]  (g) to [bend left] (h);
\end{tikzpicture}\ .
\ea
\]
\end{example}

For each fixed-point-free involution $z \in \Ifpf_{2n}$ define 
\be
\ba
\Des^=(z) &:= \{ i\in[n-1] : i+1=z(i)>z(i+1)=i\},
\\
\Asc^=(z) &:= \{ i \in [n-1] : z(i) > n \text{ and } z(i+1)>n\}.
\ea
\ee
We refer to elements of these sets as \defn{weak descents} and \defn{weak ascents}.
\begin{remark*}
An index $i \in [n-1]$ belongs to $ \Des^=(z)$ if and only if $z$ commutes with $s_i=(i,i+1)$.
Note that if the involution $z$ belongs to either $\cGm_n$ or $\cGn_n$, then $i \in [n-1]$ is contained in $\Asc^=(z)$
if and only if $zs_iz \in \{s_{n+1}, s_{n+2},\dots,s_{2n-1}\}$.
Finally, observe that if $i \in \Asc^=(z)$
then we have $z(i) < z(i+1)$ when $z \in \cGm_n$ but $z(i) > z(i+1)$ when $z \in \cGn_n$.
\end{remark*}
For  $z \in \Ifpf_{2n}$  we also define
\be
\ba
\Des^<(z) &:= \{ i \in [n-1] : z(i) > z(i+1)\}\backslash(\Asc^=(z) \sqcup  \Des^=(z)),\\
\Asc^<(z) &:= \{ i \in [n-1] : z(i) < z(i+1)\}\backslash \Asc^=(z).
\ea
\ee
The elements of these sets are \defn{strict descents} and \defn{strict ascents}.
Write $\ell : S_n \to \NN$ for the length function
with 
$\ell(w) = |\Inv(w)|$ where $\Inv(w):=\{ (i,j) \in [n]\times [n]: i<j\text{ and }w(i)>w(j)\}$.

\begin{proposition}\label{embed-des-prop}
If $z \in \I_n$ has $k$ fixed points then $\ell(\mmap(z)) + k(k-1) = \ell(\nmap(z))$ and 
\[
\ba \Des^=(\mmap(z)) &=\Des^=(\nmap(z)),
\\
\Asc^=(\mmap(z)) &=\Asc^=(\nmap(z)),
\ea
\qquad
\ba \Des^<(\mmap(z)) &=\Des^<(\nmap(z)),
\\
\Asc^<(\mmap(z)) &=\Asc^<(\nmap(z)).
\ea
\]
\end{proposition}

\begin{proof}
If $c_1<c_2<\dots<c_k$ are the fixed points of $z \in I_n$ in $[n]$, then
$\Inv(\nmap(z))$ is the disjoint union of $ \Inv(\mmap(z))$ with 
the set of pairs $(i,j) \in [2n]\times [2n]$ with $i < j$ and 
either $i,j \in \{c_1,c_2,\dots,c_k\}$ or $i,j \in \{n+1,n+2,\dots,n+k\}$,
so
$ \ell(\nmap(z)) = \ell(\mmap(z)) + 2\binom{k}{2}$.
Checking the listed equalities between the weak/strict descent/ascent sets is straightforward.
\end{proof}

The next two theorems summarize the type A case of  a few of the main results from \cite{MZ}.

 \begin{theorem}[{\cite[Thms. 1.7 and 1.8]{MZ}}]\label{cM-thm}
Let $\H=\H(S_n)$ and define $\cM$ to be the free $\LL$-module with basis $\{ M_z : z \in \cGm_n\}$.
There is a unique $\H$-module structure on $\cM$ in which
\[
H_s M_z = \begin{cases}
M_{szs}  &\text{if }i \in \Asc^<(z)\\
M_{szs} + (x-x^{-1}) M_z &\text{if }i \in \Des^<(z) \\
-x^{-1} M_z &  \text{if }i \in \Asc^=(z) \\
x M_z& \text{if }i \in \Des^=(z)
\end{cases}
\quad
\text{for $s=s_i \in \{s_1,s_2,\dots,s_{n-1}\}$.
}\]
This $\H$-module has the following additional properties:
\ben
\item[(a)]
$\cM$ is a Gelfand model for $\H$.
\item[(b)] $\cM$ has a unique $\H$-compatible bar operator with $\overline{M_z} =  M_z$ whenever $\Des^<(z) =\varnothing$.
\item[(c)] $\cM$ has a unique  basis $\{ \underline M_z : z \in \cGm_n\}$
with
$\ds \underline M_z = \overline{ \underline M_z} \in M_z + \sum_{\ell(y)<\ell(z)} x^{-1} \ZZ[x^{-1}] M_y$. 
\een
\end{theorem}

Replacing $\cGm_n$ by $\cGn_n$ and 
 $x$ by $-x^{-1}$ changes Theorem~\ref{cM-thm} to the following:

\begin{theorem}[{\cite[Thms. 1.7 and 1.8]{MZ}}]\label{cN-thm}
Let $\H=\H(S_n)$ and define $\cN$ to be the free $\LL$-module with basis $\{ N_z : z \in \cGn_n\}$.  There is a unique $\H$-module structure on $\cN$ in which
\[
H_s N_z = \begin{cases}
N_{szs}  &\text{if }i \in \Asc^<(z)\\
N_{szs} + (x-x^{-1}) N_z &\text{if }i \in \Des^<(z) \\
x N_z &  \text{if }i \in \Asc^=(z) \\
-x^{-1} N_z& \text{if }i \in \Des^=(z)
\end{cases}
\quad
\text{for $s =s_i\in \{s_1,s_2,\dots,s_{n-1}\}$.
 }\]
This $\H$-module has the following additional properties:
\ben
\item[(a)]
$\cN$ is a Gelfand model for $\H$.
\item[(b)] $\cN$ has a unique $\H$-compatible bar operator with $\overline{N_z} =  N_z$ whenever $\Des^<(z) =\varnothing$.
\item[(c)] $\cN$ has a unique  basis $\{ \underline N_z : z \in \cGn_n\}$
with
$\ds \underline N_z = \overline{ \underline N_z} \in N_z + \sum_{\ell(y)<\ell(z)} x^{-1} \ZZ[x^{-1}] N_y$. 
\een
\end{theorem}

\begin{remark}\label{cN-rmk}
The cited results in \cite{MZ} describe an $\H$-module $\cN$ with the same multiplication rule but with $\cGm_n$ rather than $\cGn_n$ as a basis.
Theorem~\ref{cN-thm} still follows directly from \cite[Thms. 1.7 and 1.8]{MZ} 
in view of Proposition~\ref{embed-des-prop}.
Specifically, the module $\cN$ in \cite{MZ} is isomorphic to our version of $\cN$ via the $\ZZ[x,x^{-1}]$-linear map
sending $N_{\mmap(z)} \mapsto N_{\nmap(z)}$ for $z \in \I_n$. 
\end{remark}

The module $\cM$ for $\H=\H(S_n)$ 
was first studied by Adin, Postnikov, and Roichman in \cite{APR}.
The results in \cite{Marberg,MZ,Zhang} give more general constructions of $\cM$ and $\cN$ for classical Weyl groups and affine type A.
Despite the formal similarities between Theorem~\ref{cM-thm} and \ref{cN-thm},
there does not appear to be any simple relationship between
the ``canonical'' bases $\{ \underline M_z\} \subset \cM$ and $\{ \underline N_z\}\subset \cN$.

By mimicking Example~\ref{kl-ex}, one can turn the modules $\cM$ and $\cN$ into $W$-graphs for 
the symmetric group $W=S_n$.
Let $\m_{yz}, \n_{yz} \in \ZZ[x^{-1}]$ 
be the polynomials indexed by $y,z \in \cGm_n$ and $y,z \in \cGn_n$, respectively, such that
$ \underline M_z = \sum_{y \in \cGm_n} \m_{yz} M_y$ and 
$  \underline N_z = \sum_{y \in \cGn_n} \n_{yz} N_y.$
Write
$ \mu^\m_{yz}$
and
$\mu^\n_{yz} $
for the coefficients of $x^{-1}$ in $\m_{yz}$ and $\n_{yz}$.
For $z \in\cGm_n$ define
\be\label{ascm-eq}\ba
\mAsc(z) :=\{ s_i : i \in  \Asc^<(z) \sqcup \Asc^=(z)\}
&= \{ s_i : i \in[n-1]\text{ and }z(i)<z(i+1)\}
\\& = \{ s_i : i \in[n-1]\text{ and }\ell(z) < \ell(s_izs_i)\}.\ea\ee
For $z \in\cGn_n$ define
\be\label{ascn-eq}\ba \nAsc(z) &:= \{ s_i : i \in \Asc^<(z) \sqcup \Des^=(z)\}
\\&\ = \{ s_i : i \in[n-1]\text{ and }z(i)<z(i+1)\text{ or }z(i)=i+1\}
\\&\ = \{ s_i : i \in[n-1]\text{ and }\ell(z) \leq \ell(s_izs_i)\}.\ea
\ee
Then let $\mOmega : \cGm_n \times \cGm_n \to \ZZ$  and $\nOmega : \cGn_n \times \cGn_n \to \ZZ$ be the  maps with 
\be\label{omega-mn-eq}
\mOmega(y,z) := \mu^\m_{yz} + \mu^\m_{zy}
\quand 
\nOmega(y,z): =  \mu^\n_{yz} + \mu^\n_{zy}.
\ee
Unlike the Kazhdan--Lusztig case, these integer coefficients can be negative.

\begin{theorem}[\cite{MZ}] 
The triples $\mGamma := (\cGm_n,  \mOmega, \mAsc)$ and $\nGamma := (\cGm_n,  \nOmega, \nAsc)$ are $S_n$-graphs whose associated Iwahori-Hecke algebra modules are Gelfand models.
\end{theorem}
 
The definitions of $\mOmega$ and $\nOmega$ here are simpler than in \cite[Thm. 1.10]{MZ},
following the conventions in Remark~\ref{reduced-rmk}.
Also, the version of $\nGamma$ here differs from what is in \cite[Thm. 1.10]{MZ} in having $\cGn_n$
as its vertex set. The two formulations are equivalent via Remark~\ref{cN-rmk}.

It is not very clear from our discussion how to actually compute the integers in \eqref{omega-mn-eq}.
We mention some inductive formulas from \cite{MZ} that can be used for this purpose:

\begin{proposition}[{See \cite[Lems. 3.7, 3.15, and 3.27]{MZ}}]
Let $z \in \Ifpf_{2n}$, $i \in \Asc^<(z)$, and $s = s_i$. 
\ben
\item[(a)] If $z \in \cGm_n$ then $ \underline  M_{szs} 
=
(H_s + x^{-1}) \underline  M_z  
-
\sum_{\substack{\ell(y)<\ell(z), \hs s \notin \mAsc(y)}}
\mu^{\m}_{yz}   \underline  M_y.$

\item[(b)]  If $z \in \cGn_n$ then $  
 \underline  N_{szs} 
=
(H_s + x^{-1})\underline  N_z
-
\sum_{\substack{\ell(y)<\ell(z), \hs s \notin \nAsc(y)}}
\mu^{\n}_{yz}   \underline  N_y.$
\een
\end{proposition}

\subsection{Bidirected edges}

As explained in the introduction, it is a natural problem to classify the cells in a given $W$-graph,
where a \defn{cell} means a strongly connected component.
The cells in the left and right Kazhdan--Lusztig $W$-graphs are called the \defn{left and right cells} of $W$.

Two vertices in a $W$-graph $\Gamma=(V,\omega,\tau)$ form a 
\defn{bidirected edge} $v \leftrightarrow w$
if $\omega(v,w)\neq 0 \neq \omega(w,v)$.
The \defn{molecules} of $\Gamma$
are the connected components for the 
undirected graph on $V$ that retains only the bidirected edges.
These subsets do not inherit a $W$-graph structure but are easier to classify than  the cells.
As mentioned in the introduction, we expect that all cells $\mGamma$ and $\nGamma$ are actually molecules
 \cite[Conj. 1.16]{MZ}.
As partial progress towards this conjecture, we will classify the molecules in $\mGamma$ and $\nGamma$ in the next section.

Before this, we need a better understanding of the bidirected edges in $\mGamma$ and $\nGamma$.
 Fix an integer  $1<i<n$ and suppose $v,w \in S_n$ are distinct.
Let $<$ denote the \defn{Bruhat order} on any symmetric group. Below, we will 
often consider this partial order restricted to the set of fixed-point-free involutions $\Ifpf_{2n}$.
 Recall that
one has $w<ws_i$  if and only if $w(i)<w(i+1)$ and $v<w$ if and only if $v^{-1}<w^{-1}$ \cite[Chapter 2]{CCG}.
It follows for $z \in \Ifpf_{2n}$  that $z < s_i zs_i$ if and only if $z(i) < z(i+1)$,
which occurs if and only if $ \ell(s_izs_i)=\ell(z)+2$.

Using just elementary algebra one can show that
$v \leftrightarrow w$ is a bidirected edge in the left (respectively, right)
Kazhdan--Lusztig $S_n$-graph if and only if $v\dK{i}w$ (respectively, $v\K{i}w$) for some $1<i<n$ \cite[Lems. 6.4.1 and 6.4.2]{CCG}.
It is known that the left and right cells in $S_n$
are all molecules \cite[\S6.5]{CCG},
so
 Theorem~\ref{A-thm} implies that
the left (respectively, right) cells in $S_n$ are the subsets on which $\QRSK$ (respectively, $\PRSK$)
is constant \cite[Thm. 1.4]{KL}.

Observe that if $\ell(v) \leq \ell(w)$,
then $v \K{i}w$ (respectively, $v \dK{i}w$) if and only if 
$vs < v < vt = w < ws$ (respectively, $sv < v < tv = w < sw$) for 
 $s=s_{i-1}$ and $t=s_i$ or for $s=s_i$ and $t=s_{i-1}$, that is, for
some choice of $\{s,t\} = \{s_{i-1},s_i\}$.
There is a similar description of the bidirected edges in
 $\mGamma$ and $\nGamma$.
First let $\mleftrightarrow{i}$ be the relation on $\cGm_n $ that has $y \mleftrightarrow{i} z$
 if and only if 
\[ s ys \leq y < tyt = z <szs\quord
s zs \leq z < tzt = y <sys
\quad\text{for some $\{s,t\} =\{s_{i-1},s_i\}$.}\]
Next define $\nleftrightarrow{i}$ to be the relation on $\cGn_n$ that has $y \nleftrightarrow{i} z$
 if and only if 
\[ s ys < y < tyt = z \leq szs \quord s zs < z < tzt = y \leq sys
\quad\text{for some $\{s,t\} =\{s_{i-1},s_i\}$.}\]
We can only have $y \mleftrightarrow{i} z$ or $y \nleftrightarrow{i} z$ if $|\ell(y) - \ell(z)| = 2$.

 \begin{lemma}
 Let $y,z\in\cGm_n$ (respectively, $y,z\in\cGn_n$).
 Then 
$y\leftrightarrow z$ is a bidirected edge in $\mGamma$ (respectively $\nGamma$)
 if and only if $y \mleftrightarrow{i} z$ (respectively  $y \nleftrightarrow{i} z$) for some $1<i<n$.
 \end{lemma}

 \begin{proof}
 We first characterize the bidirected edges in $\mGamma$. Fix $y,z\in\cGm_n$.
Given the formula \eqref{ascm-eq}, the results in our previous paper \cite[Cor. 3.14 and Lem. 3.27]{MZ} assert that $y \leftrightarrow z$ is a bidirected edge in $\mGamma$ if and only if for some $i,j \in [n-1]$
either $s_iys_i \leq y<s_jys_j = z < s_izs_i$ or
$s_izs_i \leq z<s_jzs_j = y < s_iys_i$.
The last two properties can only hold if $|i-j|=1$ so that $s_i$ and $s_j$ do not commute: 
for example, if $s_is_j=s_js_i$ and $s_iys_i \leq y<s_jys_j = z < s_izs_i$
then 
 \[\ell(z) < \ell(s_i zs_i) = \ell(s_i s_j y s_j s_i) = \ell(s_j s_i y s_i s_j)\leq \ell(s_iy s_i) + 2 \leq  \ell(y) + 2 = \ell(z)\]
which is impossible, and similarly for the other case.
Thus $y \leftrightarrow z$ is a bidirected edge in $\mGamma$ precisely when $y \mleftrightarrow{i} z$  for some $1<i<n$.

The argument to handle the bidirected edges in $\nGamma$ is similar. Fix $y,z\in\cGn_n$.
Then it follows from \eqref{ascn-eq} and \cite[Cor. 3.17 and Lem. 3.27]{MZ}   that $y \leftrightarrow z$ is a bidirected edge in $\nGamma$ if and only if for some $i,j \in [n-1]$
either $s_iys_i < y<s_jys_j = z \leq s_izs_i$ or
$s_izs_i < z<s_jzs_j = y \leq s_iys_i$.
The last two properties can again only hold if $|i-j|=1$ so that $s_i$ and $s_j$ do not commute: 
for example, if $s_is_j=s_js_i$ and $s_iys_i < y<s_jys_j = z \leq s_izs_i$
then 
 \[\ell(z) \leq \ell(s_i zs_i) = \ell(s_i s_j y s_j s_i) = \ell(s_j s_i y s_i s_j)\leq \ell(s_iy s_i) + 2 =  \ell(y) < \ell(z)\]
which is impossible, and similarly for the other case.
Thus $y \leftrightarrow z$ is a bidirected edge in $\nGamma$ precisely when $y \nleftrightarrow{i} z$  for some $1<i<n$.
 \end{proof}
 
 \subsection{Gelfand molecules}\label{mleft-sect}
 
As noted above, the molecules of the left and right Kazhdan--Lusztig graphs for $S_n$
(which are the same as the left and right cells)
 are the subsets on which $\QRSK$ and $\PRSK$ are respectively constant. 
The molecules in $\mGamma$ and $\nGamma$
 have a similar description as the fibers of slightly modified versions
 of the maps $\PRBS=\PRSK$ and $\PCBS$ from Sections~\ref{rbs-sect} and \ref{cbs-sect}.
 
If $T$ is a tableau and $\cX$ is a set, then let $T|_\cX$ be 
the tableau formed by omitting all entries of $T$ not in $\cX$.
Recall the definitions of $\cGm_n\subset \Ifpf_{2n}$ and $\cGn_n\subset \Ifpf_{2n}$
from Section~\ref{gel-sect}.

\begin{definition} 
For $y \in \cGm_n$ and $z \in \cGn_n$ define
$
\PMRSK(y) := \PRBS(y)\big|_{[n]} 
$
and
$
\PNRSK(z) := \PCBS(z)\big|_{[n]}
$.
\end{definition}

\begin{example}
We have  
\[
\ba
 \PMRSK(\mmap(2134))= \PMRSK(21563487)&=\ytab{1 &  3 & 4 \\ 2},
\\
 \PMRSK(\mmap(3214))=\PMRSK(35162487)&=\ytab{1 &  2 & 4 \\ 3},
\\
 \PMRSK(\mmap(4231))=\PMRSK(45612387)&=\ytab{1 &  2 & 3 \\ 4},
\ea
\qquad
\ba
 \PNRSK(\nmap(2134))= \PNRSK(21654387)&=\ytab{1 &  2 & 3 \\ 4},
\\
 \PNRSK(\nmap(3214))=\PNRSK(36154287)&=\ytab{1 &  2 & 4 \\ 3},
\\
 \PNRSK(\nmap(4231))=\PNRSK(46513287)&=\ytab{1 &  2 \\ 3 \\ 4}.
\ea\]
\end{example}

Let $T$ be a standard tableau with $n$ boxes and $k$ odd columns. We form a standard tableau $\mIota(T)$ with $2n$ boxes and no odd columns 
from $T$ by the following procedure.
First place the numbers $n+1$, $n+2$, \dots, $n+k$ at the end of the odd columns of $T$ going left to right;
then add the numbers $n+k+1$, $n+k+3$, \dots, $2n-1$ to the first row; and finally add $n+k+2$, $n+k+4$, \dots, $2n$ to the second row to form $\mIota(T)$.
For example, if 
\[ T= \ytab{ 1 & 2 & 3 & 4 \\ 5 & 7  \\ 6 }
\quad\text{then}
\quad
\mIota(T) =  \ytab{ 1 & 2 & 3 & 4 & 11 & 13 \\ 5 & 7 & 9 & 10 & 12 & 14  \\ 6  \\ 8}.
\]
Call an integer $i$ a \defn{transfer point} of an element $z \in \Ifpf_{2n}$
if $i \in [n]$ and $z(i) \in [2n]\setminus[n]$.

\begin{lemma}\label{pm-lem}
If $z \in \cGm_n$ then $\mIota(\PMRSK(z)) = \PRBS(z)$.
\end{lemma}

\begin{proof}
Fix $z \in \cGm_n$ and suppose
$(a_i,b_i)\in[n]\times [n]$ for $i \in [p]$ are the pairs with $a_i < b_i = z(a_i)$, ordered such that
$b_1<b_2<\dots<b_p$. Define $U := \emptyset \fromMRSK (a_1,b_1)\fromMRSK (a_2,b_2) \fromMRSK\cdots \fromMRSK (a_p,b_p)$.
This tableau is partially standard with all even columns, since every insertion $\fromMRSK (a,b)$ with $a<b$ preserves the number of odd columns,
which begins as zero.

Next let $c_1<c_2<\dots<c_k$ denote the transfer points of $z$, so that $z(c_i) = n + i$ for $i \in [k]$.
The \defn{bumping path} of $x$ inserted into $U$ 
is the sequence of positions in $U$ whose entries are changed to form $U \fromRSK x$, together with the new box that is added 
to the tableau. If $x<y$ then the bumping path of
$y$ inserted into $U\fromRSK x$ is strictly to the right of the bumping path of $x$ inserted into $U$.

It follows that the boxes added by successively Schensted inserting $c_1$, $c_2$, \dots, $c_k$ into $U$ occur in a strictly increasing sequence of columns and a weakly decreasing sequence of rows.
Since $U$ starts out with all even columns, each of these $k$ boxes 
creates a new odd column.
Moreover, the result of Schensted inserting $c_i$ 
has no dependence on any of the rows after the box added by Schensted inserting $c_{i-1}$.
The tableau
$T := U \fromRSK c_1 \fromRSK c_2 \fromRSK \dots \fromRSK c_k$
is therefore standard with $k$ odd columns,
and 
placing the numbers $n+1$, $n+2$, \dots, $n+k$ at the end of  these columns going left to right
must give the same result as
\[
U \fromMRSK (c_1,n+1)\fromMRSK (c_2,n+2) \fromMRSK \dots \fromMRSK (c_k,n+k).
\]
To turn this tableau into $\PRBS(z)$ we insert $\fromMRSK(a,a+1)$ for $a=n+k+1,n+k+3,\dots,2n-1$,
but this just adds the numbers $n+k+1,n+k+3,\dots,2n-1$ to the first row and $n+k+2,n+k+4,\dots,2n$ to the second, as each value of $a$ is larger than all other entries in the tableau. From this description, we see that $\PMRSK(z) = \PRBS(z)|_{[n]}= T$ and $\mIota(T) = \PRBS(z)$ as needed.
\end{proof}

\begin{theorem}\label{pmrsk-thm}
The operation $\PMRSK$ defines a bijection from the set of elements of $\cGm_n$ 
with $k$ transfer points
to the set of standard tableaux with $n$ boxes and $k$ odd \textbf{columns}.
\end{theorem}

\begin{proof}
The number of odd columns in $\PMRSK(z)$ for $z \in \cGm_n$ is the number of 
columns in $\PRBS(z)$ with an odd number of entries in $[n]$.
The  operation $\fromRBS (a,b)$ preserves this number
when $n < a \leq b$ and increases it by one when $a \leq n < b$. 
Since we form $\PRBS(z)$ for $z \in \cGm_n$ by first inserting a sequence of cycles $(a,b)$ with $a < b \leq n$
(resulting in a tableau with all even columns and all entries $\leq n$), 
then inserting the cycles $(c_i,n+i)$ where $c_1<\dots<c_k$ are the transfer points $z$,
and finally by inserting a sequence of cycles $(a,b)$ with $n+k < a<b$,
we see that the number of odd columns in $\PMRSK(z)$ is the number of transfer points in $z$.
 Lemma~\ref{pm-lem} shows that $\PMRSK$ is an injective map from $\cGm_n$ to the set of standard tableaux with $n$ boxes (with left inverse $\PRBS^{-1} \circ \mIota$) and therefore a bijection as the domain and codomain both have size $|I_n|$.
\end{proof}

Let $\mLambda(z)$
  be the partition shape  of $\PMRSK(z)$  for $z \in \cGm_n$.
  
 \begin{theorem}\label{mleft-thm}
 Suppose $y,z\in \cGm_n$ are distinct and $1<i<n$.
 Then 
$y \mleftrightarrow{i} z$ if and only if $\PMRSK(y) = D_i(\PMRSK(z))$,
so the molecules in $\mGamma$ are the subsets of $\cGm_n$ on which $\mLambda$ 
is constant.
 \end{theorem}

\begin{proof}
 First suppose $y \mleftrightarrow{i} z$. 
 Without loss of generality we may assume 
that $s ys \leq y < tyt = z <szs$ for some choice of $\{s,t\} =\{s_{i-1},s_i\}$. 

If $s=s_{i-1}$ and $t=s_i$, then it follows that $y(i-1)>y(i)<y(i+1)$ and $s_i(y(i-1))<s_i(y(i+1))$. 
The second inequality implies that $y(i-1) < y(i+1)$, since 
the fact that $y$ is an involution means we cannot have $y(i-1)=i+1$ and $y(i+1)=i$. 
Thus $y(i-1)$ is between $y(i)$ and $y(i+1)$, so
 by Theorem~\ref{direct-prop} we have $y\simRBS{i}z$ and $\PRBS(y) = D_i(\PRBS(z))$.
 Therefore 
\[\PMRSK(y)=\PRBS(y)|_{[n]} = D_i(\PRBS(z))|_{[n]} = D_i(\PRBS(z)|_{[n]})= D_i(\PMRSK(z))\]
by the definition of $D_i$ and the fact that $1<i<n$.

Alternatively, if $s=s_i$ and $t=s_{i-1}$, then  $y(i+1)<y(i)>y(i-1)$ and $s_{i-1}(y(i-1))<s_{i-1}(y(i+1))$. 
The second inequality implies that $y(i-1) < y(i+1)$ since we cannot have $y(i-1)=i$ and $y(i+1)=i-1$,
so $y(i+1)$ is between $y(i)$ and $y(i+1)$.
Then again by Theorem~\ref{direct-prop} we have $y\simRBS{i}z$ and $\PRBS(y) = D_i(\PRBS(z))$,
and it follows   as above that
 $\PMRSK(y)=D_i(\PMRSK(z))$.
We conclude that if  $y \mleftrightarrow{i} z$ then  $\PMRSK(y)=D_i(\PMRSK(z))$.

For the converse statement, suppose $\PMRSK(y)=D_i(\PMRSK(z))$. 
 Then we have
\[\PRBS(y)=\mIota(\PMRSK(y)) = \mIota(D_i(\PMRSK(z))) = D_i(\mIota(\PMRSK(z))) = D_i(\PRBS(z)),\] 
using Lemma~\ref{pm-lem} for the first and last equalities, and the definitions of $D_i$ and $\mIota$ for
the third equality.
The fixed-point-free involution $y\in\cGm_n\subset \Ifpf_{2n}$ cannot preserve the set $ \{i-1,i,i+1\}$.
 Since we also assume  $y\neq z$,    Theorem~\ref{direct-prop}   implies that either
\begin{itemize}
\item  $z=s_{i-1}ys_{i-1}$ and $y(i+1)$ is between $y(i-1)$ and $y(i)$, or
 \item $z=s_{i}ys_{i}$ and  $y(i-1)$ is between $y(i)$ and $y(i+1)$.
 \end{itemize} 
In the first case one has $s_iys_i \leq y < s_{i-1}ys_{i-1} = z <s_izs_i$ if $y(i-1) < y(i)$ or
$s_izs_i \leq z < s_{i-1}zs_{i-1} = y <s_iys_i$ if $y(i) < y(i-1)$.
Likewise,
in the second case one has $s_{i-1}ys_{i-1}\leq y < s_iys_i = z<s_{i-1}zs_{i-1}$ if $y(i) < y(i+1)$
or $s_{i-1}zs_{i-1}\leq z < s_izs_i = y<s_{i-1}ys_{i-1}$ if  $y(i+1) < y(i)$.
Either way  we have   $y \mleftrightarrow{i} z$ as desired.
\end{proof}

Now suppose $T$ is a standard tableau with $n$ boxes and $k$ odd rows.
By a slightly different procedure, we can form a standard tableau $\nIota(T)$ with $2n$ boxes and no odd rows from $T$.
First place the numbers $n+1$, $n+2$, \dots, $n+k$ at the end of the odd rows of $T$ going top to bottom;
then add   $n+k+1$, $n+k+2$, \dots, $2n$ to the first row and define $\nIota(T)$ to be the result.
If
\[ T= \ytab{ 1 & 2 &3 \\ 4 & 5 \\ 6 \\ 7}
\quad\text{then}
\quad
\nIota(T) =  \ytab{ 1 & 2 &3 & 8 & 11 & 12 & 13 & 14 \\ 4 & 5 \\ 6  & 9 \\ 7 & 10},
\]
for example.
We have analogues of Lemma~\ref{pm-lem} and Theorems~\ref{pmrsk-thm} and \ref{mleft-thm}:
\begin{lemma}\label{pn-lem}
If $z \in \cGn_n$ then $\nIota(\PNRSK(z)) = \PCBS(z)$.
\end{lemma}

\begin{proof}
Our argument is similar to the proof of Lemma~\ref{pm-lem}.
Let $(a_i,b_i)$ for $i \in [p]$ be the cycles of $z \in \cGn_n$ with $a_i < b_i = z(a_i) \leq n$
and $b_1<b_2<\dots<b_p$, and define $U := \emptyset \fromNRSK (a_1,b_1)\fromNRSK (a_2,b_2) \fromNRSK\cdots \fromNRSK (a_p,b_p)$.
This tableau is partially standard with all even rows, since every insertion $\fromNRSK (a,b)$ with $a<b$ preserves the number of odd rows,
which begins as zero.

Next let $c_1>c_2>\dots>c_k$ denote the transfer points of $z$, so that $z(c_i) = n + i$ for $i \in [k]$.
If $y<x$ then the bumping path of
$y$ inserted into $U\fromRSK x$ is weakly to the left of the bumping path of $x$ inserted into $U$.
This implies that the boxes added by successively Schensted inserting $c_1$, $c_2$, \dots, $c_k$ into $U$ 
must occur in a strictly increasing sequence of rows and a weakly decreasing sequence of columns.
Since $U$ starts out with all even rows, each of these $k$ boxes creates a new odd row,
and the result of Schensted inserting $c_i$ has no dependence on any of the columns after
the box added by Schensted inserting $c_{i-1}$.
It follows that 
$T := U \fromRSK c_1 \fromRSK c_2 \fromRSK \dots \fromRSK c_k$
is standard with $k$ odd rows,
and that
placing the numbers $n+1$, $n+2$, \dots, $n+k$ at the end of  these rows going top to bottom
must give the same result as
$
U \fromNRSK (c_1,n+1)\fromNRSK (c_2,n+2) \fromNRSK \dots \fromNRSK (c_k,n+k).
$
To turn this tableau into $\PCBS(z)$ we insert $\fromNRSK(a,a+1)$ for $a=n+k+1,n+k+3,\dots,2n-1$,
but this just
adds the numbers $n+k+1,n+k+2,\dots,2n$ to the first row.
From these observations, we see that $\PNRSK(z) = \PCBS(z)|_{[n]}= T$ and $\nIota(T) = \PCBS(z)$ as needed.
\end{proof}

\begin{theorem}\label{pnrsk-thm}
The operation $\PNRSK$ defines a bijection from the set of elements of $\cGn_n$ 
with $k$ transfer points
to the set of standard tableaux with $n$ boxes and $k$ odd \textbf{rows}.
\end{theorem}

\begin{proof}
The number of odd rows in $\PNRSK(z)$ for $z \in \cGn_n$ is the number of 
rows in $\PCBS(z)$ with an odd number of entries in $[n]$.
The  operation $\fromCBS (a,b)$ preserves this number
when $n < a \leq b$ and increases it by one when $a \leq n < b$. 
As in the proof of Theorem~\ref{pmrsk-thm},
the definition of $\PCBS(z)$ combined with this observation  makes it clear that 
 the number of odd columns in $\PNRSK(z)$ is the number of transfer points in $z$.
Finally, Lemma~\ref{pn-lem} shows that $\PNRSK$ is an injective map (with left inverse $\PCBS^{-1} \circ \nIota$)
from $\cGn_n$ to the set of standard tableaux with $n$ boxes,
and therefore a bijection 
since both of these sets have size $|I_n|$.
\end{proof}

 Let $\nLambda(z)$
  be the partition shape  of $\PNRSK(z)$  for $z \in \cGn_n$.
 
 \begin{theorem}\label{nleft-thm}
 Suppose $y,z\in \cGn_n$ are distinct and $1<i<n$.
 Then 
$y \nleftrightarrow{i} z$ if and only if $\PNRSK(y) = D_i(\PNRSK(z))$,
so the molecules in $\nGamma$ are the subsets of $\cGn_n$ on which $\nLambda$ 
is constant.
 \end{theorem}
 
 \begin{proof}
 Define $\ymark(j)$ for $j \in \{i-1,i,i+1\}$ as in \eqref{upsilon-def}.
  First suppose $y \nleftrightarrow{i} z$. 
 Without loss of generality we may assume 
that $s ys < y < tyt = z \leq szs$ for some choice of $\{s,t\} =\{s_{i-1},s_i\}$. 
 
  First assume $s=s_{i-1}$ and $t=s_i$, so that $i \neq y(i-1) > y(i)<y(i+1)$.
  If $z=s_{i-1}zs_{i-1}$ then 
    $(i-1,i+1)$ must be a cycle of $y$ and we must have $y(i) < i-1$,
  which means that  $\ymark(i-1)=i-1$ is between $\ymark(i)=y(i)$ and $\ymark(i+1)=i+1$.
  If $z<s_{i-1}zs_{i-1}$ then we must have $y(i-1)<y(i+1)$; since $y \in \cGn_n\subset \Ifpf_{2n}$ has no fixed points,
  this means that    $\{y(i)<y(i-1)<y(i+1)\} $ is disjoint from $ \{i-1,i,i+1\}$,
  so $\ymark(i-1)=y(i-1)$ is again between $\ymark(i)=y(i)$ and $\ymark(i+1)=y(i+1)$.
 In both cases,  Theorem~\ref{n-cb-thm} implies that $y\simCBS{i}z$ so
 \[ 
 \PNRSK(y) = 
 \PCBS(y)|_{[n]} = D_i(\PCBS(z))|_{[n]}
 = D_i(\PCBS(z)|_{[n]}) = D_i(\PNRSK(z))
 \] 
 by the definition of $D_i$ and the fact that $1<i<n$.
  
Next suppose $s=s_i$ and $t=s_{i-1}$, so that $i \neq y(i+1) < y(i) >y(i-1)$. 
What needs to be checked follows by a symmetric argument.
  If $z=s_izs_i$ then
  $(i-1,i+1)$ must be a cycle of $y$ and we must have $y(i) >i+1$,
  which means that  $\ymark(i+1)=i+1$ is between $\ymark(i-1)=i-1$ and $\ymark(i)=y(i)$.
  If $z<s_izs_i$ then we must have $y(i-1)<y(i+1)$; since $y \in \cGn_n\subset \Ifpf_{2n}$ has no fixed points,
  this means that    $\{y(i-1)<y(i+1)<y(i)\} $ is disjoint from $ \{i-1,i,i+1\}$,
  so $\ymark(i+1)=y(i+1)$ is again between $\ymark(i-1)=y(i-1)$ and $\ymark(i)=y(i)$.
Thus, we deduce by Theorem~\ref{n-cb-thm}   that $y\simCBS{i}z$ and
  as above that $\PNRSK(y) = D_i(\PNRSK(z))$.

For the converse statement, suppose $\PNRSK(y)=D_i(\PNRSK(z))$. 
 Then we have
\[\PCBS(y)=\nIota(\PNRSK(y)) = \nIota(D_i(\PNRSK(z))) = D_i(\nIota(\PNRSK(z))) = D_i(\PCBS(z)),\] 
using Lemma~\ref{pn-lem} for the first and last equalities, and the definitions of $D_i$ and $\nIota$ for
the third equality.
 Since we assume  $y\neq z$,    Theorem~\ref{n-cb-thm}   implies that either
\begin{itemize}
\item[(a)]  $z=s_{i-1}ys_{i-1}$ and $\ymark(i+1)$ is between $\ymark(i-1)$ and $\ymark(i)$, or
 \item[(b)] $z=s_{i}ys_{i}$ and  $\ymark(i-1)$ is between $\ymark(i)$ and $\ymark(i+1)$.
 \end{itemize} 
If $\ymark(j) = y(j)$ for all $j \in \{i-1,i,i+1\}$, then it is straightforward to deduce as in the proof of Theorem~\ref{mleft-thm} that 
$y\nleftrightarrow{i} z$.
If this does not occur,
then in case (a) either
\begin{itemize}
\item $(i-1,i+1)$ is a cycle of $y$ and $i+1 < y(i)$, so $s_iys_i < y < s_{i-1}ys_{i-1} = z =s_izs_i$; or
\item $(i,i+1)$ is a cycle of $y$ and $i+1<y(i-1)$, so $s_izs_i < z < s_{i-1}zs_{i-1} = y =s_iys_i$.
\end{itemize}
Similarly, if $\ymark(j) \neq y(j)$ for some $j \in \{i-1,i,i+1\}$ then in case (b) either
\begin{itemize}
\item $(i-1,i+1)$ is a cycle of $y$ and $ y(i)<i-1$, so $s_{i-1}ys_{i-1} < y < s_{i}ys_{i} = z =s_{i-1}zs_{i-1}$; or
\item $(i-1,i)$ is a cycle of $y$ and $y(i+1)<i-1$, so $s_{i-1}zs_{i-1} < z < s_{i}zs_{i} = y =s_{i-1}ys_{i-1}$.
\end{itemize}
In every case we have   $y \nleftrightarrow{i} z$ as needed.
 \end{proof}

\appendix
\section{Proof of Theorem~\ref{n-cb-thm}}\label{app-sect}

This section contains the proof of Theorem~\ref{n-cb-thm}. 
Unfortunately, the only way we know to prove this result is by a very technical case analysis.
Before commencing this, we need some preliminary notation and a few lemmas.

The \defn{bumping path} resulting from Schensted inserting a number $a$ into a tableau $T$ 
is the sequence of positions $(1,b_1),(2,b_2),\dots,(k, b_k)$ of the entries in $T$ that are changed to form $T \fromRSK a$, together with the new box that is added 
to the tableau.
Let $\BPath_{T\from a}$ denote this sequence. 
Let $b_{T\from a}(j) := b_j$ be the column of the $j$th position in the bumping path,
let $\f_{T\from a} := k$ denote the length of the path (which is also the index of the path's ``final row''), and let $\b_{T\from a}(j)$
be the value inserted into row $j$, so that $\b_{T\from a}(1) = a$.
Observe that 
 \[ b_{T\from a}(1)\ge\cdots\ge b_{T\from a}(k)
 \quand
 \b_{T\from a}(1) <\cdots  < \b_{T\from a}(k).\]
For example, if $a=2$ and 
$ T=\ytab{1&3&9\\4&5&6\\7&8}$
so that
$ 
T \fromRSK a  =\ytab{1&2&9\\3&5&6\\4&8 \\ 7}
$
then we have 
$b_{T\from a}(1) = 2$ and $b_{T\from a}(j) = 1$ for $2\leq j\leq \f_{T\from a}=4$, while 
 \[\b_{T\from a}(1)=2<\b_{T\from a}(2)=3<\b_{T\from a}(3)=4< \b_{T\from a}(4)=7.\]

Recall that a \defn{partially standard tableau} is a semistandard tableau with distinct positive entries.
If $i$ and $j$ appear in a tableau $T$ that has all distinct entries, then we write 
$i\prec_T j$ to indicate that $i$ precedes $j$ in the row reading word of $T$. 
If $i-1$, $i$, and $i+1$ are all entries in a partially standard tableau $T$ then we can evaluate $D_i(T)$ by the formula \eqref{D-def},
and it holds as usual that $D_i(D_i(T)) = T$.

Suppose $v$ and $w$ are sequences of distinct positive integers and $v$ contains $i-1$, $i$, and $i+1$ as letters.
We write $v \dK{i} w$ to mean that either 
\begin{itemize}
\item[(1)] $w=v$ when $i$ is between $i-1$ and $i+1$ in $v$, or
\item[(2)] $w$ is obtained from $v$ by swapping $i$ and $i+1$ if $i-1$ is between these numbers in $v$, or
\item[(3)] $w$ is obtained from $v$ by swapping $i-1$ and $i$ if $i+1$ is between these numbers in $v$.
\end{itemize}
When we evaluate $\PRSK(v)$ and $\PRSK(w)$ by the usual Schensted insertion definition,
it follows from Theorem~\ref{A-thm} that $\PRSK(v) = D_i(\PRSK(w))$ if and only if $v \dK{i} w$.

For the rest of this section,
fix $y\in I_n$ and suppose
 $b_1<b_2<\dots<b_k$ are the distinct numbers in $[n]$ with $a_i := y(b_i) \leq b_i$ so that $y=(a_1,b_1)(a_2,b_2)\cdots(a_k,b_k)$.
For $i \in [k]$ let 
\[T_i :=\emptyset \fromCBS (a_1,b_1) \fromCBS (a_2,b_2) \fromCBS\cdots \fromCBS (a_i,b_i)\quand T_0 := \emptyset.\]
We refer to $T_0,T_1,T_2,\dots,T_k$ as the \defn{partial tableaux} for $y$.

\begin{lemma}\label{n1-lem}
Choose indices $1\leq i < j \leq k$.
Suppose $p<q$ are entries in $T_i$ and there are no entries $r$ in $T_j$ with $p<r<q$. Then we have $p\prec_{T_i}q$ if and only if $p\prec_{T_j}q$.
\end{lemma}

\begin{proof}
We may assume that $j=i+1$, and after standardizing that $q=p+1$.
If $a_j=b_j$
then the column Beissinger insertion operation $\fromCBS(a_j,b_j)$ clearly does not change the relative order 
of $p$ and $q$ in the row reading word of $T_i$. 
On the other hand, if $a_j<b_j$ then it is well-known that the operation $\fromRSK a_j$ also does not change this order,
nor does adding $b_j$ to the end of a column. 
Therefore $p \prec_{T_i} q$ if and only if  $p\prec_{T_j} q$.
\end{proof}

The following lemma compares two bumping paths.
Here, when we say that one path is strictly (respectively, weakly) to the left of the other path,
we mean that in each row where both paths have positions, the unique position in the first path is 
strictly (respectively, weakly) to the left of the unique position in the other path.

\begin{lemma}\label{n2-lem}
Choose an index $1\leq i < k$.
Suppose $a_i<b_i$ and $a_{i+1}<b_{i+1}$. If $a_i<a_{i+1}$ then the bumping path 
$\BPath_{T_{i-1}\from a_i}$ is strictly to the left of $\BPath_{T_{i}\from a_{i+1}}$,
and if $a_i>a_{i+1}$ then the bumping path $\BPath_{T_{i-1}\from a_i}$ is weakly to the left of $\BPath_{T_{i}\from a_{i+1}}$. 
Moreover, we have $\f_{T_{i-1}\from a_i}<\f_{T_{i}\from a_{i+1}}$ when $a_{i+1}<a_i$ and 
$\f_{T_{i-1}\from a_i}\ge\f_{T_{i}\from a_{i+1}}$ when $a_i<b_i<a_{i+1}<b_{i+1}$.
\end{lemma}

However, one can have $\f_{T_{i-1}\from a_i}<\f_{T_{i}\from a_{i+1}}$ when $a_i<a_{i+1}<b_i<b_{i+1}$.

\begin{proof}
The first claim follows from the usual bumping path property for Schensted insertion mentioned earlier.
For the second claim, observe that if $a_{i+1}<a_i$ then   $\b_{T_i\from a_{i+1}}(\f_{T\from a_i})<\b_{T_{i-1}\from a_i}(\f_{T_{i-1}\from a_i})$, so $\f_{T_{i-1}\from a_i}<\f_{T_i\from a_{i+1}}$. If $a_i<b_i<a_{i+1}<b_{i+1}$ then we have
$\f_{T_i \from a_{i+1}} = \f_{U \from a_{i+1}}$ for $U := T_{i-1} \fromRSK a_i \neq T_i = T_{i-1} \fromCBS(a_i,b_i)$,
so $\f_{T_{i-1}\from a_i}\ge\f_{T_{i}\from a_{i+1}}$ again holds by the usual Schensted bumping path properties.
\end{proof}

\begin{lemma}\label{n-lem2}
Let $T$ be a partially standard tableau containing $i-1$, $i$, and $i+1$.
If $a\leq b$ are such that $T\fromCBS (a,b)$ is also partially standard, then $D_i(T)  \fromCBS (a,b) = D_i(T \fromCBS(a,b)).$
\end{lemma}
\begin{proof}
The desired property is clear if $a=b$ so assume $a<b$. 
Since $\row(D_i(T)) a \dK{i} \row(T) a$, it follows
from Theorem~\ref{A-thm} 
 that $D_i(T) \fromRSK a = \PRSK(\row(D_i(T)) a) = D_i(\PRSK(\row(T)a)) = D_i(T \fromRSK a)$.
Suppose the box of this tableau that is not in $T$ is in column $j$.
Then  adding $b$ to the end of column $j+1$ in $D_i(T) \fromRSK a$
gives
$D_i(T)  \fromCBS (a,b)$ by definition.
However, adding $b$ to the end of column $j+1$ in $D_i(T \fromRSK a)$
also apparently gives $D_i(T\fromCBS(a,b))$ since $b \notin \{i-1,i,i+1\}$, so we have $D_i(T)  \fromCBS (a,b) = D_i(T \fromCBS(a,b))$.
\end{proof}

For the rest of this section we fix $1<i<n$, and we define $\ymark(i-1)$, $\ymark(i)$, and $\ymark(i+1)$
by
\[
\ymark(j) := \begin{cases} 
y(j)&\text{if } y(j) \notin \{i-1,i,i+1\} \\
-j &\text{if }y(j)=j \\
j &\text{if } j \neq y(j) \in \{i-1,i,i+1\}
\end{cases}
\]
 as in \eqref{upsilon-def}. 
We divide the proof of Theorem~\ref{n-cb-thm} into three propositions, following this lemma:

\begin{lemma}\label{prop1-lem}
Suppose $a,b\in\{i-1,i,i+1\}$ have $|b-a|=1$ and $\ymark(a)<\ymark(b)$. Then $a\prec_{\PCBS(y)}b$.
\end{lemma}

\begin{proof}
Let $j$ be the index of the first partial tableau $T_j$ for $y$ that contains both $a$ and $b$.
We write $a\prec b$ to mean that $a \prec_{T_j} b$.
Lemma~\ref{n1-lem} implies that $a\prec_{\PCBS(y)}b$ whenever $a\prec b$,
so it is enough to show that $a \prec b$.

If $a$ and $b$ are both fixed points of $y$, then we must have $b=-\ymark(b)<-\ymark(a)=a$, so $b$ is inserted before $a$ when forming $\PCBS(y) $.
In this case, the partial tableau $T_j$ is formed by adding $a$ to the end of the first column of $T_{j-1}$,
so we have $a\prec b$ as claimed.

Suppose only one of $a$ or $b$ is a fixed point of $y$. Then we must have $y(a)=a$ and $y(b) \neq b$ since $\ymark(a) < \ymark(b)$.
If the cycle of $b$ is inserted before $a$ when forming $\PCBS(y) $,  then it follows as in the previous paragraph that $a\prec b$.
If instead $a$  is inserted first, then we have two sub-cases:
\begin{itemize}
\item[(A)] Assume $y(b)<b$. Then $y(b)<a<b=a+1$ since $|a-b|=1$ and $a$ is inserted first. 
Therefore $T_{j-1}$ is formed from $T_{j-2}$ by adding $a$ to the end of the first column,
and we have $T_j = T_{j-1} \fromCBS (y(b),b)$.
Since $a$ is the only entry in its row of $T_{j-1}$, 
it follows that $a$ remains the last entry in the first column of $T_j $,
so we  have  $a\prec b$.

\item[(B)] Assume $b<y(b)$. Then $a$ appears in the first column of $T_{j-1}$
and
 $T_j$ is formed from $T_{j-1} \fromRSK b$ by adding an extra box containing $y(b)$.
This means that $a$ appears in the first column of $T_j$ while $b$ appears in the first row, so we again have  $a\prec b$.

\end{itemize}

Suppose neither $a$ nor $b$ is a fixed point of $y$.
If $y(a) = b$ then $a=\ymark(a)<\ymark(b)= b$ and 
all entries in $T_{j-1}$ are less than $a$.
In this case, the tableau $T_j = T_{j-1} \fromCBS(a,b)$ is formed from  $T_{j-1}$ by adding $a$ and then $b$ to the end of the first row,
so clearly $a\prec b$.

Assume $y(a) \neq b$ and let   $a':=y(a)$ and $b':=y(b)$. 
If $a<b=a+1$ then we must have $a'<b'$ since $\ymark(a) < \ymark(b)$,
so either $a' < a < b < b'$ or $a < b < a'  < b'$ or $a'<b'<a<b$. The first two possibilities 
will put $b$ in the first row of $T_j$ as in case (B) above, and then  $a\prec b$ necessarily holds.

Assume $a'<b'<a<b=a+1$ so that $a$ is the largest entry in $T_{j-1}$.
If Schensted inserting $b'$ into $T_{j-1}$ bumps the corner box containing $a$,
then $a$ will appear  in $T_{j-1} \fromRSK b'$ at the end of the next row  in some column $C$,
and  $T_j$ will be formed from $T_{j-1}\fromRSK b'$ by adding $b$ to the end of column $C+1$, so $a\prec b$. 

If inserting $b'$ into $T_{j-1}$ does not bump the box containing $a$,
then $T_j$ is formed from $T_{j-2} \fromRSK a' \fromRSK b'$ 
by adding $a$ and $b$ to the end of columns $p+1$ and $q+1$, respectively, where $p$ is the column of the box added
when inserting $a'$ into $T_{j-2}$ and $q$ is the column of the box added when inserting $b'$ into $T_{j-2} \fromRSK a'$. 
We have $p<q$ since the bumping path of $a'$ is strictly to the left of the bumping path of $b'$,
so again $a \prec b$

On the other hand, if $b<a=b+1$, then since $\ymark(a)<\ymark(b)$
we must have $a'<b<a<b'$ or $b<a<a'<b'$ or $a'<b'<b<a$. 
 The first two possibilities 
will put $b$ in the first row of $T_j = T_{j-1} \fromCBS (b,b')$.
In these cases $a=b+1$ cannot also be in the first row of $T_j$,
since if $a$ were in the first row of $T_{j-1}$ then $a$, as the smallest number greater than $b$, would be bumped into the next row when $b$ is inserted
to form $T_j$.
Therefore $a\prec b$ holds as needed.

Finally assume $a'<b'<b<a=b+1$. Then, by Lemma~\ref{n2-lem}, the bumping path that results from
Schensted inserting $a'$ into $T_{j-1}$ is weakly to the left of the bumping path that results from 
Schensted inserting $b'$ into $T_{j-2}$, and the former path also ends in a later row.
Therefore $a$ appears in $T_j$ at the end of a column weakly to the left of the column
containing $b$, so we again have $a\prec b$.
\end{proof}

\begin{proposition}\label{prop1}
Assume $\ymark(i)$ is between $\ymark(i-1)$ and $\ymark(i+1)$. Then $\PCBS(y)=D_i(\PCBS(y))$.
\end{proposition}

\begin{proof}
If $ \ymark(i-1)<\ymark(i)<\ymark(i+1)$ or $ \ymark(i+1)<\ymark(i)<\ymark(i-1)$,
then Lemma~\ref{prop1-lem} implies that 
$i-1\prec_{\PCBS(y)}i\prec_{\PCBS(y)}i+1$ or $i+1\prec_{\PCBS(y)}i\prec_{\PCBS(y)}i-1$,
so $\PCBS(y)=D_i(\PCBS(y))$.
\end{proof}

\begin{figure}[h]
\begin{center}

\tikzset{%
  dot/.style={circle, draw, fill=black, inner sep=0pt, minimum width=4pt},
}
\[
\barr{c}
 \text{Case Ia:}
\\
\barr{|c|c|c|}
\hline
\begin{tikzpicture}
\node at (0,0.400) {};
\node[dot,label=below:$a$] (a) at (0,0) {};
\node[dot,label=below:$b$] (b) at (1,0) {};
\node[dot,label=below:$i-1$] (i-1) at (2,0) {};
\node[dot,label=below:$i$] (i) at (2.5,0) {};
\node[dot,label=below:$i+1$] (i+1) at (3,0) {};
\draw[loosely dotted, line width =1pt](a)--(i-1);
\draw (0,0) to [bend left] (2.5,0);
\draw (1,0) to [bend left] (3,0);
\end{tikzpicture}
&
\begin{tikzpicture}
\node at (1,0.400) {};
\node[dot,label=below:$a$] (a') at (1,0) {};
\node[dot,label=below:$i-1$] (i-1') at (2,0) {};
\node[dot,label=below:$i$] (i') at (2.5,0) {};
\node[dot,label=below:$i+1$] (i+1') at (3,0) {};
\node[dot,label=below:$c$] (c') at (4,0) {};
\draw[loosely dotted, line width =1pt](a')--(i-1');
\draw[loosely dotted, line width =1pt](i+1')--(c');
\draw (1,0) to [bend left] (2.5,0);
\draw (3,0) to [bend left] (4,0);
\end{tikzpicture}
&
\begin{tikzpicture}
\node at (1,0.400) {};
\node[dot,label=below:$i-1$] (i-1'') at (1,0) {};
\node[dot,label=below:$i$] (i'') at (1.5,0) {};
\node[dot,label=below:$i+1$] (i+1'') at (2,0) {};
\node[dot,label=below:$b$] (b'') at (3,0) {};
\node[dot,label=below:$c$] (c'') at (4,0) {};
\draw[loosely dotted, line width =1pt](i+1'')--(c'');
\draw (1.5,0) to [bend left] (3,0);
\draw (2,0) to [bend left] (4,0);
\end{tikzpicture}
\\ \hline 
\earr
\\ \\
 \text{Case Ib:}
\\
\tikzset{%
  dot/.style={circle, draw, fill=black, inner sep=0pt, minimum width=4pt},
}
\barr{|c|}
\hline
\begin{tikzpicture}
\node at (0,0.400) {};
\node[dot,label=below:$i-1$] (i-1) at (0,0) {};
\node[dot,label=below:$i$] (i) at (0.5,0) {};
\node[dot,label=below:$i+1$] (i+1) at (1,0) {};
\draw (0.5,0) to [bend left] (1,0);
\end{tikzpicture}
\\ \hline
\earr
\\ \\
 \text{Case Ic:}
\\
\tikzset{%
  dot/.style={circle, draw, fill=black, inner sep=0pt, minimum width=4pt},
}
\barr{|c|c|}
\hline
\begin{tikzpicture}
\node at (0,0.400) {};
\node[dot,label=below:$a$] (a) at (0,0) {};
\node[dot,label=below:$b$] (b) at (1,0) {};
\node[dot,label=below:$c$] (c) at (2,0) {};
\node[dot,label=below:$i-1$] (i-1) at (3,0) {};
\node[dot,label=below:$i$] (i) at (3.5,0) {};
\node[dot,label=below:$i+1$] (i+1) at (4,0) {};
\draw[loosely dotted, line width =1pt](a)--(i-1);
\draw (a) to [bend left] (i-1);
\draw (b) to [bend left] (i);
\draw (c) to [bend left] (i+1);
\end{tikzpicture}
& 
\begin{tikzpicture}
\node at (1,0.400) {};
\node[dot,label=below:$a$] (a') at (1,0) {};
\node[dot,label=below:$b$] (b') at (2,0) {};
\node[dot,label=below:$i-1$] (i-1') at (3,0) {};
\node[dot,label=below:$i$] (i') at (3.5,0) {};
\node[dot,label=below:$i+1$] (i+1') at (4,0) {};
\node[dot,label=below:$c$] (c') at (5,0) {};
\draw[loosely dotted, line width =1pt](a')--(i-1');
\draw[loosely dotted, line width =1pt](i+1')--(c');
\draw (a') to [bend left] (i-1');
\draw (b') to [bend left] (i');
\draw (i+1') to [bend left] (c');
\end{tikzpicture}
\\\hline
\begin{tikzpicture}
\node at (0,0.400) {};
\node[dot,label=below:$a$] (a) at (0,0) {};
\node[dot,label=below:$i-1$] (i-1) at (1,0) {};
\node[dot,label=below:$i$] (i) at (1.5,0) {};
\node[dot,label=below:$i+1$] (i+1) at (2,0) {};
\node[dot,label=below:$b$] (b) at (3,0) {};
\node[dot,label=below:$c$] (c) at (4,0) {};
\draw[loosely dotted, line width =1pt](a)--(i-1);
\draw[loosely dotted, line width =1pt](i+1)--(c);
\draw (a) to [bend left] (i-1);
\draw (i) to [bend left] (b);
\draw (i+1) to [bend left] (c);
\end{tikzpicture}
&
\begin{tikzpicture}
\node at (1,0.400) {};
\node[dot,label=below:$i-1$] (i-1') at (1,0) {};
\node[dot,label=below:$i$] (i') at (1.5,0) {};
\node[dot,label=below:$i+1$] (i+1') at (2,0) {};
\node[dot,label=below:$a$] (a') at (3,0) {};
\node[dot,label=below:$b$] (b') at (4,0) {};
\node[dot,label=below:$c$] (c') at (5,0) {};
\draw[loosely dotted, line width =1pt](i+1')--(c');
\draw (i-1') to [bend left] (a');
\draw (i') to [bend left] (b');
\draw (i+1') to [bend left] (c');
\end{tikzpicture}
\\\hline
\earr
\\ \\
 \text{Case Id:}
\\
\tikzset{%
  dot/.style={circle, draw, fill=black, inner sep=0pt, minimum width=4pt},
}
\barr{|c|}
\hline
\begin{tikzpicture}
\node at (0,0.400) {};
\node[dot,label=below:$i-1$] (i-1) at (0,0) {};
\node[dot,label=below:$i$] (i) at (0.5,0) {};
\node[dot,label=below:$i+1$] (i+1) at (1,0) {};
\node[dot,label=below:$c$] (c) at (2,0) {};
\draw[loosely dotted, line width =1pt](i+1)--(c);
\draw (0,0) to [bend left] (0.5,0);
\draw (1,0) to [bend left] (2,0);
\end{tikzpicture}
\\ \hline
\earr
\earr
\]
\end{center}
\caption{Possibilities for $y$ when $ \ymark(i-1)<\ymark(i)<\ymark(i+1)$.}\label{caseI-fig}
\end{figure}

\begin{figure}[h]
\begin{center}
\[
\barr{c}
\text{Case IIa:}
\\
\tikzset{%
  dot/.style={circle, draw, fill=black, inner sep=0pt, minimum width=4pt},
}
\barr{|c|c|c|}
\hline
\begin{tikzpicture}
\node at (0,0.400) {};
\node[dot,label=below:$a$] (a) at (0,0) {};
\node[dot,label=below:$b$] (b) at (1,0) {};
\node[dot,label=below:$i-1$] (i-1) at (2,0) {};
\node[dot,label=below:$i$] (i) at (2.5,0) {};
\node[dot,label=below:$i+1$] (i+1) at (3,0) {};
\draw[loosely dotted, line width =1pt](a)--(i-1);
\draw (a) to [bend left] (i);
\draw (b) to [bend left] (i-1);
\end{tikzpicture}
&
\begin{tikzpicture}
\node at (5,0.400) {};
\node[dot,label=below:$a$] (a') at (5,0) {};
\node[dot,label=below:$i-1$] (i-1') at (6,0) {};
\node[dot,label=below:$i$] (i') at (6.5,0) {};
\node[dot,label=below:$i+1$] (i+1') at (7,0) {};
\node[dot,label=below:$c$] (c') at (8,0) {};
\draw[loosely dotted, line width =1pt](a')--(i-1');
\draw[loosely dotted, line width =1pt](i+1')--(c');
\draw (a') to [bend left] (i');
\draw (i-1') to [bend left] (c');
\end{tikzpicture}
&
\begin{tikzpicture}
\node at (10,0.400) {};
\node[dot,label=below:$i-1$] (i-1'') at (10,0) {};
\node[dot,label=below:$i$] (i'') at (10.5,0) {};
\node[dot,label=below:$i+1$] (i+1'') at (11,0) {};
\node[dot,label=below:$b$] (b'') at (12,0) {};
\node[dot,label=below:$c$] (c'') at (13,0) {};
\draw[loosely dotted, line width =1pt](i+1'')--(c'');
\draw (i'') to [bend left] (b'');
\draw (i-1'') to [bend left] (c'');
\end{tikzpicture}
\\\hline\earr
\\
\\
\text{Case IIb:}
\\
\tikzset{%
  dot/.style={circle, draw, fill=black, inner sep=0pt, minimum width=4pt},
}
\barr{|c|c|}
\hline
\begin{tikzpicture}
\node at (0,0.400) {};
\node[dot,label=below:$a$] (a) at (0,0) {};
\node[dot,label=below:$b$] (b) at (1,0) {};
\node[dot,label=below:$c$] (c) at (2,0) {};
\node[dot,label=below:$i-1$] (i-1) at (3,0) {};
\node[dot,label=below:$i$] (i) at (3.5,0) {};
\node[dot,label=below:$i+1$] (i+1) at (4,0) {};
\draw[loosely dotted, line width =1pt](a)--(i-1);
\draw (a) to [bend left] (i+1);
\draw (b) to [bend left] (i);
\draw (c) to [bend left] (i-1);
\end{tikzpicture}
&
\begin{tikzpicture}
\node at (8,0.400) {};
\node[dot,label=below:$a$] (a') at (8,0) {};
\node[dot,label=below:$b$] (b') at (9,0) {};
\node[dot,label=below:$i-1$] (i-1') at (10,0) {};
\node[dot,label=below:$i$] (i') at (10.5,0) {};
\node[dot,label=below:$i+1$] (i+1') at (11,0) {};
\node[dot,label=below:$c$] (c') at (12,0) {};
\draw[loosely dotted, line width =1pt](a')--(i-1');
\draw[loosely dotted, line width =1pt](i+1')--(c');
\draw (a') to [bend left] (i+1');
\draw (b') to [bend left] (i');
\draw (c') to [bend right] (i-1');
\end{tikzpicture}
\\ \hline
\begin{tikzpicture}
\node at (0,0.400) {};
\node[dot,label=below:$a$] (a) at (0,0) {};
\node[dot,label=below:$i-1$] (i-1) at (1,0) {};
\node[dot,label=below:$i$] (i) at (1.5,0) {};
\node[dot,label=below:$i+1$] (i+1) at (2,0) {};
\node[dot,label=below:$b$] (b) at (3,0) {};
\node[dot,label=below:$c$] (c) at (4,0) {};
\draw[loosely dotted, line width =1pt](a)--(i-1);
\draw[loosely dotted, line width =1pt](i+1)--(c);
\draw (a) to [bend left] (i+1);
\draw (b) to [bend right] (i);
\draw (c) to [bend right] (i-1);
\end{tikzpicture}
& 
\begin{tikzpicture}
\node at (8,0.400) {};
\node[dot,label=below:$i-1$] (i-1') at (8,0) {};
\node[dot,label=below:$i$] (i') at (8.5,0) {};
\node[dot,label=below:$i+1$] (i+1') at (9,0) {};
\node[dot,label=below:$a$] (a') at (10,0) {};
\node[dot,label=below:$b$] (b') at (11,0) {};
\node[dot,label=below:$c$] (c') at (12,0) {};
\draw[loosely dotted, line width =1pt](i+1')--(c');
\draw (a') to [bend right] (i+1');
\draw (b') to [bend right] (i');
\draw (c') to [bend right] (i-1');
\end{tikzpicture}
\\ \hline
\earr
\\
\\
\barr{c} \text{Case IIc:}
\\
\tikzset{%
  dot/.style={circle, draw, fill=black, inner sep=0pt, minimum width=4pt},
}
\barr{|c|c|}
\hline
\begin{tikzpicture}
\node at (0,0.400) {};
\node[dot,label=below:$a$] (a) at (0,0) {};
\node[dot,label=below:$i-1$] (i-1) at (1,0) {};
\node[dot,label=below:$i$] (i) at (1.5,0) {};
\node[dot,label=below:$i+1$] (i+1) at (2,0) {};
\draw[loosely dotted, line width =1pt](a)--(i-1);
\draw (a) to [bend left] (i-1);
\end{tikzpicture}
&
\begin{tikzpicture}
\node at (4,0.400) {};
\node[dot,label=below:$i-1$] (i-1') at (4,0) {};
\node[dot,label=below:$i$] (i') at (4.5,0) {};
\node[dot,label=below:$i+1$] (i+1') at (5,0) {};
\node[dot,label=below:$c$] (c') at (6,0) {};
\draw[loosely dotted, line width =1pt](c')--(i+1');
\draw (c') to [bend right] (i-1');
\end{tikzpicture}
\\ \hline
\earr
\earr
\quad
\barr{c}
\text{Case IId:}
\\
\tikzset{%
  dot/.style={circle, draw, fill=black, inner sep=0pt, minimum width=4pt},
}
\barr{|c|}
\hline
\begin{tikzpicture}
\node at (0,0.400) {};
\node[dot,label=below:$i-1$] (i-1) at (0,0) {};
\node[dot,label=below:$i$] (i) at (0.5,0) {};
\node[dot,label=below:$i+1$] (i+1) at (1,0) {};
\end{tikzpicture}
\\ \hline
\earr
\earr
\earr
\]
\end{center}
\caption{Possibilities for $y$ when $ \ymark(i+1)<\ymark(i)<\ymark(i-1)$.}\label{caseII-fig}
\end{figure}

\begin{remark}
We can be more specific about the possibilities for $y$ 
when Proposition~\ref{prop1} applies.
If $ \ymark(i-1)<\ymark(i)<\ymark(i+1)$ then one of the following must occur:
\begin{itemize}

\item[(Ia)] $i-1$ is a fixed point of $y$ and $\ymark(i+1)=y(i+1)>y(i)=\ymark(i)$. 

\item[(Ib)] $i-1$ is a fixed point of $y$ and $y(i+1)=i$. 

\item[(Ic)] 
$y(i-1)<y(i)<y(i+1)$ and $\{i-1,i,i+1\}\cap\{y(i-1),y(i),y(i+1)\}=\varnothing$. 

\item[(Id)]
$y(i-1)=i$ and $y(i+1)>i+1$.
\end{itemize}
If instead $ \ymark(i+1)<\ymark(i)<\ymark(i-1)$ then we must be in one of the following cases:
\begin{itemize}

\item[(IIa)] $i+1$ is a fixed point of $y$ and $\ymark(i-1)=y(i-1)>y(i)=\ymark(i)$.

\item[(IIb)]  $y(i+1)<y(i)<y(i-1)$ and $\{i-1,i,i+1\}\cap\{y(i-1),y(i),y(i+1)\}=\varnothing$.

\item[(IIc)] $i$ and $i+1$ are fixed points of $y$ but $y(i-1)\neq i-1$.

\item[(IId)] $i-1$, $i$, $i+1$ are all fixed points of $y$. 
\end{itemize}
Figures~\ref{caseI-fig} and \ref{caseII-fig} show the arc diagrams of $y$ in these cases.
\end{remark}

Our second proposition goes as follows.

\begin{proposition}\label{prop2}
Assume $\ymark(i+1)$ is between $\ymark(i-1)$ and $\ymark(i)$.  Then $\PCBS(y)=D_i(\PCBS(s_{i-1} ys_{i-1}))$.
\end{proposition}

\begin{figure}[h]
\begin{center}
\[
\barr{c}
\text{Cases 1a $\leftrightarrow$ 2a:}
\\
\tikzset{%
  dot/.style={circle, draw, fill=black, inner sep=0pt, minimum width=4pt},
}
\barr{|ccc|}
\hline
\begin{tikzpicture}[baseline=(z.base)]
\node (z) at (0,0.000) {};
\node at (0,0.400) {};
\node[dot,label=below:$a$] (a) at (0,0) {};
\node[dot,label=below:$b$] (b) at (1,0) {};
\node[dot,label=below:$i-1$] (i-1) at (2,0) {};
\node[dot,label=below:$i$] (i) at (2.5,0) {};
\node[dot,label=below:$i+1$] (i+1) at (3,0) {};
\draw[loosely dotted, line width =1pt](a)--(i-1);
\draw (0,0) to [bend left] (3,0);
\draw (1,0) to [bend left] (2.5,0);
\end{tikzpicture}
& \longleftrightarrow &
\begin{tikzpicture}[baseline=(z.base)]
\node (z) at (7,0.000) {};
\node at (7,0.400) {};
\node[dot,label=below:$a$] (a') at (7,0) {};
\node[dot,label=below:$b$] (b') at (8,0) {};
\node[dot,label=below:$i-1$] (i-1') at (9,0) {};
\node[dot,label=below:$i$] (i') at (9.5,0) {};
\node[dot,label=below:$i+1$] (i+1') at (10,0) {};
\draw[loosely dotted, line width =1pt](a')--(i-1');
\draw (7,0) to [bend left] (10,0);
\draw (8,0) to [bend left] (9,0);
\end{tikzpicture}
\\ \hline
\begin{tikzpicture}[baseline=(z.base)]
\node (z) at (0,0.000) {};
\node at (0,0.400) {};
\node[dot,label=below:$a$] (a) at (0,0) {};
\node[dot,label=below:$i-1$] (i-1) at (1,0) {};
\node[dot,label=below:$i$] (i) at (1.5,0) {};
\node[dot,label=below:$i+1$] (i+1) at (2,0) {};
\node[dot,label=below:$c$] (c) at (3,0) {};
\draw[loosely dotted, line width =1pt](a)--(i-1);
\draw[loosely dotted, line width =1pt](i+1)--(c);
\draw (0,0) to [bend left] (2,0);
\draw (1.5,0) to [bend left] (3,0);
\end{tikzpicture}
& \longleftrightarrow &
\begin{tikzpicture}[baseline=(z.base)]
\node (z) at (7,0.000) {};
\node at (7,0.400) {};
\node[dot,label=below:$a$] (a') at (7,0) {};
\node[dot,label=below:$i-1$] (i-1') at (8,0) {};
\node[dot,label=below:$i$] (i') at (8.5,0) {};
\node[dot,label=below:$i+1$] (i+1') at (9,0) {};
\node[dot,label=below:$c$] (c') at (10,0) {};
\draw[loosely dotted, line width =1pt](a')--(i-1');
\draw[loosely dotted, line width =1pt](i+1')--(c');
\draw (7,0) to [bend left] (9,0);
\draw (8,0) to [bend left] (10,0);
\end{tikzpicture}
\\ \hline
\begin{tikzpicture}[baseline=(z.base)]
\node (z) at (0,0.000) {};
\node at (0,0.400) {};
\node[dot,label=below:$i-1$] (i-1) at (0,0) {};
\node[dot,label=below:$i$] (i) at (0.5,0) {};
\node[dot,label=below:$i+1$] (i+1) at (1,0) {};
\node[dot,label=below:$b$] (b) at (2,0) {};
\node[dot,label=below:$c$] (c) at (3,0) {};
\draw[loosely dotted, line width =1pt](i+1)--(c);
\draw (0.5,0) to [bend left] (3,0);
\draw (1,0) to [bend left] (2,0);
\end{tikzpicture}
& \longleftrightarrow &
\begin{tikzpicture}[baseline=(z.base)]
\node (z) at (7,0.000) {};
\node at (7,0.400) {};
\node[dot,label=below:$i-1$] (i-1') at (7,0) {};
\node[dot,label=below:$i$] (i') at (7.5,0) {};
\node[dot,label=below:$i+1$] (i+1') at (8,0) {};
\node[dot,label=below:$b$] (b') at (9,0) {};
\node[dot,label=below:$c$] (c') at (10,0) {};
\draw[loosely dotted, line width =1pt](i+1')--(c');
\draw (7,0) to [bend left] (10,0);
\draw (8,0) to [bend left] (9,0);
\end{tikzpicture}
\\ \hline
\earr
\\
\\
\text{Cases 1b $\leftrightarrow$ 2b:}
\\
\tikzset{%
  dot/.style={circle, draw, fill=black, inner sep=0pt, minimum width=4pt},
}
\barr{|ccc|}
\hline
\begin{tikzpicture}[baseline=(z.base)]
\node (z) at (0,0.000) {};
\node at (0,0.400) {};
\node[dot,label=below:$a$] (a) at (0,0) {};
\node[dot,label=below:$b$] (b) at (1,0) {};
\node[dot,label=below:$c$] (c) at (2,0) {};
\node[dot,label=below:$i-1$] (i-1) at (3,0) {};
\node[dot,label=below:$i$] (i) at (3.5,0) {};
\node[dot,label=below:$i+1$] (i+1) at (4,0) {};
\draw[loosely dotted, line width =1pt](a)--(i-1);
\draw (0,0) to [bend left] (3,0);
\draw (1,0) to [bend left] (4,0);
\draw (2,0) to [bend left] (3.5,0);
\end{tikzpicture}
& \longleftrightarrow &
\begin{tikzpicture}[baseline=(z.base)]
\node (z) at (8,0.000) {};
\node at (8,0.400) {};
\node[dot,label=below:$a$] (a') at (8,0) {};
\node[dot,label=below:$b$] (b') at (9,0) {};
\node[dot,label=below:$c$] (c') at (10,0) {};
\node[dot,label=below:$i-1$] (i-1') at (11,0) {};
\node[dot,label=below:$i$] (i') at (11.5,0) {};
\node[dot,label=below:$i+1$] (i+1') at (12,0) {};
\draw[loosely dotted, line width =1pt](a')--(i-1');
\draw (8,0) to [bend left] (11.5,0);
\draw (9,0) to [bend left] (12,0);
\draw (10,0) to [bend left] (11,0);
\end{tikzpicture}
\\ \hline
\begin{tikzpicture}[baseline=(z.base)]
\node (z) at (0,0.000) {};
\node at (0,0.400) {};
\node[dot,label=below:$a$] (a) at (0,0) {};
\node[dot,label=below:$b$] (b) at (1,0) {};
\node[dot,label=below:$i-1$] (i-1) at (2,0) {};
\node[dot,label=below:$i$] (i) at (2.5,0) {};
\node[dot,label=below:$i+1$] (i+1) at (3,0) {};
\node[dot,label=below:$c$] (c) at (4,0) {};
\draw[loosely dotted, line width =1pt](a)--(i-1);
\draw[loosely dotted, line width =1pt](i+1)--(c);
\draw (0,0) to [bend left] (2,0);
\draw (1,0) to [bend left] (3,0);
\draw (2.5,0) to [bend left] (4,0);
\end{tikzpicture}
& \longleftrightarrow &
\begin{tikzpicture}[baseline=(z.base)]
\node (z) at (8,0.000) {};
\node at (8,0.400) {};
\node[dot,label=below:$a$] (a') at (8,0) {};
\node[dot,label=below:$b$] (b') at (9,0) {};
\node[dot,label=below:$i-1$] (i-1') at (10,0) {};
\node[dot,label=below:$i$] (i') at (10.5,0) {};
\node[dot,label=below:$i+1$] (i+1') at (11,0) {};
\node[dot,label=below:$c$] (c') at (12,0) {};
\draw[loosely dotted, line width =1pt](a')--(i-1');
\draw[loosely dotted, line width =1pt](i+1')--(c');
\draw (8,0) to [bend left] (10.5,0);
\draw (9,0) to [bend left] (11,0);
\draw (10,0) to [bend left] (12,0);
\end{tikzpicture}
\\ \hline
\begin{tikzpicture}[baseline=(z.base)]
\node (z) at (0,0.000) {};
\node at (0,0.400) {};
\node[dot,label=below:$a$] (a) at (0,0) {};
\node[dot,label=below:$i-1$] (i-1) at (1,0) {};
\node[dot,label=below:$i$] (i) at (1.5,0) {};
\node[dot,label=below:$i+1$] (i+1) at (2,0) {};
\node[dot,label=below:$b$] (b) at (3,0) {};
\node[dot,label=below:$c$] (c) at (4,0) {};
\draw[loosely dotted, line width =1pt](a)--(i-1);
\draw[loosely dotted, line width =1pt](i+1)--(c);
\draw (0,0) to [bend left] (1,0);
\draw (2,0) to [bend left] (3,0);
\draw (1.5,0) to [bend left] (4,0);
\end{tikzpicture}
& \longleftrightarrow &
\begin{tikzpicture}[baseline=(z.base)]
\node (z) at (8,0.000) {};
\node at (8,0.400) {};
\node[dot,label=below:$a$] (a') at (8,0) {};
\node[dot,label=below:$i-1$] (i-1') at (9,0) {};
\node[dot,label=below:$i$] (i') at (9.5,0) {};
\node[dot,label=below:$i+1$] (i+1') at (10,0) {};
\node[dot,label=below:$b$] (b') at (11,0) {};
\node[dot,label=below:$c$] (c') at (12,0) {};
\draw[loosely dotted, line width =1pt](a')--(i-1');
\draw[loosely dotted, line width =1pt](i+1')--(c');
\draw (8,0) to [bend left] (9.5,0);
\draw (10,0) to [bend left] (11,0);
\draw (9,0) to [bend left] (12,0);
\end{tikzpicture}
\\ \hline
\begin{tikzpicture}[baseline=(z.base)]
\node (z) at (0,0.000) {};
\node at (0,0.400) {};
\node[dot,label=below:$i-1$] (i-1) at (0,0) {};
\node[dot,label=below:$i$] (i) at (0.5,0) {};
\node[dot,label=below:$i+1$] (i+1) at (1,0) {};
\node[dot,label=below:$a$] (a) at (2,0) {};
\node[dot,label=below:$b$] (b) at (3,0) {};
\node[dot,label=below:$c$] (c) at (4,0) {};
\draw[loosely dotted, line width =1pt](i+1)--(c);
\draw (0,0) to [bend left] (2,0);
\draw (0.5,0) to [bend left] (4,0);
\draw (1,0) to [bend left] (3,0);
\end{tikzpicture}
& \longleftrightarrow &
\begin{tikzpicture}[baseline=(z.base)]
\node (z) at (8,0.000) {};
\node at (8,0.400) {};
\node[dot,label=below:$i-1$] (i-1') at (8,0) {};
\node[dot,label=below:$i$] (i') at (8.5,0) {};
\node[dot,label=below:$i+1$] (i+1') at (9,0) {};
\node[dot,label=below:$a$] (a') at (10,0) {};
\node[dot,label=below:$b$] (b') at (11,0) {};
\node[dot,label=below:$c$] (c') at (12,0) {};
\draw[loosely dotted, line width =1pt](i+1')--(c');
\draw (8,0) to [bend left] (12,0);
\draw (8.5,0) to [bend left] (10,0);
\draw (9,0) to [bend left] (11,0);
\end{tikzpicture}
\\ \hline
\earr
\\
\\
\text{Cases 1c $\leftrightarrow$ 2c:}
\\
\tikzset{%
  dot/.style={circle, draw, fill=black, inner sep=0pt, minimum width=4pt},
}
\barr{|ccc|}
\hline
\begin{tikzpicture}[baseline=(z.base)]
\node (z) at (0,0.000) {};
\node at (0,0.400) {};
\node[dot,label=below:$i-1$] (i-1) at (0,0) {};
\node[dot,label=below:$i$] (i) at (0.5,0) {};
\node[dot,label=below:$i+1$] (i+1) at (1,0) {};
\node[dot,label=below:$c$] (c) at (2,0) {};
\draw[loosely dotted, line width =1pt](i+1)--(c);
\draw (0,0) to [bend left] (1,0);
\draw (0.5,0) to [bend left] (2,0);
\end{tikzpicture}
& \longleftrightarrow &
\begin{tikzpicture}[baseline=(z.base)]
\node (z) at (6,0.000) {};
\node at (6,0.400) {};
\node[dot,label=below:$i-1$] (i-1') at (6,0) {};
\node[dot,label=below:$i$] (i') at (6.5,0) {};
\node[dot,label=below:$i+1$] (i+1') at (7,0) {};
\node[dot,label=below:$c$] (c') at (8,0) {};
\draw[loosely dotted, line width =1pt](i+1')--(c');
\draw (6,0) to [bend left] (8,0);
\draw (6.5,0) to [bend left] (7,0);
\end{tikzpicture}
\\ \hline
\earr
\earr
\]
\end{center}
\caption{Possibilities for $y$ when  $\ymark(i+1)$ is between $\ymark(i-1)$ and $\ymark(i)$.} \label{case12-fig}
\end{figure}

\begin{proof}
In this proof let $z := s_{i-1} ys_{i-1}$. We wish to show that $\PCBS(y) = D_i(\PCBS(z))$.
By hypothesis either
$\ymark(i-1)<\ymark(i+1)<\ymark(i)$ or $ \ymark(i)<\ymark(i+1)<\ymark(i-1)$.
If the first set of inequalities holds then we must be in one of the following cases:
\begin{itemize}
\item[(1a)] $i-1$ is a fixed point of $y$ and $\ymark(i+1)=y(i+1)<y(i)=\ymark(i)$. 

\item[(1b)]   $y(i-1)<y(i+1)<y(i)$ and $\{i-1,i,i+1\}\cap\{y(i-1),y(i),y(i+1)\}=\varnothing$. 

\item[(1c)]  $y(i-1)=i+1$ and $y(i)>i+1$. 
\end{itemize}
If instead $\ymark(i)<\ymark(i+1)<\ymark(i-1)$ then we must be in one of these cases:
\begin{itemize}
\item[(2a)] $i$ is a fixed point of $y$ and $\ymark(i+1)=y(i+1)<y(i-1)=\ymark(i-1)$.

\item[(2b)]  $y(i)<y(i+1)<y(i-1)$, $\{i-1,i,i+1\}\cap\{y(i-1),y(i),y(i+1)\}=\varnothing$.

\item[(2c)] $y(i)=i+1$ and $y(i-1)>i+1$.

\end{itemize}
Figure~\ref{case12-fig}  shows the possibilities for the arc diagrams of $y$ in these cases.

Notice that cases (2a), (2b), and (2c) are obtained from cases (1a), (1b), and (1c) by interchanging $y$ and $z$. 
Since $D_i$ is an involution, it suffices to show that $\PCBS(y) = D_i(\PCBS(z))$ just in cases (1a), (1b), and (1c).
We consider each of these it turn.
Throughout this proof,  $a<b<c$ are the integers with $\{a,b,c\} = \{ y(i-1),y(i),y(i+1)\}$.
\begin{itemize}
\item[(1a)] Assume $i-1$ is a fixed point of $y$ and $\ymark(i+1)=y(i+1)<y(i)=\ymark(i)$. 
Then we are in one of the three possible subcases indicated in Figure~\ref{case12-fig}:
\begin{itemize}
\item[i.] We could have $y=\cdots(i-1,i-1)(b,i)(a,i+1)\cdots$ and $z=\cdots(b,i-1)(i,i)(a,i+1)\cdots$ where $a<b<c=i-1$. 
In this case, denote the partial tableau for $y$ obtained just before inserting $(i-1,i-1)$ by $T$. This is also the partial tableau for $z$ obtained just before inserting $(b,i-1)$.
Then let $T_y$, $T'_y$, and $T''_y$ be the partial tableaux for $y$ obtained just after inserting $(i-1,i-1)$,  $(b,i)$, and $(a,i+1)$, respectively. 
Define $T_z$, $T'_z$, and $T''_z$
relative to $z$ similarly.\footnote{
That is, 
let
$T_z$, $T'_z$, and $T''_z$ be the partial tableaux for $z$ obtained just after inserting $(b,i-1)$,  $(i,i)$, and $(a,i+1)$, respectively.
In the next few arguments we will often make similar definitions of $T_y$, $T'_y$, $T''_y$ 
and $T_z$, $T'_z$, $T''_z$:  the first three objects will be partial tableaux for $y$ obtained after inserting certain cycles of $y$,
while the last three objects will be partial tableaux for $z$ obtained after inserting  corresponding (but possibly different) cycles of $z$.
}
Now consider the insertions $T_y\fromCBS(b,i)$ and $T\fromCBS(b,i-1)$. If inserting $b$ leads to a new box with entry $d<i-1$, then $i-1$ has no effect on the insertion $T_y\fromCBS(b,i)$, so we have $T'_z=s_{i-1}(T'_y)$. If $d=i-1$, then the last bump of this insertion 
will involve some $d'<i-1$ bumping $i-1$ to the next row.
For $T\fromCBS(b,i-1)$ this will bump $d'$ into a new row and then put $i-1$ in the last box of the second column. As a result, we have $T'_z=s_{i-1}(T'_y)=D_{i}(T'_y)$. 
After inserting $(a,i+1)$, we have $i-1\prec_{T''_y} i+1\prec_{T''_y} i$ by Lemma~\ref{n2-lem}. Also, we have $T''_z=s_{i-1}(T''_y) = D_i(T''_y)$. By Lemmas~\ref{n1-lem} and \ref{n-lem2}, 
it follows that
\[i-1\prec_{\PCBS(y)} i+1\prec_{\PCBS(y)} i\quand \PCBS(z)=D_i(\PCBS(y)).\]

\item[ii.] We could have 
\[y=\cdots(i-1,i-1)(a,i+1)\cdots(i,c)\cdots \quand z=\cdots(i,i)(a,i+1)\cdots(i-1,c)\cdots\] where $a<b=i-1<i+1<c$. 
In this case, denote the partial tableau for $y$  obtained just before inserting $(i-1,i-1)$ by $T$. This is also the partial tableau for $z$ obtained just before inserting $(i,i)$.
Then let $T_y$ be the partial tableau for $y$ obtained just after inserting $(i-1,i-1)$, 
let $T'_y$ be the partial tableau for $y$ obtained just after inserting $(a,i+1)$, 
let $T''_y$ be the partial tableau for $y$ obtained just before inserting $(i,c)$, 
and let $T'''_y$ be the partial tableau for $y$ obtained just after inserting $(i,c)$. Define $T_z$, $T'_z$, $T''_z$, and $T'''_z$ relative to $z$ similarly.
Then we have $i-1\prec_{T'_y} i+1$ and by Lemma~\ref{n1-lem}, $i-1\prec_{T''_y} i+1$. Also, $T''_z$ is just $T''_y$ after replacing $i-1$ by $i$. After the insertion $\fromCBS(i,c)$, we get $i-1\prec_{T'''_y} i+1\prec_{T'''_y} i$. By an argument similar to case 1a(i), we have $T'''_z=s_{i-1}(T'''_y)=D_{i}(T'''_y)$, so Lemmas~\ref{n1-lem} and \ref{n-lem2} imply that
\[i-1\prec_{\PCBS(y)} i+1\prec_{\PCBS(y)} i\quand \PCBS(z)=D_i(\PCBS(y)).\]

\item[iii.] Finally, we could have $y=\cdots(i-1,i-1)\cdots(i+1,b)\cdots(i,c)\cdots$ and $z=\cdots(i,i)\cdots(i+1,b)\cdots(i-1,c)\cdots$ where $a=i-1<i+1<b<c$.
In this case, the proof is the same as in subcase ii after replacing any references to $(a,i+1)$ by $(i+1,b)$.
\end{itemize}

\item[(1b)] Assume $y(i-1)=a<y(i+1)=b<y(i)=c$ and $\{i-1,i,i+1\}\cap\{a,b,c\}=\varnothing$.  Then we are in one of 
 the four possible subcases indicated in Figure~\ref{case12-fig}:
\begin{itemize}
\item[i.] We could have $y=\cdots(a,i-1)(c,i)(b,i+1)\cdots$ and $  z=\cdots(c,i-1)(a,i)(b,i+1)\cdots$ where $a<b<c<i-1$,
so that $y$ and $z$ have arc diagrams
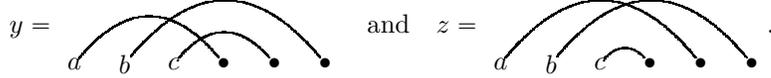
\[
  y=\arcstart
{
*{a}  \arc{1.5}{rrr}  & *{b}$\cdots$ \arc{2}{rrrr} &*{c}  \arc{1}{rr}     & *{\bullet} & *{\bullet} & *{\bullet}   
}
\arcstop
\quand   z=\arcstart
{
*{a}   \arc{2}{rrrr}  & *{b} \arc{2}{rrrr} & *{c}  \arc{0.5}{r}     & *{\bullet} & *{\bullet} & *{\bullet}   
}
\arcstop.
\]
We address this case by considering the insertions into each row recursively. 
Denote the partial tableau for $y$ obtained just before inserting $(a,i-1)$ by $T$. This is also the partial tableau for $z$ obtained just before inserting $(c,i-1)$.
Let
$T_y$, $T'_y$, and  $T''_y$  be the partial tableaux for $y$ obtained just after inserting $(a,i-1)$, $(c,i)$, and $(b,i+1)$. 
Define $T_z$, $T'_z$, and $T''_z$ relative to $z$ similarly. 
Then write 
\[
\ba \b^y_a(k) &:= \b_{T\from a}(k),
\\
\b^y_c(k) &:= \b_{T_y\from c}(k),
\\
\b^y_b(k) &:= \b_{T'_y\from b}(k).
\ea\]
 Define $\b^z_a(k)$, $\b^z_b(k)$ and $\b^z_c(k)$ similarly. Also, write 
\[
\f_y(a):=\f_{T\from a},
\quad 
\f_y(c):=\f_{T_y\from c},
\quand
\f_y(b):=\f_{T'_y\from b}.
\] 
Then define $\f_z(a)$, $\f_z(b)$, and $\f_z(c)$ relative to $z$ similarly.
If \[j<m:=\min\{\f_y(a),\f_y(b),\f_y(c),\f_z(a),\f_z(b),\f_z(c)\},\] then by Lemma~\ref{n2-lem} we have $\b^y_{a}(j+1)< \b^y_{b}(j+1)< \b^y_{c}(j+1)$. Since 
\begin{align*}
\b^y_a(j)\b^y_c(j)\b^y_b(j)&\K{}\b^y_c(j)\b^y_a(j)\b^y_b(j)\\
&=\b^z_c(j)\b^z_a(j)\b^z_b(j),
\end{align*}
we see that the $j$th rows of $T''_y$ and $T''_z$ are the same. 
Then without loss of generality, we can assume $m=1$. By Lemma~\ref{n2-lem}, we have 
\[\f_y(c)<\f_y(b)\quand \f_z(c)<\f_z(a).\] 
First assume $\f_y(a)>\f_y(c)=1<\f_y(b)$. Consider the first rows of $T''_y$ and $T''_z$.
These rows differ only in the entries of their last box, which are $i$ for $y$ and $i-1$ for $z$. 
For the rows above, according to Lemma~\ref{n2-lem}, we have $\BPath_{T\leftarrow a}\le \BPath_{T'_y\leftarrow b}$, so  
\[i-1\prec_{T''_y} i+1\prec_{T''_y} i \quand T''_z=D_i(T''_y).\]  The desired result now follows from Lemma~\ref{n-lem2}.
Next assume $\f_y(a)=1\le\f_y(c)<\f_y(b)$ and then $\f_z(c)=\f_y(a)=1$. 
Then the insertions into the first row of $T_y\leftarrow c$, $T'_y\leftarrow b$, and $T''_y$ appear as
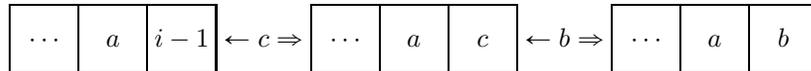
\[
\ytabb{\cdots&a&i-1}\leftarrow c\Rightarrow\ytabb{\cdots&a&c}\leftarrow b
\Rightarrow\ytabb{\cdots&a&b}
\]
so  $\b^y_{c}(2)=i-1$ and $\b^y_{b}(2)=c$. On the other hand,
the insertions into the first row of  $T_z\leftarrow a$, $T'_z\leftarrow b$, and $T''_z$ appear as
\[
\ytabb{\cdots&c&i-1}\leftarrow a
\Rightarrow\ytabb{\cdots&a&i-1}\leftarrow b
\Rightarrow\ytabb{\cdots&a&b}
\]
so $\b^z_{a}(2)=c$ and $\b^z_{b}(2)=i-1$. 
Now, consider the insertion into the second row. Since there are no numbers greater than $i-1$ in the tableau $T$, we see that $\f_y(c)=2$. Then for the second row of $T'_y\leftarrow b$, we have 
\[
\ytabb{\cdots&i-1&i}\leftarrow c
\Rightarrow\ytabb{\cdots&c&i}
\]
and $\b^y_{b}(3)=i-1$. For the second row of $T'_z\leftarrow b$, we have
\[
\ytabb{\cdots&c&i}\leftarrow i-1
\Rightarrow\ytabb{\cdots&c&i-1}
\]
and $\b^z_{b}(3)=i$. So we have $\f_y(b)=3$ and the last two rows for $T''_y$ and $T''_z$ are 
\[
\ytabb{\cdots&\cdots&c&i\\\cdots&i-1&i+1}
\quand
\ytabb{\cdots&\cdots&c&i-1\\\cdots&i&i+1}.
\]
Since $T''_z=D_i(T''_y)$ the desired result now follows from Lemmas~\ref{n1-lem} and \ref{n-lem2}.

\item[ii.] We could have $  y=\cdots(a,i-1)(b,i+1)\cdots(i,c)\cdots$ and $z=\cdots(a,i)(b,i+1)\cdots(i-1,c)\cdots$ where $a<b<i-1<i+1<c$,
so that $y$ and $z$ have arc diagrams
\[
  y=\arcstart
{
*{a}   \arc{1}{rr}  & *{b}\arc{1.5}{rrr} & *{\bullet}   & *{\bullet}\arc{1}{rr}  & *{\bullet} & *{c}   
}
\arcstop
\quand  z=\arcstart
{
*{a}   \arc{1.5}{rrr}  & *{b}\arc{1.5}{rrr} & *{\bullet} \arc{1.5}{rrr}  & *{\bullet}  & *{\bullet} & *{c}   
}
\arcstop.
\]
In this case, denote the partial tableaux for $y$ after inserting $(a,i-1)$ and $(b,i+1)$ 
by $T_y$ and $T'_y$.
Then let $T''_y$ be the partial tableau for $y$ obtained just before inserting $(i,c)$, and
let $T'''_y$ be the partial tableau for $y$ obtained just after inserting $(i,c)$. Define $T_z$, $T'_z$, $T''_z$, and $T'''_z$ relative to $z$ similarly. 
Then write 
\[\ba
\b^y_a(k)&:=\b_{T\from a}(k),
\\
\b^y_b(k)&:=\b_{T_y\from b}(k),
\\
\b^y_c(k) &:= \b_{T''_y\from c}(k).
\ea\] Define $\b^z_a(k)$, $\b^z_b(k)$ and $\b^z_c(k)$ similarly. Also, write 
\[ \f_y(a):= \f_{T\from a}, 
\quad \f_y(b):=\f_{T_y\from b},
\quand
\f_y(c):=\f_{T''_y\from c}.\] 
Define $\f_z(a)$, $\f_z(b)$, and $\f_z(c)$ relative to $z$ similarly.
Then  $T'_z$ is just $T'_y$ after replacing $i-1$ by $i$. Also, by considering the result of $\fromCBS(a,i-1)$ and $\fromCBS(b,i+1)$ for $y$, 
by Lemma~\ref{n2-lem}, we have $\b^y_a(j)<\b^y_b(j)$ for all indices $j \leq \min \{\f_y(a), \f_y(b)\}$. 
Assume $\f_y(a)=l$. Then either
\be\label{ccass1} \f_y(b)=k<l\quand i-1\prec_{T'_y} \b^y_b(\f_y(b))\prec_{T'_y} i+1,\ee
or we have 
 \be\label{ccass2}\f_y(b)\ge l \quad\text{and therefore}\quad \b^y_a(l)<\b^y_b(l).\ee
When \eqref{ccass2} occurs, in the tableau $T'_y$, we either have (I) $\b^y_a(l)$ and $i-1$ are the last two elements in the same row, or (II) $i-1$ is in some row $l'<l$ and both $\b^y_a(l)$ and $i-1$ are the last element in their rows. In subcase (I), we must have $\b^y_b(l)< i-1$, as otherwise $\b^y_a(l)$, $i-1$, $\b^y_b(l)$ and $i+1$ will be in the same row after two insertions, which leads to $\b^y_b(l)=i$ which is impossible. Since there are no elements greater than $i-1$ before inserting $(a,i-1)$, we have 
\[ \b^y_b(\f_y(b))=\b^y_{b}(l+1)=i-1\quand i-1\prec_{T'_y} i+1.\]
In subcase (II), we must have $\f_y(b)=l$ because 
 $\b^y_a(l)$ is the last element in its row in $T'_y$ and $\b^y_a(l)<\b^y_b(l)$.
Since $i-1$ is in the column directly after the one containing $\b^y_a(l)$ in $T'_y$, it must be in the same column in $T''_y$ because $\b^y_a(j)<\b^y_b(j)$ for all $j\leq l$. Also, $\b^y_b(l)$ is in the column directly after the one containing $\b^y_a(l)$ in $T''_y$, so 
\[\b^y_b(\f_y(b))=\b^y_{b}(l+1)=i-1\quand i-1\prec_{T'_y} i+1.
\]
Thus, for both cases \eqref{ccass1} and \eqref{ccass2}, Lemma~\ref{n1-lem} implies that $i-1\prec_{T''_y} i+1$. 
By the above discussion, we see that $i-1$ and $i+1$ appear in different rows of $T''_y$. Thus after the insertion $\fromCBS(i,c)$, 
we have $i-1\prec_{T'''_y} i+1\prec_{T'''_y} i$. For $z$, after inserting $(i-1,c)$ we similarly 
 have \[i\prec_{T'''_z} i+1\prec_{T'''_z} i-1 \quand T'''_z=D_i(T'''_y).\] The desired result now follows from Lemmas~\ref{n1-lem} and \ref{n-lem2}.

\item[iii.] We could have $y=\cdots(a,i-1)\cdots(i+1,b)\cdots(i,c)\cdots$ and $z=\cdots(a,i)\cdots(i+1,b)\cdots(i-1,c)\cdots$ where $a<i-1<i+1<b<c$,
so that 
\[
  y=\arcstart
{
*{a}   \arc{0.5}{r}  & *{\bullet} & *{\bullet}  \arc{1.5}{rrr}     & *{\bullet}\arc{0.5}{r}  & *{b} & *{c}   
}
\arcstop
\quand   z=\arcstart
{
*{a}   \arc{1}{rr}  & *{\bullet}\arc{2}{rrrr}      & *{\bullet}  & *{\bullet}\arc{0.5}{r}  & *{b} & *{c}    
}
\arcstop.
\]
Write $T_y$ for the partial tableau for $y$ obtained after inserting $(a,i-1)$.
Let $T'_y$ be the partial tableau for $y$ obtained after inserting $(i+1,b)$, 
let $T''_y$ be the partial tableau for $y$ obtained before inserting $(i,c)$,
 and 
 let $T'''_y$ be the partial tableau for $y$ obtained after inserting $(i,c)$. 
 Define $T_z$, $T'_z$, $T''_z$, $T'''_z$ relative to $z$ similarly. Then $T'_z$ is just $T'_y$ after replacing $i-1$ by $i$.
Also, the insertion $(i+1,b)$ makes $i-1\prec_{T'_y} i+1$. Thus, by Lemma~\ref{n1-lem}, 
we have $i-1\prec_{T''_y} i+1$. Using the same argument as in case ii, 
after the insertion $\fromCBS(i,c)$, we get $i-1\prec_{T'''_y} i+1\prec_{T'''_y} i$. 
For $z$, similarly, we get \[i\prec_{T'''_z} i+1\prec_{T'''_z} i-1 \quand T'''_z=D_i(T'''_y)\] 
so the desired result follows from Lemmas~\ref{n1-lem} and \ref{n-lem2}.

\item[iv.] Finally, we could have $y=\cdots(i-1,a)\cdots(i+1,b)\cdots(i,c)\cdots$ and $  z=\cdots(i,a)\cdots(i+1,b)\cdots(i-1,c)\cdots$ where $i+1<a<b<c$,
so that 
\[
  y=\arcstart
{
*{\bullet}   \arc{1.5}{rrr}  & *{\bullet} \arc{2}{rrrr} & *{\bullet}  \arc{1}{rr}     & *{a} & *{b} & *{c}   
}
\arcstop
\quand  z=\arcstart
{
*{\bullet}   \arc{2.5}{rrrrr}  & *{\bullet} \arc{1}{rr} & *{\bullet}  \arc{1}{rr}     & *{a} & *{b} & *{c}   
}
\arcstop.
\]
In this case, the proof is the same as case iii after replacing $(a,i-1)$ by $(i-1,a)$.
\end{itemize}

\item[(1c)] Assume $y(i-1)=i+1$ and $y(i)>i+1$, so that $a=i-1 < b=i+1<c=y(i)$. Then
$y=\cdots(i-1,i+1)\cdots(i,c)\cdots$ and $  z=\cdots(i,i+1)\cdots(i-1,c)\cdots$ have arc diagrams
\[
  y=\arcstart
{
*{\bullet}\arc{1}{rr}  & *{\bullet}\arc{1}{rr} & *{\bullet} & *{c}  
}
\arcstop
\quand  z=\arcstart
{
*{\bullet}\arc{1.5}{rrr}  & *{\bullet}\arc{0.5}{r} & *{\bullet} & *{c}  
}
\arcstop.
\]
In this case, let $T_y$ denote the partial tableau for $y$ obtained after inserting $(i-1,i+1)$, 
let $T'_y$ be the partial tableau for $y$ obtained before inserting $(i,c)$, 
and let $T''_y$ be the partial tableau for $y$ obtained after inserting $(i,c)$. 
Define $T_z$, $T'_z$, and $T''_z$ relative to $z$ similarly. Then $T_z$ is just $T_y$ after replacing $i-1$ by $i$ and $i-1\prec_{T_y} i+1$. 
Thus, by Lemma~\ref{n1-lem}, we have $i-1\prec_{T'_y} i+1$. 
Using the same argument as in case 1b(ii), we get $i-1\prec_{T''_y} i+1\prec_{T''_y} i$. 
For $z$, similarly, we get $i\prec_{T''_z} i+1\prec_{T''_z} i-1$ and $ T''_z=D_i(T''_y),$
so the desired result follows from Lemmas~\ref{n1-lem} and \ref{n-lem2}.
\end{itemize}
In all of these cases we deduce that $\PCBS(y) = D_i(\PCBS(z))$ as needed.
\end{proof}

Our next result is the final piece of the proof of Theorem~\ref{n-cb-thm}.

\begin{proposition}\label{prop3}
Assume $\ymark(i-1)$ is between $\ymark(i)$ and $\ymark(i+1)$.  Then $\PCBS(y)=D_i(\PCBS(s_{i} ys_{i}))$.
\end{proposition}

\begin{figure}[h]
\begin{center}
\[
\barr{c}
\text{Cases 3a $\leftrightarrow$ 4a:}
\\
\tikzset{%
  dot/.style={circle, draw, fill=black, inner sep=0pt, minimum width=4pt},
}
\barr{|ccc|}
\hline
\begin{tikzpicture}[baseline=(z.base)]
\node (z) at (0,0.000) {};
\node at (0,0.400) {};
\node[dot,label=below:$a$] (a) at (0,0) {};
\node[dot,label=below:$i-1$] (i-1) at (1,0) {};
\node[dot,label=below:$i$] (i) at (1.5,0) {};
\node[dot,label=below:$i+1$] (i+1) at (2,0) {};
\draw[loosely dotted, line width =1pt](a)--(i-1);
\draw (a) to [bend left] (i+1);
\end{tikzpicture}
& \longleftrightarrow &
\begin{tikzpicture}[baseline=(z.base)]
\node (z) at (7,0.000) {};
\node at (7,0.400) {};
\node[dot,label=below:$a$] (a') at (6,0) {};
\node[dot,label=below:$i-1$] (i-1') at (7,0) {};
\node[dot,label=below:$i$] (i') at (7.5,0) {};
\node[dot,label=below:$i+1$] (i+1') at (8,0) {};
\draw[loosely dotted, line width =1pt](a')--(i-1');
\draw (a') to [bend left] (i');
\end{tikzpicture}
\\ \hline
\begin{tikzpicture}[baseline=(z.base)]
\node (z) at (0,0.000) {};
\node at (0,0.400) {};
\node[dot,label=below:$i-1$] (i-1) at (0,0) {};
\node[dot,label=below:$i$] (i) at (0.5,0) {};
\node[dot,label=below:$i+1$] (i+1) at (1,0) {};
\node[dot,label=below:$c$] (c) at (2,0) {};
\draw[loosely dotted, line width =1pt](i+1)--(c);
\draw (c) to [bend right] (i+1);
\end{tikzpicture}
& \longleftrightarrow &
\begin{tikzpicture}[baseline=(z.base)]
\node (z) at (7,0.000) {};
\node at (7,0.400) {};
\node[dot,label=below:$i-1$] (i-1') at (6,0) {};
\node[dot,label=below:$i$] (i') at (6.5,0) {};
\node[dot,label=below:$i+1$] (i+1') at (7,0) {};
\node[dot,label=below:$c$] (c') at (8,0) {};
\draw[loosely dotted, line width =1pt](i+1')--(c');
\draw (c') to [bend right] (i');
\end{tikzpicture}
\\\hline
\earr
\\
\\
\text{Cases 3b $\leftrightarrow$ 4b:}
\\
\tikzset{%
  dot/.style={circle, draw, fill=black, inner sep=0pt, minimum width=4pt},
}
\barr{|ccc|}
\hline
\begin{tikzpicture}[baseline=(z.base)]
\node (z) at (0,0.000) {};
\node at (0,0.400) {};
\node[dot,label=below:$i-1$] (i-1) at (0,0) {};
\node[dot,label=below:$i$] (i) at (0.5,0) {};
\node[dot,label=below:$i+1$] (i+1) at (1,0) {};
\draw (i-1) to [bend left] (i+1);
\end{tikzpicture}
& \longleftrightarrow &
\begin{tikzpicture}[baseline=(z.base)]
\node (z) at (5,0.000) {};
\node at (5,0.400) {};
\node[dot,label=below:$i-1$] (i-1') at (5,0) {};
\node[dot,label=below:$i$] (i') at (5.5,0) {};
\node[dot,label=below:$i+1$] (i+1') at (6,0) {};
\draw (i-1') to [bend left] (i');
\end{tikzpicture}
\\ \hline
\earr
\\
\\
\text{Cases 3c $\leftrightarrow$ 4c:}
\\
\tikzset{%
  dot/.style={circle, draw, fill=black, inner sep=0pt, minimum width=4pt},
}
\barr{|ccc|}
\hline
\begin{tikzpicture}[baseline=(z.base)]
\node (z) at (0,0.000) {};
\node at (0,0.400) {};
\node[dot,label=below:$a$] (a) at (0,0) {};
\node[dot,label=below:$b$] (b) at (1,0) {};
\node[dot,label=below:$i-1$] (i-1) at (2,0) {};
\node[dot,label=below:$i$] (i) at (2.5,0) {};
\node[dot,label=below:$i+1$] (i+1) at (3,0) {};
\draw[loosely dotted, line width =1pt](a)--(i-1);
\draw (a) to [bend left] (i-1);
\draw (b) to [bend left] (i+1);
\end{tikzpicture}
& \longleftrightarrow &
\begin{tikzpicture}[baseline=(z.base)]
\node (z) at (7,0.000) {};
\node at (7,0.400) {};
\node[dot,label=below:$a$] (a') at (7,0) {};
\node[dot,label=below:$b$] (b') at (8,0) {};
\node[dot,label=below:$i-1$] (i-1') at (9,0) {};
\node[dot,label=below:$i$] (i') at (9.5,0) {};
\node[dot,label=below:$i+1$] (i+1') at (10,0) {};
\draw[loosely dotted, line width =1pt](a')--(i-1');
\draw (a') to [bend left] (i-1');
\draw (b') to [bend left] (i');
\end{tikzpicture}
\\ \hline
\begin{tikzpicture}[baseline=(z.base)]
\node (z) at (0,0.000) {};
\node at (0,0.400) {};
\node[dot,label=below:$a$] (a) at (0,0) {};
\node[dot,label=below:$i-1$] (i-1) at (1,0) {};
\node[dot,label=below:$i$] (i) at (1.5,0) {};
\node[dot,label=below:$i+1$] (i+1) at (2,0) {};
\node[dot,label=below:$c$] (c) at (3,0) {};
\draw[loosely dotted, line width =1pt](a)--(i-1);
\draw[loosely dotted, line width =1pt](i+1)--(c);
\draw (a) to [bend left] (i-1);
\draw (c) to [bend right] (i+1);
\end{tikzpicture}
& \longleftrightarrow &
\begin{tikzpicture}[baseline=(z.base)]
\node (z) at (7,0.000) {};
\node at (7,0.400) {};
\node[dot,label=below:$a$] (a') at (7,0) {};
\node[dot,label=below:$i-1$] (i-1') at (8,0) {};
\node[dot,label=below:$i$] (i') at (8.5,0) {};
\node[dot,label=below:$i+1$] (i+1') at (9,0) {};
\node[dot,label=below:$c$] (c') at (10,0) {};
\draw[loosely dotted, line width =1pt](a')--(i-1');
\draw[loosely dotted, line width =1pt](i+1')--(c');
\draw (a') to [bend left] (i-1');
\draw (c') to [bend right] (i');
\end{tikzpicture}
\\ \hline
\begin{tikzpicture}[baseline=(z.base)]
\node (z) at (0,0.000) {};
\node at (0,0.400) {};
\node[dot,label=below:$i-1$] (i-1) at (0,0) {};
\node[dot,label=below:$i$] (i) at (0.5,0) {};
\node[dot,label=below:$i+1$] (i+1) at (1,0) {};
\node[dot,label=below:$b$] (b) at (2,0) {};
\node[dot,label=below:$c$] (c) at (3,0) {};
\draw[loosely dotted, line width =1pt](i+1)--(c);
\draw (b) to [bend right] (i-1);
\draw (c) to [bend right] (i+1);
\end{tikzpicture}
& \longleftrightarrow &
\begin{tikzpicture}[baseline=(z.base)]
\node (z) at (7,0.000) {};
\node at (7,0.400) {};
\node[dot,label=below:$i-1$] (i-1') at (7,0) {};
\node[dot,label=below:$i$] (i') at (7.5,0) {};
\node[dot,label=below:$i+1$] (i+1') at (8,0) {};
\node[dot,label=below:$b$] (b') at (9,0) {};
\node[dot,label=below:$c$] (c') at (10,0) {};
\draw[loosely dotted, line width =1pt](i+1')--(c');
\draw (b') to [bend right] (i-1');
\draw (c') to [bend right] (i');
\end{tikzpicture}
\\ \hline
\earr
\\
\\
\text{Cases 3d $\leftrightarrow$ 4d:}
\\
\tikzset{%
  dot/.style={circle, draw, fill=black, inner sep=0pt, minimum width=4pt},
}
\barr{|ccc|}
\hline
\begin{tikzpicture}[baseline=(z.base)]
\node (z) at (0,0.000) {};
\node at (0,0.400) {};
\node[dot,label=below:$a$] (a) at (0,0) {};
\node[dot,label=below:$b$] (b) at (1,0) {};
\node[dot,label=below:$c$] (c) at (2,0) {};
\node[dot,label=below:$i-1$] (i-1) at (3,0) {};
\node[dot,label=below:$i$] (i) at (3.5,0) {};
\node[dot,label=below:$i+1$] (i+1) at (4,0) {};
\draw[loosely dotted, line width =1pt](a)--(i-1);
\draw (0,0) to [bend left] (3.5,0);
\draw (1,0) to [bend left] (3,0);
\draw (2,0) to [bend left] (4,0);
\end{tikzpicture}
& \longleftrightarrow &
\begin{tikzpicture}[baseline=(z.base)]
\node (z) at (8,0.000) {};
\node at (8,0.400) {};
\node[dot,label=below:$a$] (a') at (8,0) {};
\node[dot,label=below:$b$] (b') at (9,0) {};
\node[dot,label=below:$c$] (c') at (10,0) {};
\node[dot,label=below:$i-1$] (i-1') at (11,0) {};
\node[dot,label=below:$i$] (i') at (11.5,0) {};
\node[dot,label=below:$i+1$] (i+1') at (12,0) {};
\draw[loosely dotted, line width =1pt](a')--(i-1');
\draw (8,0) to [bend left] (12,0);
\draw (9,0) to [bend left] (11,0);
\draw (10,0) to [bend left] (11.5,0);
\end{tikzpicture}
\\ \hline
\begin{tikzpicture}[baseline=(z.base)]
\node (z) at (0,0.000) {};
\node at (0,0.400) {};
\node[dot,label=below:$a$] (a) at (0,0) {};
\node[dot,label=below:$b$] (b) at (1,0) {};
\node[dot,label=below:$i-1$] (i-1) at (2,0) {};
\node[dot,label=below:$i$] (i) at (2.5,0) {};
\node[dot,label=below:$i+1$] (i+1) at (3,0) {};
\node[dot,label=below:$c$] (c) at (4,0) {};
\draw[loosely dotted, line width =1pt](a)--(i-1);
\draw[loosely dotted, line width =1pt](i+1)--(c);
\draw (a) to [bend left] (i);
\draw (b) to [bend left] (i-1);
\draw (c) to [bend right] (i+1);
\end{tikzpicture}
& \longleftrightarrow &
\begin{tikzpicture}[baseline=(z.base)]
\node (z) at (8,0.000) {};
\node at (8,0.400) {};
\node[dot,label=below:$a$] (a') at (8,0) {};
\node[dot,label=below:$b$] (b') at (9,0) {};
\node[dot,label=below:$i-1$] (i-1') at (10,0) {};
\node[dot,label=below:$i$] (i') at (10.5,0) {};
\node[dot,label=below:$i+1$] (i+1') at (11,0) {};
\node[dot,label=below:$c$] (c') at (12,0) {};
\draw[loosely dotted, line width =1pt](a')--(i-1');
\draw[loosely dotted, line width =1pt](i+1')--(c');
\draw (a') to [bend left] (i+1');
\draw (b') to [bend left] (i-1');
\draw (c') to [bend right] (i');
\end{tikzpicture}
\\ \hline
\begin{tikzpicture}[baseline=(z.base)]
\node (z) at (0,0.000) {};
\node at (0,0.400) {};
\node[dot,label=below:$a$] (a) at (0,0) {};
\node[dot,label=below:$i-1$] (i-1) at (1,0) {};
\node[dot,label=below:$i$] (i) at (1.5,0) {};
\node[dot,label=below:$i+1$] (i+1) at (2,0) {};
\node[dot,label=below:$b$] (b) at (3,0) {};
\node[dot,label=below:$c$] (c) at (4,0) {};
\draw[loosely dotted, line width =1pt](a)--(i-1);
\draw[loosely dotted, line width =1pt](i+1)--(c);
\draw (a) to [bend left] (i);
\draw (b) to [bend right] (i-1);
\draw (c) to [bend right] (i+1);
\end{tikzpicture}
& \longleftrightarrow &
\begin{tikzpicture}[baseline=(z.base)]
\node (z) at (8,0.000) {};
\node at (8,0.400) {};
\node[dot,label=below:$a$] (a') at (8,0) {};
\node[dot,label=below:$i-1$] (i-1') at (9,0) {};
\node[dot,label=below:$i$] (i') at (9.5,0) {};
\node[dot,label=below:$i+1$] (i+1') at (10,0) {};
\node[dot,label=below:$b$] (b') at (11,0) {};
\node[dot,label=below:$c$] (c') at (12,0) {};
\draw[loosely dotted, line width =1pt](a')--(i-1');
\draw[loosely dotted, line width =1pt](i+1')--(c');
\draw (a') to [bend left] (i+1');
\draw (b') to [bend right] (i-1');
\draw (c') to [bend right] (i');
\end{tikzpicture}
\\ \hline
\begin{tikzpicture}[baseline=(z.base)]
\node (z) at (0,0.000) {};
\node at (0,0.400) {};
\node[dot,label=below:$i-1$] (i-1) at (0,0) {};
\node[dot,label=below:$i$] (i) at (0.5,0) {};
\node[dot,label=below:$i+1$] (i+1) at (1,0) {};
\node[dot,label=below:$a$] (a) at (2,0) {};
\node[dot,label=below:$b$] (b) at (3,0) {};
\node[dot,label=below:$c$] (c) at (4,0) {};
\draw[loosely dotted, line width =1pt](i+1)--(c);
\draw (a) to [bend right] (i);
\draw (b) to [bend right] (i-1);
\draw (c) to [bend right] (i+1);
\end{tikzpicture}
& \longleftrightarrow &
\begin{tikzpicture}[baseline=(z.base)]
\node (z) at (8,0.000) {};
\node at (8,0.400) {};
\node[dot,label=below:$i-1$] (i-1') at (8,0) {};
\node[dot,label=below:$i$] (i') at (8.5,0) {};
\node[dot,label=below:$i+1$] (i+1') at (9,0) {};
\node[dot,label=below:$a$] (a') at (10,0) {};
\node[dot,label=below:$b$] (b') at (11,0) {};
\node[dot,label=below:$c$] (c') at (12,0) {};
\draw[loosely dotted, line width =1pt](i+1')--(c');
\draw (a') to [bend right] (i+1');
\draw (b') to [bend right] (i-1');
\draw (c') to [bend right] (i');
\end{tikzpicture}
\\ \hline
\earr
\\
\\
\text{Cases 3e $\leftrightarrow$ 4e:}
\\
\tikzset{%
  dot/.style={circle, draw, fill=black, inner sep=0pt, minimum width=4pt},
}
\barr{|ccc|}
\hline
\begin{tikzpicture}[baseline=(z.base)]
\node (z) at (0,0.000) {};
\node at (0,0.400) {};
\node[dot,label=below:$a$] (a) at (0,0) {};
\node[dot,label=below:$i-1$] (i-1) at (1,0) {};
\node[dot,label=below:$i$] (i) at (1.5,0) {};
\node[dot,label=below:$i+1$] (i+1) at (2,0) {};
\draw[loosely dotted, line width =1pt](a)--(i-1);
\draw (a) to [bend left] (i);
\draw (i-1) to [bend left] (i+1);
\end{tikzpicture}
& \longleftrightarrow &
\begin{tikzpicture}[baseline=(z.base)]
\node (z) at (7,0.000) {};
\node at (7,0.400) {};
\node[dot,label=below:$a$] (a') at (6,0) {};
\node[dot,label=below:$i-1$] (i-1') at (7,0) {};
\node[dot,label=below:$i$] (i') at (7.5,0) {};
\node[dot,label=below:$i+1$] (i+1') at (8,0) {};
\draw[loosely dotted, line width =1pt](a')--(i-1');
\draw (a') to [bend left] (i+1');
\draw (i-1') to [bend left] (i');
\end{tikzpicture}
\\ \hline
\earr
\earr
\]
\end{center}
\caption{Possibilities for $y$ when  $\ymark(i-1)$ is between $\ymark(i)$ and $\ymark(i+1)$.}\label{cases34-fig}
\end{figure}

\begin{proof}
 In this proof let $z := s_{i} ys_{i}$.
We wish to show that $\PCBS(y) = D_i(\PCBS(z))$.
By hypothesis either
$\ymark(i)<\ymark(i-1)<\ymark(i+1)$ or $ \ymark(i+1)<\ymark(i-1)<\ymark(i).$
If $\ymark(i)<\ymark(i-1)<\ymark(i+1)$ then we must be in one of the following cases:
\begin{itemize}
\item[(3a)] $i-1$ and $i$ are fixed points of $y$ and $\ymark(i+1)=y(i+1)\neq i+1$. 

\item[(3b)] $i$ is a fixed point of $y$ and $y(i-1)=i+1$. 

\item[(3c)] $i$ is the only fixed point of $y$ in $\{i-1,i,i+1\}$ and $\ymark(i-1)=y(i-1)<y(i+1)=\ymark(i+1)$. 

\item[(3d)] 
$y(i)<y(i-1)<y(i+1)$ and $\{i-1,i,i+1\}\cap\{y(i-1),y(i),y(i+1)\}=\varnothing$. 

\item[(3e)] 
$y(i-1)=i+1$ and  $y(i)<i-1$. 
\end{itemize}
If instead $\ymark(i+1)<\ymark(i-1)<\ymark(i)$ then we must be in one of these cases:
\begin{itemize}
\item[(4a)] $i-1$ and $i+1$ are fixed points of $y$ and $\ymark(i)=y(i)\neq i$.

\item[(4b)] $i+1$ is a fixed point of $y$ and $y(i-1)=i$.

\item[(4c)] $i+1$ is the only fixed point of $y$ in $\{i-1,i,i+1\}$ and $\ymark(i-1)=y(i-1)<y(i)=\ymark(i)$.

\item[(4d)] 
$y(i+1)<y(i-1)<y(i)$ and $\{i-1,i,i+1\}\cap\{y(i-1),y(i),y(i+1)\}=\varnothing$.

\item[(4e)] 
$y(i-1)=i$ and $y(i+1)<i-1$.
\end{itemize}
Figure~\ref{cases34-fig}  shows the possibilities for the arc diagrams of $y$ in these cases.

Cases (3a)-(3e) are obtained from cases (4a)-(4e) by interchanging $y$ and $z$. 
Thus, it suffices to show that $\PCBS(y) = D_i(\PCBS(z))$ just in cases (3a)-(3e).
We consider each of these it turn.
As in the proof of the previous proposition,
throughout this argument we define
  $a<b<c$ to be the integers with $\{a,b,c\} = \{ y(i-1),y(i),y(i+1)\}$.
\begin{itemize}
\item[(3a)] Assume $i-1$ and $i$ are fixed points of $y$ and $\ymark(i+1)=y(i+1)\notin\{i-1,i, i+1\}$.
Then the first partial tableau for $y$ containing all three of $i-1$, $i$, and $i+1$ will 
have $i-1$ and $i$ in the first column, with $i-1$ in a row above $i$. This tableau will also have $i+1$ either in the first row (when $i+1 < y(i+1)$)
or in column two or greater (when $y(i+1) < i+1$). Either way, it follows from Lemma~\ref{n1-lem} that
\[i\prec_{\PCBS(y)}i-1\prec_{\PCBS(y)}i+1.\] 
The first partial tableau for $z$ containing all three of $i-1$, $i$, and $i+1$ is obtained from the first partial tableau for $y$ containing all three of $i-1$, $i$, and $i+1$ by interchanging $i$ and $i+1$. So we have $\PCBS(z)=s_i(\PCBS(y))=D_i(\PCBS(y))$ by Lemma~\ref{n-lem2}.

\item[(3b)] Assume $i$ is a fixed point of $y$ and $y(i-1)=i+1$, so that $a=i-1<b=i<c=i+1$.
In this case, $y=\cdots(i,i)(i-1,i+1)\cdots$ and $z=\cdots(i-1,i)(i+1,i+1)\cdots$. 
Let $T_y$ and $T'_y$ denote the partial tableaux for $y$ obtained after inserting $(i,i)$ and $(i-1,i+1)$. 
Define $T_z$ and $T'_z$ relative to $z$ similarly. 
If $i=2$, we have $T'_y=\ytab{1&3\\2}$ and $T'_z=\ytab{1&2\\3}$. Hence
\[ i\prec_{T'_y} i-1\prec_{T'_y} i+1\quand T'_z=s_i(T'_y)=D_i(T'_y).\]
If $i>2$, then $i$ is not in row 1 of $T_y$ but $i-1$ and $i+1$ are in row 1 of $T'_y$. 
Therefore the two insertions $T_y\fromCBS(i,i)$ and $T'_y\fromCBS(i-1,i+1)$ will not interact, and nor will
$T_z\fromCBS(i-1,i)$ and $T'_z\fromCBS(i+1,i+1)$. Thus, we have $i\prec_{T'_y} i-1\prec_{T'_y} i+1$ and 
$T'_z=s_i(T'_y)=D_i(T'_y)$, so by Lemmas~\ref{n1-lem} and \ref{n-lem2}
\[i\prec_{\PCBS(y)}i-1\prec_{\PCBS(y)}i+1\quand \PCBS(z)=D_i(\PCBS(y)).\]

\item[(3c)] Assume $i$ is the only fixed point of $y$ in $\{i-1,i,i+1\}$ and $\ymark(i-1)=y(i-1)<y(i+1)=\ymark(i+1)$. 
Then we are in one of the three possible subcases indicated in Figure~\ref{cases34-fig}:
\begin{itemize}
\item[i.] We could have $y=\cdots(a,i-1)(i,i)(b,i+1)\cdots$ and $z=\cdots(a,i-1)(b,i)(i+1,i+1)\cdots$ where $a<b<i-1<i=c$. 
In this case, denote the partial tableaux for $y$ after inserting $(a,i-1)$, $(i,i)$, and $(b,i+1)$
by $T$,  $T_y$, and $T'_y$. Note that $T$ is also the partial tableau for $z$ obtained just before inserting $(a,i-1)$. Define $T_z$ and $T'_z$ relative to $z$ similarly.
By an argument similar to case 1a(i), we get $T'_z=s_i(T'_y)=D_i(T'_y)$, so
 Lemma~\ref{n2-lem} implies
 $i\prec_{T''_y} i-1\prec_{T''_y} i+1$. 
According to Lemma~\ref{n-lem2}, we then have 
\[i\prec_{\PCBS(y)} i-1\prec_{\PCBS(y)} i+1\quand \PCBS(z)=D_i(\PCBS(y)).\]

\item[ii.] We could have 
\[y=\cdots(a,i-1)(i,i)\cdots(i+1,c)\cdots\quand z=\cdots(a,i-1)(i+1,i+1)\cdots(i,c)\cdots\] where $a<i-1<b=i<i+1<c$.
In this case, let $T$ denote the partial tableau for $y$ obtained after inserting $(a,i-1)$.
Then let $T_y$ be the partial tableau obtained after inserting $(i,i)$, 
let $T'_y$ be the partial tableau obtained before inserting $(i+1,c)$, 
and let $T''_y$  be the partial tableau obtained after inserting $(i+1,c)$. 
Define $T_z$, $T'_z$, and $T''_z$ relative to $z$ similarly.
By an argument similar to case 1a(ii), we deduce that
\[i\prec_{T''_y} i-1\prec_{T''_y} i+1
\quand
T''_z=s_i(T''_y)=D_i(T''_y).\] 
Then by Lemmas~\ref{n1-lem} and \ref{n-lem2} we have
 \[i\prec_{\PCBS(y)} i-1\prec_{\PCBS(y)} i+1\quand \PCBS(z)=D_i(\PCBS(y)).\]

\item[iii.] We could have $y=\cdots(i-1,b)\cdots(i,i)\cdots(i+1,c)\cdots$ and $z=\cdots(i-1,b)\cdots(i+1,i+1)\cdots(i,c)\cdots$ where $a=i<i+1<b<c$.
In this case, the proof is the same as case ii after replacing $(a,i-1)$ by $(i-1,b)$.
\end{itemize}

\item[(3d)] Assume $y(i)=a<y(i-1)=b<y(i+1)=c$ and $\{i-1,i,i+1\}\cap\{a,b,c\}=\varnothing$.
Then we are in one of the four possible subcases indicated in Figure~\ref{cases34-fig}:
\begin{itemize}
\item[i.] We could have $  y=\cdots(b,i-1)(a,i)(c,i+1)\cdots$ and $  z=\cdots(b,i-1)(c,i)(a,i+1)\cdots$ where $a<b<c<i-1$,
so that the arc diagrams of $y$ and $z$ are
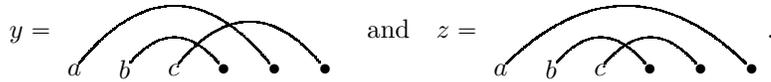
\[
  y=\arcstart
{
*{a}   \arc{2}{rrrr}  & *{b} \arc{1}{rr} & *{c}  \arc{1.5}{rrr}     & *{\bullet} & *{\bullet} & *{\bullet}   
}
\arcstop
\quand   z=\arcstart
{
*{a}   \arc{2}{rrrrr}  & *{b} \arc{1}{rr} & *{c}  \arc{1}{rr}     & *{\bullet} & *{\bullet} & *{\bullet}     
}
\arcstop.
\]
We again consider two sequences of successive row insertions.
Denote the partial tableau 
for $y$ obtained 
just before inserting $(b,i-1)$ by $T$. This is also the partial tableau for $z$ obtained just before inserting $(b,i-1)$
Then let $T_y$, $T'_y$, and $T''_y$ be the partial tableaux for $y$ obtained
just after inserting $(b,i-1)$, $(a,i)$, and $(c,i+1)$.
Define $T_z$, $T'_z$, and $T''_z$ relative to $z$ similarly. Note that $T_z=T_y$.
Then write 
\[
\ba \b^y_b(k)&:=\b_{T\from b}(k),
\\
\b^y_a(k)&:=\b_{T_y\from a}(k),
\\
\b^y_c(k)&:=\b_{T'_y\from c}(k).
\ea
\] Define $\b^z_a(k)$, $\b^z_b(k)$ and $\b^z_c(k)$ similarly. Also, write 
\[ 
\f_y(b):=\f_{T\from b}, 
\f_y(a):=\f_{T_y\from a}, 
\quand
\f_y(c):=\f_{T'_y\from c}.\] For $z$, we define $\f_z(a),\f_z(b),\f_z(c)$ similarly.
When \[j\le m:=\min\{\f_y(a),\f_y(b),\f_y(c),\f_z(a),\f_z(b),\f_z(c)\},\] 
 Lemma~\ref{n2-lem} implies that we have $\b_{a}(j+1)< \b_{b}(j+1)< \b_{c}(j+1)$. 
 Since 
\begin{align*}
\b^y_b(j)\b^y_a(j)\b^y_c(j)&\K{}\b^y_b(j)\b^y_c(j)\b^y_a(j)\\
&=\b^z_b(j)\b^z_c(j)\b^z_a(j), 
\end{align*}
the $j$th rows of $T''_y$ and $T''_z$ are the same.
Without loss of generality, we can assume $m=1$. By Lemma~\ref{n2-lem}, we have 
\[ \f_y(a)>\f_y(b)=\f_z(b)\quand \f_z(c)<\f_z(a).\] 
Now we have two cases:
\ben
\item[(I)]  $1=\f_y(c)<\f_y(b)<\f_y(a)$ and $\f_z(c)=1$, or  
\item[(II)] $\f_y(b)=1<\f_y(c)$.
\een
In case (I), consider the first rows of $T''_y$ and $T''_z$. These rows differ only in the entries of their last box, which are $i+1$ for $T''_y$ and $i$ for $T''_z$. For the other rows, according to Lemma~\ref{n2-lem}, we have $\BPath_{T_y\leftarrow a}\le \BPath_{T\leftarrow b}$, so we must have \[i\prec_{T''_y} i-1\prec_{T''_y} i+1\quand T''_z=s_i(T''_y)=D_i(T''_y),\] so the desired result  follows from Lemmas~\ref{n1-lem} and \ref{n-lem2}.
Now suppose we are instead in case (II).
We consider the insertions that construct $ T'_y$, $ T''_y$, $ T'_z$, and $T''_z$. The first rows of $T_y\leftarrow a$, $ T'_y\leftarrow c$, and $T''_y$ are shown below; here $a'=\b^y_{a}(2)$ and we use parentheses to indicate entries that are only present when $a'<b$:
\begin{align*}
\ytabb{\cdots&(a')&(\cdots)&b&i-1}\leftarrow  a
&\Rightarrow\ytabb{\cdots& a&(\cdots)&(b)&i-1}\leftarrow  c\\
&\Rightarrow\ytabb{\cdots& a&(\cdots)&(b)& c}.
\end{align*}
In this situation we have $\b^y_{a}(2)=a'\le b$ and $\b^y_{c}(2)=i-1$.
Similarly, the first rows of $T_z\leftarrow c$, $ T'_z\leftarrow a$, and $ T''_z$ appear as
\begin{align*}
\ytabb{\cdots&(a')&(\cdots)&b&i-1}\leftarrow  c
&\Rightarrow\ytabb{\cdots&(a')&(\cdots)&b& c}\leftarrow  a\\
&\Rightarrow\ytabb{\cdots& a&(\cdots)&(b)& c},
\end{align*}
so $\b^z_{c}(2)=i-1$ and $\b^z_{a}(2)=a'\le b$.
Since $\f_y(b)=1$, the length of the second row of $T_y=T_z$ must be at least 2 less than the length of the first row. Therefore we have two further subcases.
\begin{itemize}
\item Assume $a'$ is greater than every entry in row $2$ of $T_y$. Then $\f_y(a)=2$, and 
the second rows of $T'_y\leftarrow c$ and $T''_y$ appear as \[
\ytabb{\cdots&a'&i}\leftarrow i-1
\Rightarrow\ytabb{\cdots&a'&i-1}
\]
so $\b^y_{c}(3)=i$.
Similarly, we have $\f_z(c)=2$ 
and the second rows of $T'_z\leftarrow a$ and $T''_z$ appear as
\[
\ytabb{\cdots&i-1&i}\leftarrow a'
\Rightarrow\ytabb{\cdots&a'&i}
\] 
so $\b^z_{c}(3)=i-1$.
As a result, we have $\f_y(c)=3$ and the last two rows of $T''_y$ and $T''_z$ appear as
\[
\ytabb{\cdots&\cdots&a'&i-1\\\cdots&i&i+1}
\quand
\ytabb{\cdots&\cdots&a'&i\\\cdots&i-1&i+1}.
\]
Therefore $T''_z=D_i(T''_y)$ and the desired result follows from Lemmas~\ref{n1-lem} and \ref{n-lem2}.

\item Alternatively suppose that $a'$ is not greater than every entry in row $2$ of $T_y$. Then we have $\f_y(a)>\f_y(c)=2$ and the second rows of $T'_y\leftarrow c$ and $ T''_y$ are 
\[
\ytabb{\cdots&a'&\cdots}\leftarrow i-1
\Rightarrow\ytabb{\cdots&a'&\cdots&i-1&i+1}
\]
so $\b^y_{a}(3)=a''$.
Similarly, we have $\f_z(c)=2$ 
and the second rows of $T'_z\leftarrow c$ and $ T''_z$ appear as
\[
\ytabb{\cdots&i-1&i}\leftarrow a'
\Rightarrow\ytabb{\cdots&a'&\cdots&i-1&i}
\] 
so $\b^z_{a}(3)=a''$.
As a result, we have $i\prec_{T''_y}i-1 \prec_{T''_y}i+1$ and $T''_z=D_i(T''_y)$.
Therefore, the desired result again follows from Lemmas~\ref{n1-lem} and \ref{n-lem2}.
\end{itemize}

\item[ii.] We could have $  y=\cdots(b,i-1)(a,i)\cdots(i+1,c)\cdots$ and $  z=\cdots(b,i-1)(a,i+1)\cdots(i,c)\cdots$ where $a<b<i-1<i+1<c$,
so that $y$ and $z$ have arc diagrams
\[
  y=\arcstart
{
*{a}   \arc{1.5}{rrr}  & *{b}\arc{0.5}{r} & *{\bullet}   & *{\bullet}& *{\bullet}\arc{0.5}{r}   & *{c}   
}
\arcstop
\quand   z=\arcstart
{
*{a}   \arc{2}{rrrr}  & *{b}\arc{0.5}{r} & *{\bullet}   & *{\bullet} \arc{1}{rr} & *{\bullet} & *{c}   
}
\arcstop.
\]
In this case, denote the partial tableaux for $y$ 
obtained just after inserting $(b,i-1)$, 
just after inserting $(a,i)$,
just before inserting $(i+1,c)$,
and just after inserting $(i+1,c)$ 
by $T$, $T_y$, $T'_y$, and $T''_y$, respectively.
Define $T_z$, $T'_z$, and $T''_z$ relative to $z$ similarly.
Then   $T_z$ is just $T_y$ after replacing $i$ by $i+1$. We abbreviate by writing 
\[ \f_y(b):=\f_{T\from b}\quand  \f_y(a):=\f_{T_y\from a}.\]
Since $a<b$, we must have $\f_y(b)<\f_y(a)$ by Lemma~\ref{n2-lem} and then $i\prec_{T'_y} i-1$.
After the insertions $T'_y\fromCBS(i+1,c)$ and $T'_z\fromCBS(i,c)$, 
we get 
\[i\prec_{T''_y} i-1\prec_{T''_y} i+1 \quand T''_z=s_i(T''_y)=D_i(T''_y)\] so the desired result follows from Lemmas~\ref{n1-lem} and \ref{n-lem2}.

\item[iii.] We could have 
$  y=\cdots(a,i)\cdots(i-1,b)\cdots(i+1,c)\cdots$ and $  z=\cdots(a,i+1)\cdots(i-1,b)\cdots(i,c)\cdots$ where $a<i-1<i+1<b<c$,
so that 
\[
  y=\arcstart
{
*{a}   \arc{1}{rr}  & *{\bullet}\arc{1.5}{rrr} & *{\bullet}  & *{\bullet}\arc{1}{rr}       & *{b} & *{c}   
}
\arcstop
\quand   z=\arcstart
{
*{a}   \arc{1.5}{rrr}  & *{\bullet}\arc{1.5}{rrr} & *{\bullet} \arc{1.5}{rrr}      & *{\bullet}  & *{b} & *{c}     
}
\arcstop.
\]
Let $T_y$ denote the partial tableau for $y$ obtained after inserting $(a,i)$, 
let $T'_y$ 
be the partial tableau for $y$ obtained after inserting $(i-1,b)$, 
let $T''_y$ be the partial tableau for $y$ obtained before inserting $(i+1,c)$, and
let $T'''_y$ be the partial tableau for $y$ obtained after inserting $(i+1,c)$. 
Define $T_z$, $T'_z$, $T''_z$, and $T'''_z$ relative to $z$ similarly.
Then $T'_z$ is just $T'_y$ after replacing $i$ by $i+1$,
and we have $i\prec_{T'_y} i-1$. Thus, by Lemma~\ref{n1-lem}, we have $i\prec_{T''_y} i-1$. 
Now consider the insertion  $T''_y\fromCBS(i+1,c)$ and $T''_z\fromCBS(i,c)$. Then $i+1$ is in the first row of $T'''_y$ 
while $i$ is in the first row of $T'''_z$, so  \[i\prec_{T'''_y} i-1_{T'''_y}\prec i+1\quand T'''_z=s_i(T'''_y)=D_i(T'''_y).\] 
Thus the desired result follows from Lemmas~\ref{n1-lem} and \ref{n-lem2}.

\item[iv.] Finally, we could have $  y=\cdots(i,a)\cdots(i-1,b)\cdots(i+1,c)\cdots$ and $  z=\cdots(i+1,a)\cdots(i-1,b)\cdots(i,c)\cdots$ where $i+1<a<b<c$,
so that 
\[
  y=\arcstart
{
*{\bullet}   \arc{2}{rrrr}  & *{\bullet}\arc{1}{rr}  & *{\bullet}   \arc{1.5}{rrr}    & *{a} & *{b} & *{c}   
}
\arcstop
\quand  z=\arcstart
{
*{\bullet}   \arc{2}{rrrr}  & *{\bullet} \arc{2}{rrrr} & *{\bullet}  \arc{0.5}{r}     & *{a} & *{b} & *{c}   
}
\arcstop.
\]
In this case, the proof is similar to case iii after replacing $(a,i)$ by $(i,a)$.
\end{itemize}

\item[(3e)] Assume $y(i-1)=i+1$ and $y(i)=a<i-1$. In this case, $y=\cdots(a,i)(i-1,i+1)\cdots$ and $  z=\cdots(i-1,i)(a,i+1)\cdots$
have arc diagrams 
\[
  y=\arcstart
{
 *{a}\arc{1}{rr} &  *{\bullet}\arc{1}{rr}  & *{\bullet} & *{\bullet}
}
\arcstop
\quand   z=\arcstart
{
 *{a}\arc{1.5}{rrr} &  *{\bullet}\arc{0.5}{r}  & *{\bullet} & *{\bullet}
}
\arcstop.
\]
Let $T$ denote the partial tableau for $y$ obtained just before inserting $(a,i)$.
Then let $T_y$ and $T'_y$  be the partial tableaux  for $y$ obtained after inserting $(a,i)$ and $(i-1,i+1)$, respectively.
Likewise define $T_z$ and $T'_z$ to be the partial tableaux  for $z$ obtained after inserting $(i-1,i)$ and $(a,i+1)$, respectively.
Then write 
\[ \b^y_a(k):= \b_{T\from a}(k)
\quand \b^y_{i-1}(k):=\b_{T_y\from i-1}(k).\] 
Also, write 
\[ \f_y(a):=\f_{T\from a}\quand
\f_y(i-1):=\f_{T_y\from i-1}.\] 
There are now two cases, as we either have (I) $\f_y(a)=1$ or (II) $\f_y(a)=j>1$.
In case (I), since $\f_y(a)=1$, the first two rows of $T'_y$ and $T'_z$ appear as 
\[
T'_y=\ytabb{\cdots&\cdots&a&i-1\\\cdots&i&i+1}
\quand
T'_z=\ytabb{\cdots&\cdots&a&i\\\cdots&i-1&i+1}.
\]
Thus, we have $i\prec_{T'_y}i+1\prec_{T'_y}i-1$ and $T'_z=s_i(T'_y)$ and the desired result follows from Lemma~\ref{n-lem2}. 
Suppose instead that we are in case (II).
 If $i$ is in the first row of $T_y$, then $a$ and $i$ are the last two elements in the first row of $T_y$.
 Thus $T_y$ and $T_z$ appear as below, where we write $f=\f_y(a)$ and $a_j=\b^y_a(j)$ for $j<f$
(so $a=a_1$):
\[
T_y=\ytabb{\cdots&a_1&i\\\cdots&a_2\\\cdots&\cdots\\\cdots&a_f}
\quand
T_z=\ytabb{\cdots&a_2&i-1&i\\\cdots&\cdots\\\cdots&a_f\\\cdots}.
\]
After inserting $(i-1,i+1)$ into $T_y$ and $(a,i+1)=(a_1,i+1)$ into $T_z$ we get 
\[
T'_y=\ytabb{\cdots&a_1&i-1&i+1\\\cdots&a_2&i\\\cdots&\cdots\\\cdots&a_f}
\quand
T'_z=\ytabb{\cdots&a_1&i-1&i\\\cdots&a_2&i+1\\\cdots&\cdots\\\cdots&a_f}
\]
so we have \[i\prec_{T'_y}i-1\prec_{T'_y}i+1 \quand T'_z=s_{i-1}(T'_y)=D_i(T'_y).\]
If $i$ is not in the first row of $T_y$,
then $\f_y(a)>\f_y(i-1)$, so $\BPath_{T\leftarrow a}\cap\BPath_{T_y\leftarrow i-1}=\varnothing$. 
Thus, interchanging the order in which we insert the cycles involving $a$ and $i-1$ does not 
change anything but the position of $i$ and $i+1$. 
Consequently, we have 
\[ i\prec_{T'_y}i-1\prec_{T'_y}i+1\quand  T'_z=s_i(T'_y)=D_i(T'_y).\]
Either way, the desired result follows from Lemmas~\ref{n1-lem} and \ref{n-lem2}.
\end{itemize}
In all of these cases we deduce that $\PCBS(y) = D_i(\PCBS(z))$ as needed.
\end{proof}

Propositions~\ref{prop1}, \ref{prop2}, and \ref{prop3}
address the three cases in Theorem~\ref{n-cb-thm}. By combining these results, the theorem follows.

\section*{Declarations}

\subsection*{Funding}

This work was partially supported by grant GRF 16306120 from the Hong Kong Research Grants Council and by grant 2023M741827 from the China Postdoctoral Science Foundation. 
The authors have no other relevant financial or non-financial interests to disclose.
 
 \subsection*{Acknowledgments}
 
We are very grateful to Joel Lewis for helpful discussions which benefited this work, and for contributing the proof of Proposition~\ref{only-fixed-prop}.

\end{document}